\documentclass[a4paper, 11pt]{amsart}
\pdfoutput=1
\title{Homotopy coherent companionships and conjunctions}
\author{Jaco Ruit}
\address{Max-Planck-Institut f\"ur Mathematik, Vivatsgasse 7, Bonn, Germany}
\email{ruit@mpim-bonn.mpg.de} 

\usepackage{amssymb}
\usepackage{amsfonts}
\usepackage{amsthm}

\usepackage{kpfonts} 
\usepackage[mathscr]{eucal}

\usepackage[english]{babel}

\usepackage{tikz-cd}{}  
\usetikzlibrary{backgrounds}
\usetikzlibrary{patterns}
\usepackage{enumitem}
\usepackage[linktocpage]{hyperref}
\usepackage{aliascnt}
\usepackage{varioref}
\usepackage{xstring}
\usepackage{xcolor}
\usepackage{bbm}
\usepackage{breakurl}
\usepackage{mathdots}
\usepackage{float}

\usepackage[left=3cm,right=3cm,top=3.2cm,bottom=3.2cm]{geometry} 

\newtheorem{thmintro}{Theorem}

\newcommand{\inewtheorem}[2]{
	\newaliascnt{#1}{thmintro}
	\newtheorem{#1}[#1]{#2}
	\aliascntresetthe{#1}
}
\inewtheorem{propintro}{Proposition}
\inewtheorem{corintro}{Corollary}

\newtheorem{theorem}{Theorem}
\newcommand{\jnewtheorem}[2]{
	\newaliascnt{#1}{theorem}
	\newtheorem{#1}[#1]{#2}
	\aliascntresetthe{#1}
}
\numberwithin{theorem}{section}
\jnewtheorem{lemma}{Lemma}
\jnewtheorem{proposition}{Proposition}
\jnewtheorem{corollary}{Corollary}
\jnewtheorem{conjecture}{Conjecture}

\theoremstyle{definition}
\jnewtheorem{definition}{Definition}
\jnewtheorem{example}{Example}
\jnewtheorem{remark}{Remark}
\jnewtheorem{question}{Question}
\jnewtheorem{fact}{Fact}
\jnewtheorem{construction}{Construction}
\jnewtheorem{notation}{Notation}
\jnewtheorem{convention}{Convention}

\labelformat{theorem}{Theorem #1}
\labelformat{lemma}{Lemma #1}
\labelformat{proposition}{Proposition #1}
\labelformat{corollary}{Corollary #1}
\labelformat{definition}{Definition #1}
\labelformat{example}{Example #1}
\labelformat{section}{Section #1}
\labelformat{subsection}{Subsection #1}
\labelformat{equation}{(#1)}
\labelformat{remark}{Remark #1}
\labelformat{exercise}{Exercise #1}
\labelformat{question}{Question #1}
\labelformat{chapter}{Chapter #1}
\labelformat{construction}{Construction #1}
\labelformat{conjecture}{Conjecture #1}
\labelformat{notation}{Notation #1}
\labelformat{convention}{Convention #1}
\labelformat{thmintro}{Theorem #1}
\labelformat{corintro}{Corollary #1}
\labelformat{propintro}{Corollary #1}

\numberwithin{table}{subsection}
\labelformat{table}{Table #1}

\hypersetup{
	pdftitle={On functor double infinity categories},
	pdfauthor={Jaco Ruit},
	bookmarksnumbered=true,     
	bookmarksopen=true,         
	bookmarksopenlevel=1,       
	colorlinks=true,   
	linkcolor = {cyan!60!black},
	urlcolor = {cyan!60!black},
	citecolor = {cyan!60!black},   
	pdfencoding=unicode,      
	pdfstartview=Fit,           
	pdfpagemode=UseOutlines,    
	pdfpagelayout=OneColumn,
	breaklinks
}

\newcommand{\op}{\mathrm{op}}
\newcommand{\map}{\mathrm{Map}}
\def\colim{\qopname\relax m{colim}}
\newcommand{\fun}{\mathrm{Fun}}
\newcommand{\FUN}{\mathrm{FUN}}
\newcommand{\Cat}{\mathrm{Cat}}
\newcommand{\CAT}{\mathrm{CAT}}
\newcommand{\lax}{\mathrm{lax}}

\renewcommand{\S}{\mathscr{S}}
\newcommand{\C}{\mathscr{C}}
\renewcommand{\D}{\mathscr{D}}
\renewcommand{\P}{\mathscr{P}}
\newcommand{\Q}{\mathscr{Q}}
\newcommand{\E}{\mathscr{E}}
\newcommand{\X}{\mathscr{X}}
\newcommand{\Y}{\mathscr{Y}}

\newcommand{\N}{\mathbb{N}}

\renewcommand{\Vert}{\mathrm{Vert}}
\newcommand{\Hor}{\mathrm{Hor}}
\newcommand{\Seg}{\mathrm{Seg}}
\newcommand{\Hom}{\mathrm{Hom}}
\newcommand{\tp}{\mathrm{t}}

\newcommand{\DblCat}{\mathrm{DblCat}}
\newcommand{\PSh}{\mathrm{PSh}}
\newcommand{\Sq}{\mathrm{Sq}}	
\newcommand{\id}{\mathrm{id}}
\newcommand{\vop}{\mathrm{vop}}
\newcommand{\CCAT}{\mathbb{C}\mathrm{at}}
\newcommand{\SSPAN}{\mathbb{S}\mathrm{pan}}
\newcommand{\FFUN}{\mathbb{F}\mathrm{un}}
\newcommand{\hop}{\mathrm{hop}}
\newcommand{\vrectangle}{{\ooalign{\lower.3ex\hbox{$\sqcup$}\cr\raise.4ex\hbox{$\sqcap$}}}}
\newcommand{\Sp}{\mathrm{Sp}}
\newcommand{\comp}{\mathfrak{comp}}
\newcommand{\conj}{\mathfrak{conj}}
\newcommand{\adj}{\mathfrak{adj}}

\begin{document}

\begin{abstract}
We demonstrate that companionships and conjunctions in double $\infty$-categories --- and more generally, in double Segal spaces ---
extend to functors out of the free-living companionship and conjunction respectively.  Specifically, we prove that these 
extensions are (homotopically) unique: the corresponding spaces of extensions are contractible under suitable completeness assumptions. 
The developed theory is then put to use to give a characterization 
of companions and conjoints in functor double Segal spaces in terms of so-called companionable and conjointable 2-cells. 
We end with an application of our results to $(\infty,2)$-category theory.
\end{abstract}

\maketitle

\setcounter{tocdepth}{1}
\tableofcontents
	
\section{Introduction}

Double $\infty$-categories are two-dimensional $\infty$-categorical structures with objects, two directions of 1-cells: a horizontal and vertical direction, and a notion of 2-cells.
They are the weak $\infty$-categorical generalizations of Ehresmann's strict double categories \cite{Ehresmann}, and were first introduced in the Ph.D.\ thesis of Haugseng \cite{HaugsengPhD}. 
It is shown by Moser \cite{Moser} that the strict double categories embed fully faithfully into double $\infty$-categories.
In a precise sense, double $\infty$-categories can be viewed as a generalization of $(\infty,2)$-categories, 
which admit now two distinct directions for arrows. There are many examples of double $\infty$-categories. For instance, 
in \cite{HaugsengSpans}, Haugseng constructs 
the double $\infty$-categories of spans. In the realm of higher algebra, there are double $\infty$-categories of algebras, maps of algebras and bimodules \cite[Section 4.4]{HA}. 
In \cite{EquipI}, we discuss the construction of the double $\infty$-category 
of $\infty$-categories, functors and profunctors, as well as an internal variant. More generally, Haugseng constructs  a double $\infty$-category 
of suitably enriched $\infty$-categories in \cite{HaugsengEnriched}. 

\subsection{Double Segal spaces and the cartesian closed structure}

Throughout this article, we will view double $\infty$-categories as special kinds of \textit{double Segal spaces}. 
From this point of view, double $\infty$-categories are double Segal spaces with an extra \textit{completeness} or \textit{univalence} condition.  
A self-contained and detailed introduction to double Segal spaces, alongside with the different variations 
on completeness, is found in \ref{section.dss}. 

In particular, we will demonstrate here that the ambient $\infty$-category $$\Cat^2(\S)$$
of double Segal spaces is cartesian closed. To do so, we will use the language of \textit{exponential ideals}, to be recollected in \ref{section.prelims},
and verify that various localizations of double Segal spaces are  exponential ideals in the $\infty$-category $\PSh(\Delta^{\times 2})$ of presheaves on $\Delta^{\times 2}$. 
If $\P$ and $\Q$ are double Segal spaces, then the associated internal hom in $\Cat^2(\S)$ of functors between $\P$ and 
$\Q$ is denoted by 
$$
\FFUN(\P,\Q).
$$
This double Segal space of functors may 
be computed as the internal hom in $\PSh(\Delta^{\times 2})$. The vertical and horizontal arrows 
of $\FFUN(\P,\Q)$ are called \textit{vertical natural transformations} and \textit{horizontal natural transformations}.
A horizontal natural transformation $\alpha : h \rightarrow k$ between two
functors $h,k : \P \rightarrow \Q$ is compromised of the following data:
\begin{itemize}
	\item for every object $x\in \P$, a horizontal arrow in $\Q$ $$\alpha_x : h(x) \rightarrow k(x),$$
	\item for every vertical arrow $f : x \rightarrow y$ in $\P$, a naturality 2-cell in $\Q$
	\[
		\begin{tikzcd}
			h(x) \arrow[r, " \alpha_x"name=f]\arrow[d, "h(f)"'] & k(x) \arrow[d, "k(f)"] \\
			h(y) \arrow[r, "\alpha_y"name=t] & k(y),
			\arrow[from=f,to=t,Rightarrow, shorten <= 6pt]
		\end{tikzcd}
	\]
\end{itemize}
and the data of (usually, infinitely) many coherences. This is a double $\infty$-categorical variant 
on the notions of vertical and horizontal natural transformations that were studied by Grandis and Par\'e \cite{GrandisPareLimits} for strict double categories.

\subsection{Companionships and conjunctions}
The main result of this paper concerns \textit{companionships} and \textit{conjunctions} in double Segal spaces.
These are double categorical analogs of the adjunctions
for 2-categories, and were introduced by Grandis--Par\'e \cite{GrandisPare} in the strict context.
Likewise, these analogs play an important
role in the theory and applications of double categories \cite{ShulmanFramedBicats} \cite{DawsonParePronk} \cite{Vasilakopoulou}.
The notions of companions and conjoints may be readily generalized to the weak $\infty$-categorical setting.

In this higher setting, companions and conjoints were first considered  (albeit under a different name) by Gaitsgory and Rozenblyum to set up the $(\infty,2)$-categorical theory for 
their six-functor formalisms in derived algebraic geometry \cite{GR}. 
Moreover, companions and conjoints play a central role in a double $\infty$-categorical approach to \textit{formal category theory} \cite{EquipI}. In this approach, 
companionships and conjunctions witness (co)representability of \textit{abstract profunctors}, i.e.\ abstract families of presheaves.

A companionship between a vertical arrow $f : x \rightarrow y$ and a horizontal arrow $F: x \rightarrow y$ 
in a double Segal space $\P$ is witnessed by the following data:
\begin{enumerate}
	\item a \textit{companionship unit} 2-cell 
	\[
		\eta = \begin{tikzcd}
			x \arrow[r,equal, ""name=f]\arrow[d,equal] & x \arrow[d, "f"] \\
			x \arrow[r, "F"'name=t] & y,
			\arrow[from=f,to=t,Rightarrow, shorten <= 6pt, shorten >= 6pt]
		\end{tikzcd}
	\]
	\item a \textit{companionship counit} 2-cell 
	\[
		\epsilon = \begin{tikzcd}
			x \arrow[r,"F"name=f]\arrow[d,"f"'] & y \arrow[d, equal] \\
			y \arrow[r, equal, ""name=t] & y,
			\arrow[from=f,to=t,Rightarrow, shorten <= 6pt]
		\end{tikzcd}
	\]
	\item an equivalence in the space of 2-cells of $\P$ between the vertical identity 2-cell 
	\[
		\begin{tikzcd}
			x \arrow[d,"f"'name=f]\arrow[r,equal] & x\arrow[d, "f"name=t]\\
			y \arrow[r,equal] & y,
			\arrow[from=f,to=t,equal, shorten <= 10pt, shorten >= 10pt]
		\end{tikzcd}
	\]and the vertical pasting of $\eta$ and $\epsilon$,
	\item an equivalence in the space of 2-cells of $\P$ between the horizontal identity 2-cell 
	\[
		\begin{tikzcd}
			x \arrow[r,"F"name=f]\arrow[d,equal] & y \arrow[d, equal] \\
			x \arrow[r, "F"'name=t] & y,
			\arrow[from=f,to=t,equal, shorten <= 6pt, shorten <= 6pt, shorten >= 6pt]
		\end{tikzcd}
	\]and the horizontal pasting of $\eta$ and $\epsilon$.
\end{enumerate}
In this context, $F$ is called the companion of $f$. The equivalences of (3) and (4) witness the \textit{triangle identities} for the companionship. The terminology 
is chosen to reflect the similarity with adjunctions.
The definition of conjunctions is (formally) dual to the definition of companionships, and will be given in \ref{section.comp-conj}. 
We will also consider some examples of companionships and conjunctions in this section.

It is a celebrated result of Riehl and Verity \cite{RiehlVerityAdj} that particular choices of adjunction data in an $(\infty,2)$-category 
uniquely (in the homotopic sense) upgrade to a homotopy coherent adjunction, i.e.\ a functor out of the \textit{free-living adjunction} of Schanuel--Street \cite{SchanuelStreet}. 
One of the central objectives of this paper is to prove an analogous result for companionships and conjunctions.

We will show that, likewise, every companionship in $\P$ can uniquely be upgraded to a \textit{homotopy coherent} companionship. 
Analogously, a homotopy coherent companionship is another name for a functor 
from the \textit{free-living companionship} double Segal space $$\comp$$
to $\P$. The double Segal space $\comp$ is level-wise discrete and described as follows. 
The space $\comp_{n,m}$ of $n \times m$ grids of 2-cells in $\comp$ is given by the \textit{set}
$$
{\mathrm{Poset}}([n] \times [m], [1])
$$
of maps of posets $[n] \times [m] \rightarrow [1]$.
We will show the following results:

\begin{thmintro}
Any companionship in a double Segal space $\P$ extends to a homotopy coherent companionship $\comp \rightarrow \P$. Moreover, 
suppose that $\eta$ is a companionship unit. Then the space of 
functors $\comp \rightarrow \P$ extending $\eta$ is contractible.
\end{thmintro}

\begin{thmintro}
Suppose that $f$ is a vertical arrow in a locally complete double Segal space $\P$. 
If $f$ has a companion, then the space of functors $\comp \rightarrow \P$ extending $f$ is contractible.
\end{thmintro}

\noindent The precise statements appear as \ref{thm.htpy-coh-comp-1} and \ref{thm.htpy-coh-comp-2}. 
Note that the above results justify that the double Segal space $\comp$ carries the name of the free-living companionship. 
We will explain in \ref{section.comp-conj} that one can use the above results to formally obtain dual results for conjunctions. 
In this case, 
the double Segal space $\comp$ is replaced by its appropriate dual, the \textit{free-living conjunction}, denoted by
$$
\conj.
$$
We will elucidate the relation of our results to the work of Riehl--Verity \cite{RiehlVerityAdj} in \ref{rem.free-adj}.

\subsection{Companions of vertical natural transformations} 

We will use the results of \ref{section.comp-conj} in \ref{section.comp-conj-fundbl} 
to  study companions (and dually, conjoints) in functor double Segal spaces. 
The key step is the introduction of the 
notion of \textit{companionable} and \textit{conjointable} 2-cells in a double Segal space. In a precise sense, 
this is a double $\infty$-categorical variant on the $(\infty,2)$-categorical notion of \textit{adjointability} \cite[Definition 7.3.1.2]{HTT} for lax commutative squares. We will explain this in \ref{ex.conjointable-sq}. The 
main result of \ref{section.comp-conj-fundbl} is the following:

\begin{thmintro}\label{thmintro.companions-fun-dblcat}
Let $\P$ and $\Q$ be double Segal spaces, so that $\Q$ is locally complete. Suppose that $\alpha : h \rightarrow k$ is a horizontal natural transformation between functors $h,k : \P \rightarrow \Q$. 
Then the following assertions are equivalent:
\begin{enumerate}
	\item $\alpha$ is a companion in $\FFUN(\P,\Q)$,
	\item for every vertical arrow $f : x \rightarrow y$ in $\P$, the associated naturality 2-cell 
	\[
		\begin{tikzcd}
			h(x) \arrow[r, " \alpha_x"name=f]\arrow[d, "h(f)"'] & k(x) \arrow[d, "k(f)"] \\
			h(y) \arrow[r, "\alpha_y"'name=t] & k(y)
			\arrow[from=f,to=t,Rightarrow, shorten <= 6pt, shorten >= 6pt]
		\end{tikzcd}
	\]
	is companionable in $\Q$.
\end{enumerate}
\end{thmintro}

\noindent The precise statement appears as \ref{thm.companions-fun-dblcat} in the paper.

\subsection{An application to $(\infty,2)$-category theory} As evident in the 
work of Gaitsgory--Rozenblyum \cite{GR}, having good foundations of double $\infty$-categories 
can be fruitful in building the theory of $(\infty,2)$-categories. We will attest to this principle by giving the following 
application of \ref{thmintro.companions-fun-dblcat} to $(\infty,2)$-category theory.

Firstly, we will show in \ref{section.app-lax-nt} that, given $(\infty,2)$-categories $\X$ and $\Y$, one can 
construct an associated double $\infty$-category 
$$
\FFUN^\lax(\X, \Y),
$$
as a certain functor double $\infty$-category, such that:
\begin{itemize}
	\item its objects are functors $\X \rightarrow \Y$,
	\item its vertical arrows are natural transformations,
	\item its horizontal arrows are \textit{lax} natural transformations.
\end{itemize}
Here, the lax natural transformations are defined using a double categorical 
construction of the Gray tensor product that is due to Gaitsgory--Rozenblyum.
The {horizontal fragment} is given by the $(\infty,2)$-category $\FUN^\lax(\X,\Y)$ 
of functors and lax natural transformations.
We will show that \ref{thmintro.companions-fun-dblcat} can be applied to give a quick, double categorical proof of the following theorem of Haugseng \cite{HaugsengLax}:

\begin{thmintro}\label{thmintro.haugseng}
	Let $\beta : k \rightarrow h$ be a lax natural transformation between functors $k, h : \X \rightarrow \Y$. Then the 
	following assertions are equivalent:
	\begin{enumerate}
		\item $\beta$ is a right adjoint in $\FUN^\lax(\X, \Y)$,
		\item for any arrow $f : x \rightarrow y$ in $\X$, the horizontal morphisms in the lax square
		\[
			\begin{tikzcd}
				k(x) \arrow[r, "\beta_x"] \arrow[d, "k(f)"']& |[alias=f]| h(x) \arrow[d, "h(f)"] \\
				|[alias=t]| k(y) \arrow[r, "\beta_y"'] & h(y)
				\arrow[from=f,to=t, Rightarrow]
			\end{tikzcd}
		\]
		admit left adjoints and the associated mate of the square is an equivalence.	
	\end{enumerate}
\end{thmintro}

\subsection*{Conventions} We will use the language of $\infty$-categories as developed by Joyal and Lurie. Moreover, we  use the following notational conventions:
\begin{itemize}
	\item The  $\infty$-categories of spaces (i.e.\ $\infty$-groupoids) and $\infty$-categories are denoted by $\S$ and $\Cat_\infty$ respectively. 
	\item For an $\infty$-category $\C$, we write $\PSh(\C) := \fun(\C^\op, \S)$ for the $\infty$-category of pre\-sheaves on $\C$.
	\item We will use uppercase notation for $(\infty,2)$-categorical upgrades of particular $\infty$-categories. The double $\infty$-categorical variants will be decorated with blackboard bolds.
	For instance, in this article, we will see both the $(\infty,2)$-category $\CAT_\infty$ and the double $\infty$-category $\CCAT_\infty$ of $\infty$-categories.
\end{itemize}

\subsection*{Acknowledgements}
The results in this article formed a part of my Ph.D. thesis, and  I would like to express my gratitude towards my supervisor, Lennart Meier, for the helpful comments and discussions on the contents of this paper.
I thank Viktoriya Ozornova for her remarks on a preliminary version.
During a period of writing this paper, the author was funded by the Dutch Research Council (NWO) through the grant ``The interplay of orientations and symmetry'', grant no. OCENW.KLEIN.364.
The author is also grateful to the Max Planck Institute for Mathematics in
Bonn for its hospitality and financial support.

\section{Preliminaries on exponentiability}\label{section.prelims}

This first section is dedicated to reviewing and collecting both the 
definitions and basic results concerning \textit{exponentiable objects} and \textit{ideals} in $\infty$-categories, and \textit{(locally) cartesian closed} $\infty$-categories.

\begin{definition}
	Let $\C$ be an $\infty$-category with finite products. Then an object $x \in \C$ is called \textit{exponentiable} 
	if the functor $x \times (-) : \C \rightarrow \C$ admits a right adjoint. If all objects 
	in $\C$ are exponentiable, then $\C$ is called \textit{cartesian closed}.
\end{definition}

\begin{notation}
	Let $\C$ be a cartesian closed $\infty$-category. Then we will denote the internal hom functor by 
	$$
	\Hom_\C : \C^\op \times \C \rightarrow \C,
	$$
	which has the defining property that its composite with the Yoneda embedding 
	$\C \rightarrow \PSh(\C)$ is given by 
	$$\C^\op \times \C \rightarrow \PSh(\C) : (x,y)  \mapsto \map_\C((-) \times x, y).$$
\end{notation}

There is also a local notion of exponentiability. To this end, note that if $\C$ is an $\infty$-category with pullbacks then all 
its slices admit finite products and these are computed by pullbacks in $\C$.

\begin{proposition}\label{prop.exp-arrows}
	Suppose that $f : x\rightarrow y$ is an arrow in an $\infty$-category $\C$ with pullbacks. Then the following assertions are equivalent:
	\begin{enumerate}
		\item the functor $f^* : \C/y \rightarrow \C/x$ admits a right adjoint,
		\item the arrow $f$ is exponentiable when viewed as an object of $\C/y$,
		\item for every arrow $g : a \rightarrow y$, the functor $(f^*g)^* : \C/a \rightarrow \C/(a \times_y x)$ admits a right adjoint.
	\end{enumerate}
\end{proposition}
\begin{proof}
 	The product functor $f \times (-) : \C/y \rightarrow \C/y$ decomposes as a composite 
	\begin{equation*}
	\C/y \xrightarrow{f^*} \C/x \xrightarrow{f_!} \C/y,
	\end{equation*}
	and this proves that (1) implies (2). It is clear that (3) implies (1).
	Hence, it remains to show that (2) implies (3). Suppose that  we have an adjunction 
	$$
	f \times (-) : \C/y \rightleftarrows \C/y : [f,-]
	$$
	with unit $\eta$. Let $g : a\rightarrow y$ be an arrow.
	Passing to slices over $g$ gives rise to an adjunction
	$$
	\C/a \simeq (\C/y)/g \rightleftarrows (\C/y)/(f\times g) \simeq \C/(a \times_y x)
	$$
	where the right adjoint is given by the composite 
	$$
	(\C/y)/(f\times g)\xrightarrow{[f,-]} (\C/y)/[f,f\times g] \xrightarrow{\eta_{g}^*} (\C/y)/g,
	$$
	on account of \cite[Proposition 5.2.5.1]{HTT}. The left adjoint is precisely given by base change along 
	$f^*g$.
\end{proof}

\begin{definition}\label{def.exp-arrow}
Let $\C$ be an $\infty$-category with pullbacks. Then an arrow $f : x\rightarrow y$ in $\C$
is called \textit{exponentiable} if the equivalent conditions of \ref{prop.exp-arrows} are met. 
If every arrow in $\C$ is exponentiable, then $\C$ is called \textit{locally cartesian closed}.
\end{definition}

\begin{example}
	The exponentiable arrows in $\Cat_\infty$ are also called \textit{Conduch\'e fibrations}. They were studied 
	by Ayala and Francis in \cite{AyalaFrancis} and by Lurie in \cite[Section B.3]{HA}.
\end{example}

\begin{example}
	Since colimits in $\infty$-toposes are universal, all $\infty$-toposes are locally cartesian closed.
\end{example}

We highlight the following implications of \ref{prop.exp-arrows}:

\begin{corollary}
	The class of exponentiable arrows in an $\infty$-category with pullbacks 
	is closed under pullbacks.
\end{corollary}

\begin{corollary}
	An $\infty$-category $\C$ with pullbacks is locally cartesian closed if and only if the 
	slice $\C/x$ is cartesian closed for every $x\in \C$.
\end{corollary}

It will be helpful later in this article, to recognize subcategories of cartesian closed $\infty$-categories 
that are again cartesian closed. To this end, we will use the formalism of \textit{exponential ideals} (cf.\ \cite[Section 1.5]{Sketches}).

\begin{definition}
	Let $\C$ be a cartesian closed $\infty$-category.
	A fully faithful functor $i : \mathscr{I} \rightarrow \C$  is called an \textit{exponential ideal}
	if for every $x \in \C$ and $y\in \mathscr{I}$, the internal hom $\Hom_\C(x,iy)$ is contained in $ \mathscr{I}$.
\end{definition}

\begin{proposition}
	Let $\C$ be a cartesian closed $\infty$-category.
	Suppose that $i : \mathscr{I} \rightarrow \C$ is an exponential ideal so that its essential image is closed under finite products. Then $\mathscr{I}$ is again cartesian closed, and $i$ 
	preserves internal homs.
\end{proposition} 
\begin{proof}
	Since $i$ is fully faithful, it reflects limits.  Hence $\mathscr{I}$ is closed under finite products and $i$ preserves them. It follows that for objects $t, x, y$ in $\mathscr{I}$, 
	we have natural equivalences
	$$
	\map_\mathscr{I}(t \times x, y) \simeq \map_\C(it \times ix, iy) \simeq \map_\C(it, \Hom_\C(ix, iy)).
	$$
	Since $\Hom_\C(ix, iy)$ is in the image of $i$, the result follows.
\end{proof}

\begin{remark}
	Note that the full inclusion of a reflective subcategory is always closed under finite products.
\end{remark}

The following observation is immediate, and we omit its proof:

\begin{lemma}\label{lem.exp-ideals}
	Suppose that $i : \mathscr{I} \rightarrow \C$ is a reflective subcategory of a cartesian closed $\infty$-category $\C$, so that $i$ 
	admits a left adjoint $L : \C \rightarrow \mathscr{I}$. 
	Let $x \in \C$ be an object. Then the following assertions are equivalent:
	\begin{enumerate}
		\item the functor 
		$$
		x \times (-) : \C \rightarrow \C 
		$$
		preserves $L$-local equivalences,
		\item if $S$ is a class of generating $L$-local equivalences, then $x \times f$ is a $L$-local equivalence for every $f \in S$,
		\item the functor $\Hom_\C(x,-)$ carries objects in $\mathscr{I}$ to objects in $\mathscr{I}$.
	\end{enumerate}
\end{lemma}

\begin{proposition}\label{prop.exp-ideals-reflector}
	A reflective subcategory $i : \mathscr{I} \rightarrow \C$ is an exponential ideal if and only if the reflector 
	$L : \C \rightarrow \mathscr{I}$ preserves finite products.
\end{proposition}
\begin{proof}
	It readily follows from characterization (1) of \ref{lem.exp-ideals} that $i$ is an exponential ideal if $L$ preserves finite products. 
	Conversely, suppose that $i$ is an exponential ideal. 
	Then we have to show that for any two objects $x,y \in \C$, the canonical comparison map 
	$L(x\times y) \rightarrow Lx \times Ly$ is an equivalence. Let us denote the unit of the adjunction of $(L,i)$ by $\eta$. 
	Then we obtain a commutative diagram 
	\[
		\begin{tikzcd}
			L(x \times y) \arrow[r, "L(x \times \eta_y)"]\arrow[d] & L(x \times iLy) \arrow[d] \arrow[r,"L(\eta_x \times iLy)"] & L(iLx \times iLy)\arrow[d,"\simeq"]\\
			Lx \times Ly \arrow[r, "Lx \times L\eta_y"', "\simeq"] & Lx \times LiLy \arrow[r, "L\eta_x \times LiLy"', "\simeq"] & LiLx \times LiLy.
		\end{tikzcd}
	\]
	The rightmost vertical arrow is an equivalence since $L$ is a reflector and $i$ preserves finite products. By characterization (1) of \ref{lem.exp-ideals}, the two top horizontal arrows are equivalences as well.
	Hence, the desired result follows from 2-out-of-3.
\end{proof}

\section{Preliminaries on double Segal spaces}\label{section.dss}
In this section, we will review and study a selection of types of Barwick's \textit{double} and \textit{2-fold Segal spaces}
with different completeness conditions. We will then proceed to show that 
their ambient $\infty$-categories are cartesian closed by using the theory of exponential ideals 
that was discussed in \ref{section.prelims}.

\subsection{Double Segal spaces}\label{ssection.dss} 
Double Segal spaces arise as certain presheaves on $\Delta^{\times 2}$. Throughout, 
 we will write $$[n,m] \in \PSh(\Delta^{\times 2})$$ for the image of $([n],[m])$ under the Yoneda embedding $\Delta^{\times 2} \rightarrow \PSh(\Delta^{\times 2})$.
\begin{definition}
	The presheaves $[0,0]$, $[1,0]$, $[0,1]$ and $[1,1]$ are called the \textit{free-living object} (or \textit{0-cell}), the \textit{free-living horizontal} and \textit{vertical arrow} (or \textit{1-cell}), and the \textit{free-living 2-cell}, 
	respectively.
\end{definition}

\begin{definition}\label{def.dss}
	A double Segal space is a presheaf on $\PSh(\Delta^{\times 2})$ that is local with respect to the class of  following maps:
	\begin{enumerate}
		\item[(\textit{Seg})] the spine inclusions $$[1,m] \cup_{[0,m]} [1,m] \dotsb \cup_{[0,m]} [1,m]  \rightarrow [n,m],$$ and 
	  $$[n,1] \cup_{[n,0]} [n,1] \dotsb \cup_{[n,0]} [n,1] \rightarrow [n,m]$$ for $n,m \geq 0$.
	\end{enumerate}
	We will write 
	$$
	\Cat^2(\S) \subset \PSh(\Delta^{\times 2})
	$$
	for the reflective full subcategory of double Segal spaces.
\end{definition}

\begin{remark}
	Unraveling the definitions, we observe that a double Segal space $\P$ contains: 
	\begin{itemize} 
	\item a space $\P_{0,0}$ of objects,
	\item a space $\P_{0,1}$ of \textit{vertical arrows},
	\item a space $\P_{1,0}$ of \textit{horizontal arrows},
	\item and a space $\P_{1,1}$ of 
	\textit{2-cells}.
	\end{itemize}
	In general, the Segal condition for $\P$ 
	tells us that we may view $\P_{n,m}$ as a grid of compatible cells. 
	 The 2-cells and the compatible arrows of $\P$ that have the same direction may be composed in a mutually compatible and coherently associative fashion. A 2-cell may be pictured as a square 
	\[
		\begin{tikzcd}
			a \arrow[r, ""'name=f]\arrow[d] & b \arrow[d] \\
			c \arrow[r, ""name=t] & d
			\arrow[from=f, to=t, Rightarrow]
		\end{tikzcd}
	\]
	where the horizontal/vertical directed arrows are horizontal/vertical arrows of $\P$. 
\end{remark}

\begin{remark}\label{rem.segal-spaces}
	Double Segal spaces are generalizations of Rezk's Segal spaces \cite{RezkSeg}. Recall that the $\infty$-category $$\Seg(\S) \subset \PSh(\Delta)$$ of Segal spaces is
	defined to be the reflective subcategory 
	generated by the spine inclusions $[1] \cup_{[0]} [1] \cup_{[0]} \dotsb \cup_{[0]} [1] \rightarrow [n]$ for $n \geq 0$.
	Here, we implicitly view every $[n] \in \Delta$ as an object of $\PSh(\Delta)$ via the Yoneda embedding. 
\end{remark}

\begin{construction}\label{cons.dualities}
	We will use three involutive operations on $\PSh(\Delta^{\times 2})$: 
	$$
	(-)^\tp, (-)^\hop, (-)^\vop : \PSh(\Delta^{\times 2}) \rightarrow \PSh(\Delta^{\times 2}),
	$$
	called the \textit{transpose}, \textit{horizontal opposite}, and \textit{vertical opposite} respectively. They are defined by restricting along the functors 
	\begin{gather*}
	\mathrm{t} : \Delta \times \Delta \rightarrow \Delta \times \Delta : ([n],[m]) \mapsto ([m],[n]), \\ 
	\hop : \Delta \times \Delta \xrightarrow{\op \times \id} \Delta \times \Delta, \quad \vop :  \Delta \times \Delta \xrightarrow{\id \times \op} \Delta \times \Delta.
	\end{gather*}
	Here $\op : \Delta \rightarrow \Delta$ denotes the usual opposite involution, i.e.\ it is restricted 
	from the involution functor $(-)^\op$ for categories.

	One readily verifies that these involutions restrict to involutions 
	$$
	(-)^\tp, (-)^\hop, (-)^\vop : \Cat^2(\S) \rightarrow \Cat^2(\S),
	$$
	on the $\infty$-category 
	of double Segal spaces.
\end{construction}

\subsection{Two-fold Segal spaces}
As will be dicussed later in further detail, double Segal spaces may be used 
to obtain Barwick's model of $(\infty,2)$-categories. To this end, we will need 
to consider the double Segal spaces with only degenerate vertical arrows.

\begin{definition}
	A double Segal space $X$ is called a \textit{2-fold Segal space} if the structure map $$(\id, s_0)^* : X_{0, 0} \rightarrow X_{0,1}$$is an equivalence. 
	In other words, a presheaf $X \in \PSh(\Delta^{\times 2})$ is a 2-fold Segal space 
	if it is local with respect to (\textit{Seg}) of \ref{def.dss}, and additionally to: \begin{enumerate}
		\item[(\textit{deg})] the degeneracy map $
		[0,1] \rightarrow [0,0].
		$
	\end{enumerate}
	We will write $$\Seg^2(\S) \subset \PSh(\Delta^{\times 2})$$
for the reflective subcategory of $\PSh(\Delta^{\times 2})$ spanned by those 2-fold Segal spaces.
\end{definition}
	
Note that every 2-fold Segal space $X \in \PSh(\Delta^{\times 2})$ has the feature that the restriction $X_{0,\bullet}$ is an essentially constant simplicial space. 
It will be convenient to consider the subcategory of such presheaves.
	
\begin{definition}
		We will write 
		$$(-)_h : \PSh(\Delta^{\times 2})_{\mathrm{deg}} \rightarrow \PSh(\Delta^{\times 2})
		$$
		for the inclusion of the full subcategory of those presheaves $X \in \PSh(\Delta^{\times 2})$ 
		so that $X_{0, \bullet}$ is essentially constant. This is called the 
		\textit{horizontal inclusion} functor. Note that this is a inclusion of a reflective 
		subcategory, with generating local maps given by the degeneracy maps $[0,n] \rightarrow [0,0]$ for all $n$.
\end{definition}

\begin{construction}
		There is also a \textit{vertical inclusion} functor 
		$$
		(-)_v : \PSh(\Delta^{\times 2})_{\mathrm{deg}} \rightarrow \PSh(\Delta^{\times 2})
		$$
		given by the composite 
		$$
		\PSh(\Delta^{\times 2})_{\mathrm{deg}}  \xrightarrow{(-)_h} \PSh(\Delta^{\times 2}) \xrightarrow{(-)^\tp} \PSh(\Delta^{\times 2}) \xrightarrow{(-)^\hop} \PSh(\Delta^{\times 2}).
		$$
		Note that this restricts to a functor $\Seg^2(\S) \rightarrow \Cat^2(\S)$.
\end{construction}

\begin{construction}\label{con.gaunt-nerve}
	Let $\mathrm{Gaunt}_2$ be the full subcategory of the category of 2-categories that is spanned by the 
	\textit{gaunt 2-categories}. Then one can write down the bicosimplicial object 
	$$
	\Delta^{\times 2} \rightarrow \mathrm{Gaunt}_2 : ([n], [m]) \mapsto [n;m,\dotsc,m],
	$$where $[n;m\dotsc, m]$ denotes the globular 2-category that may be pictured as
	\[
			{[n;m, \dotsc, m]} := \begin{tikzcd}[column sep = large]
				0 \arrow[r, "0"name=a0, looseness = 2, bend left = 70pt ]\arrow[r, "1"'name=a1, bend left = 45pt]\arrow[r, "m-1"name=am1, bend right = 45pt]\arrow[r, "m"'name=am, looseness = 2, bend right = 70pt] & 1	
				\arrow[from=a0,to=a1, Rightarrow, shorten <= 3pt, shorten >= 3pt]
				\arrow[from=a1,to=am1, phantom, "\vdots", yshift = 3pt]
				\arrow[from=am1,to=am, Rightarrow, shorten <= 3pt, shorten >= 3pt] \arrow[r, phantom, "\cdots"] &  
				n-1\arrow[r, "0"name=b0, looseness = 2, bend left = 70pt ]\arrow[r, "1"'name=b1, bend left = 45pt]\arrow[r, "m-1"name=bm1, bend right = 45pt]\arrow[r, "m"'name=bm, looseness = 2, bend right = 70pt] & n.	
				\arrow[from=b0,to=b1, Rightarrow, shorten <= 3pt, shorten >= 3pt]
				\arrow[from=b1,to=bm1, phantom, "\vdots", yshift = 3pt]
				\arrow[from=bm1,to=bm, Rightarrow, shorten <= 3pt, shorten >= 3pt]
			\end{tikzcd}
	\]
	The restriction to $\{[0]\} \times \Delta$ is constant, so that the we obtain an induced functor 
	$$
	N : \mathrm{Gaunt}_2 \rightarrow \PSh(\Delta^{\times 2})_\mathrm{deg},
	$$
	so that the image of a gaunt 2-category $G$ is level-wise described by the \textit{set}
		$$
		NG_{n,m} = \mathrm{Hom}([n;m,\dotsc, m], G)
		$$
		of functors from $[n;m,\dotsc,m]$ to $G$.
	One readily verifies that the functor $N$ is fully faithful; we will leave the inclusion $N$ implicit throughout this chapter. 
\end{construction}

\begin{remark}\label{rem.quotient-grid}
	The obvious commutative square
	\[
				\begin{tikzcd}
					\{0, \dotsc, n\}_h \times {[m]_v} \arrow[r]\arrow[d] & {[n,m] = [n]_h \times [m]_v} \arrow[d] \\
					\{0, \dotsc, n\}_h  \arrow[r] & {[n;m,\dotsc,m]}_h,
				\end{tikzcd}
			\]
	of bisimplicial spaces is a pushout square in $\PSh(\Delta^{\times 2})$, as one readily verifies. If $X$ is a presheaf in $\PSh(\Delta^{\times 2})_\mathrm{deg}$, 
	then $X$ is local with respect to the left map in this square. Consequently, $X$ is also local to the right (quotient) map in the square. 
	It follows that the we obtain a natural equivalence
	$$
	\map_{\PSh(\Delta^{\times2})}([n;m\dotsc,m]_h, X) \rightarrow \map_{\PSh(\Delta^{\times2})}([n,m], X) = X_{n,m}.
	$$
	
	This observation implies that $\PSh(\Delta^{\times 2})_{\mathrm{deg}}$ is a presheaf $\infty$-topos. Namely, let $\Theta_2' \subset \mathrm{Gaunt}_2$ be the full subcategory spanned by the 
	2-categories $[n;m, \dotsc, m]$, $n,m \geq 0$. Then we obtain an adjunction
	$$
	\PSh(\Theta_2') \rightleftarrows \PSh(\Delta^{\times 2})_{\mathrm{deg}}
	$$ where the left adjoint is the cocontinuous extension of the restricted nerve $N|\Theta_2'$, 
	and the right adjoint is induced by restriction along the bicosimplicial object $\Delta^{\times 2} \rightarrow \Theta_2'$. 
	Using the above, one readily concludes that this sets up an adjoint equivalence.

	In turn, this can be used to recover the comparison of \cite{BergnerRezk2} between 2-fold Segal spaces and \textit{$\Theta_2$-spaces}. Namely, 
	the subcategory  $\Theta_2 \subset \mathrm{Gaunt}_2$ of \textit{Joyal disks} $[n;m_1,\dotsc, m_n]$ (see \cite{RezkTheta}) is generated by $\Theta_2'$ under retracts. 
	So the inclusion $\Theta_2' \subset \Theta_2$ induces an equivalence $\PSh(\Theta_2) \rightarrow \PSh(\Theta_2')$. A direct verification 
	shows that the 
	composed equivalence $$\PSh(\Theta_2) \rightarrow \PSh(\Theta_2') \rightarrow \PSh(\Delta^{\times 2})_{\mathrm{deg}}$$
	derives to an equivalence between $\Theta_2$-spaces and 2-fold Segal spaces.
\end{remark}

\begin{remark}
	One may picture the horizontal and vertical inclusions 
	of the image of the 2-globe $[1;1]$
	as
	\[
		[1;1]_h = \begin{tikzcd}
			0 \arrow[d,equal]\arrow[r, "0"name=f] & 1 \arrow[d,equal] \\
			0 \arrow[r, "1"name=t] & 1
			\arrow[from=f,to=t, Rightarrow, shorten <= 6pt]
		\end{tikzcd}
		\quad \text{and}\quad
		[1;1]_v = \begin{tikzcd}
			0 \arrow[r,equal, ""name=f]\arrow[d, "1"'] & 0 \arrow[d,"0"] \\
			0 \arrow[r, equal, ""name=t] & 1.
			\arrow[from=f,to=t,Rightarrow, shorten <= 6pt,]
		\end{tikzcd}
	\]
\end{remark}

\begin{construction}\label{con.1core}
		Note that there is an adjunction $\Delta \times \Delta \rightleftarrows \Delta$, where the left adjoint is given by projection onto the first coordinate, 
		and the right adjoint is given by the inclusion of the subcategory $\Delta \times \{[0]\}$. This gives rise to an adjunction 
		$\PSh(\Delta) \rightleftarrows \PSh(\Delta^{\times 2})$ between presheaf $\infty$-categories, that in turn restricts to an adjunction 
		$$
		\PSh(\Delta) \rightleftarrows \PSh(\Delta^{\times 2})_\mathrm{deg} : (-)^{(1)}.
		$$
		We will leave the fully faithful left adjoint implicit throughout this chapter; every simplicial 
		space is viewed as a presheaf in $\PSh(\Delta^{\times 2})_\mathrm{deg}$ via this functor. 
		The above adjunction also restricts to give an adjunction $$\Seg(\S) \rightleftarrows \Seg^2(\S) : (-)^{(1)}$$ as well.
\end{construction}

We will use the following special names for the vertical and horizontal 
opposites of 2-fold Segal spaces:

\begin{definition}
	If $X$ is a 2-fold Segal space, then its \textit{1-opposite} and \textit{2-opposite} 
	are defined by respectively $X^{1-\op} := X^{\hop}$ and $X^{2-\op}$ := $X^{\vop}$.
\end{definition}

Moreover, we recall the definition of the associated mapping Segal spaces of a 2-fold Segal space:

\begin{definition}
	If $x$ and $y$ are objects of a 2-fold Segal space $X$, then we will write $X(x,y)$ for the Segal space 
	of maps from $x$ to $y$ that is
	defined by the pullback square 
	\[
		\begin{tikzcd}
			X(x,y)\arrow[d] \arrow[r] & X_{1,\bullet} \arrow[d] \\
			\{(x,y)\} \arrow[r] & X_{0,\bullet}^{\times 2}.
		\end{tikzcd}
	\]
\end{definition}

\subsection{Horizontal and vertical fragments} On account of \cite[Proposition 4.12]{HaugsengSpans}, the inclusion $(-)_h : \Seg^2(\S) \rightarrow \Cat^2(\S)$ admits a right adjoint. 
We will give an alternative short proof of this here. 

\begin{construction}
We will write 
$$
\Hor(-) : \PSh(\Delta^{\times 2}) \rightarrow \PSh(\Delta^{\times 2})_\mathrm{deg}.
$$
for the functor that is induced by the bicosimplicial object
$$
\Delta^{\times 2} \rightarrow\PSh(\Delta^{\times 2}) : ([n],[m]) \mapsto [n;m,\dotsc, m]_h.
$$
\end{construction}

\begin{proposition}\label{prop.hor-right-adjoint}
	The functor $\Hor(-)$ is right adjoint to $(-)_h$. 
\end{proposition}
\begin{proof}
	Let us write  
	 $L : \PSh(\Delta^{\times 2}) \rightarrow \PSh(\Delta^{\times 2})_\mathrm{deg}$ for the left adjoint to $(-)_h$.
	Recall from \ref{rem.quotient-grid} that we have a quotient map 
	$
	[n,m] \rightarrow [n;m,\dotsc,m]_h
	$
	that is natural in $\Delta^{\times 2}$ and is carried to an equivalence by $L$.
	Let $X$ and $Y$ be presheaves on $\Delta^{\times 2}$. Then we now obtain natural equivalences 
	\begin{align*}
		\map_{\PSh(\Delta^{\times 2})}(X, \Hor(\P)) &\simeq {\lim_{[n,m] \in (\Delta^{\times 2}/X)^\op}}\map_{\PSh(\Delta^{\times 2})}([n,m], \Hor(\P)) \\
		&\simeq{\lim_{[n,m] \in (\Delta^{\times 2}/X)^\op}}\map_{\PSh(\Delta^{\times 2})}([n;m \dotsc, m]_h, \P)  \\ 
		& \simeq{\lim_{[n,m] \in (\Delta^{\times 2}/X)^\op}}\map_{\PSh(\Delta^{\times 2})}((L[n,m])_h, \P) \\
		& \simeq \map_{\PSh(\Delta^{\times 2})}((LX)_h, \P).
	\end{align*}
	In the last step, we used that $(-)_h$ preserves colimits, which follows from the easily verifiable fact that $\PSh(\Delta^{\times 2})_\mathrm{deg}$ 
	is closed under colimits in $\PSh(\Delta^{\times 2})$.
\end{proof}

\begin{remark}\label{rem.desc-hor}
Note that the functor $\Hor(-)$ carries double Segal spaces to 2-fold Segal spaces. Consequently, we obtain a restricted adjunction 
$$
(-)_h : \Seg^2(\S) \rightleftarrows \Cat^2(\S) : \Hor(-).
$$ 
If $\P$ is a double Segal space, then $\Hor(\P)$ is called the \textit{horizontal fragment} of $\P$.
This 2-fold Segal space may  be described by the pullback squares
\[
	\begin{tikzcd}
		\Hor(\P)_{n,m} \arrow[r]\arrow[d] & \P_{n,m} \arrow[d] \\
		\P_{0,0}^{\times(n+1)} \arrow[r] & \P_{0,m}^{\times(n+1)}
	\end{tikzcd}
\]
level-wise. This follows from \ref{rem.quotient-grid}; see also \cite[Remark 4.13]{HaugsengSpans}.
\end{remark}

\begin{definition}
	It now directly follows that the vertical inclusion $(-)_v : \Seg^2(\S) \rightarrow \Cat^2(\S)$ admits a right adjoint as well.
	We will write 
	$$
	\Vert(-) :  \Cat^2(\S) \rightarrow \Seg^2(\S) : \P \mapsto \Hor((\P^\hop)^\tp)
	$$
	for its right adjoint. This is called the \textit{vertical fragment} functor.
\end{definition}

\begin{remark}\label{rem.vertical-fragment}
One may show that the functors $\Hor(-)$ and $\Vert(-)$ are compatible with taking opposites as follows: 
there are natural equivalences
\begin{gather*}
	\Hor(\P^\hop) \simeq \Hor(\P)^{1-\op},\quad \Hor(\P^\vop) \simeq \Hor(\P)^{2-\op},  \\ 
	 \Vert(\P^\hop) \simeq \Vert(\P)^{2-\op}, \quad \Vert(\P^\vop) \simeq \Vert(\P)^{1-\op},
\end{gather*} for every double Segal space $\P$.

	Note that $\Vert(\P) \simeq \Hor(\P^\tp)^{2-\op}$ and 
	$X_v \simeq (X^{2-\op})_h^\tp$.
	One could also have defined the vertical inclusion and vertical fragments in such a way that there would not appear  2-opposites in these descriptions. 
	However, our introduced convention seems more natural regarding the 
	material that will follow.
\end{remark}
	
\subsection{Completeness conditions} So far, we have not talked about completeness conditions 
for 2-fold and double Segal spaces. We will start by recalling the situation for Segal spaces \cite{RezkSeg}.

\begin{definition}
	We define $J \in \PSh(\Delta)$ to be the simplicial set defined by the pushout square 
	\[
		\begin{tikzcd}[column sep = large]
		{[1]^{\sqcup 2}} \arrow[r, "\{0\leq 2\} \sqcup \{1 \leq 3\}"]\arrow[d] & {[3]} \arrow[d] \\
		{[0]^{\sqcup 2}} \arrow[r] & J
		\end{tikzcd}
	\] 
	of simplicial spaces.
\end{definition}

\begin{definition}
	A Segal space $X$ is called \textit{complete} if it is local with respect to the map 
	$J \rightarrow [0].$
\end{definition}

\begin{remark}\label{rem.joyal-tierney}It is a celebrated result of Joyal and Tierney \cite{JoyalTierney} that complete Segal spaces are a model for $\infty$-categories.
	Precisely, the functor
	$$
	\Cat_\infty \rightarrow \PSh(\Delta) : \C \mapsto ([n] \mapsto \map_{\Cat_\infty}([n],\C)),
	$$
	obtained by restricting the Yoneda embedding, is fully faithful, with image given by the complete Segal spaces.
\end{remark}

To discuss the two-dimensional completeness conditions, we will need an additional suspended version of $J$.

\begin{notation}\label{def.invertible-2cell}
	We will write $[1;J]$ for the object that fits in the pushout square 
	\[
	\begin{tikzcd}[column sep = large]
		{[1;1]}^{\sqcup 2} \arrow[rr, "{[1;\{0\leq 2\}] \sqcup [1;\{1 \leq 3\}]}"]\arrow[d] && {[1;3]} \arrow[d] \\
		{[1]}^{\sqcup 2} \arrow[rr] && {[1;J]}
	\end{tikzcd}
	\]
	in $\PSh(\Delta^{\times 2})_\mathrm{deg}$. Note that there is a canonical map $[1;J] \rightarrow [1]$.
\end{notation}

\begin{notation}\label{not.local-maps-dss}
	We will consider the following labeled families of maps: 
	\begin{enumerate}
		\item[(\textit{hc})] $[1]_h \times J_v \rightarrow [1]_h$,
		\item[(\textit{vc})] $J_h \times [1]_v \rightarrow [1]_v$,
		\item[(\textit{lhc})]  $[1;J]_h \rightarrow [1]_h$,
		\item[(\textit{lvc})]   $[1;J]_v \rightarrow [1]_v$.
	\end{enumerate}
	\end{notation}

\begin{remark}\label{rem.replacement-hc}
	In light of \ref{rem.quotient-grid}, 
	we have a pushout square 
	\[
				\begin{tikzcd}
					\{0, 1\}_h \times {J_v} \arrow[r]\arrow[d] & {[1]_h \times J_v} \arrow[d] \\
					\{0, 1\}_h   \arrow[r] & {[1;J]}_h
				\end{tikzcd}
			\]
	in $\PSh(\Delta^{\times 2})$. Consequently, the map (\textit{hc}) is contained in the saturated
	class generated by
	 (\textit{lhc}) and the map $J_v \rightarrow [0]$. Conversely, the map $J_v \rightarrow [0]_v$ is a retract 
	of (\textit{hc}), so that (\textit{lhc}) and $J_v \rightarrow [0]_v$ are contained in the saturation 
	of (\textit{hc}).
\end{remark}

The following result appears in greater generality in \cite[Section 7]{HaugsengSpans} as well.

\begin{proposition}\label{prop.locally-complete}
	Let $X$ be a 2-fold Segal space. Then the following assertions are equivalent:
	\begin{enumerate}
		\item for every two objects $x,y \in X$, the Segal space $X(x,y)$ is complete,
		\item $X$ is local with respect to the map (\textit{hc}),
		\item $X$ is local with respect to the map (\textit{lhc}).
	\end{enumerate}
\end{proposition}
\begin{proof}
	 It follows from \ref{rem.replacement-hc} that (2) and (3) are equivalent, since $J_v \rightarrow [0]_v$ is contained in the saturation of (\textit{Seg}) and (\textit{deg}).
	To show that (1) and (2) are equivalent, we use that the map $[1]_h \times J_v \rightarrow [1]_h$ is cofibered under $\{0,1\}_h$. 
	So, we obtain a commutative triangle 
	\[
	\begin{tikzcd}[row sep = small, column sep = tiny]
	\map_{\PSh(\Delta^{\times 2})}([1]_h, X)\arrow[dr] \arrow[rr] && \map_{\PSh(\Delta^{\times 2})}([1]_h \times J_v, X).\arrow[dl] \\
	& \map_{\PSh(\Delta^{\times 2})}(\{0,1\}_h, X)
	\end{tikzcd}
	\]
	Now, (2) holds if and only if the top arrow is an equivalence. This can be checked fiber-wise, which 
	precisely recovers condition (1).
\end{proof}

\begin{definition}
	Let $X$ be a 2-fold Segal space. Then $X$ is called \textit{locally complete} if the equivalent conditions of \ref{prop.locally-complete} 
	are met. We will write $$\Seg^2(\S)^{lc}\subset \Seg^2(\S)$$ for the reflective subcategory of locally complete 
	2-fold Segal spaces.
\end{definition}

\begin{proposition}\label{prop.complete}
	Let $X$ be a 2-fold Segal space. Then the following are equivalent:
	\begin{enumerate}
		\item $X^{(1)}$ is a complete Segal space,
		\item $X$ is local with respect to (\textit{vc}).
	\end{enumerate}
\end{proposition}
\begin{proof}
	This follows from the dual form of \ref{rem.replacement-hc} and the fact that $(\textit{lvc})$ is contained in the 
	saturation of (\textit{Seg}) and (\textit{deg}).
\end{proof}

\begin{definition}
	We say that a 2-fold Segal space $X$ is \textit{complete} if it is locally complete and the equivalent conditions of 
	\ref{prop.complete} are met.
\end{definition}

The reflective subcategory of $\Seg^2(\S)$ that is spanned by the 
complete 2-fold Segal spaces meets the characterization of \cite{BarwickSchommerPries} of the $\infty$-category of $(\infty,2)$-categories; see \cite[Theorem 4.16]{BarwickSchommerPries}. 
We will write $$\Cat_{(\infty,2)} \subset \Seg^2(\S)$$ for this full subcategory, and also refer to complete 2-fold Segal spaces as $(\infty,2)$-categories. 
There are concrete comparisons between complete 2-fold Segal spaces and other models for $(\infty,2)$-categories in \cite{LurieInfty2}, \cite{BergnerRezk1} and \cite{BergnerRezk2}.

\begin{remark}
	Note that the adjunction $\Seg(\S) \rightleftarrows \Seg^2(\S) : (-)^{(1)}$ restricts to an adjunction 
	$\Cat_\infty \rightleftarrows \Cat_{(\infty,2)}: (-)^{(1)}$.
\end{remark}

\begin{remark}
	One easily verifies that the nerve functor that was constructed in \ref{con.gaunt-nerve} has image 
	contained in $\Cat_{(\infty,2)}$, so it defines an inclusion 
	$$
	\mathrm{Gaunt}_2 \rightarrow \Cat_{(\infty,2)}.
	$$
	Although we will not use it in this chapter, it is worth mentioning that this functor has been extended to a fully faithful functor defined on the whole of $\Cat_2$ by Moser \cite{Moser}.
	Here, $\Cat_2$ denotes the underlying $\infty$-category (actually, a weak $(3,1)$-category) of the \textit{canonical model structure on 2-categories} \cite{Lack}. 
\end{remark}

We will now discuss several completeness conditions for double Segal spaces in terms of completeness of their 
vertical and horizontal fragments.

\begin{definition}
	A double Segal space $\P$ is called \textit{locally complete} if $\Vert(\P)$ and $\Hor(\P)$ are locally complete, i.e.\ whenever it local with respect to 
	(\textit{lhc}) and (\textit{vhc}). 
	We will write 
	$$
	\Cat^2(\S)^{lc} \subset \Cat^2(\S)
	$$
	for the full subcategory spanned by the locally complete double Segal spaces.
\end{definition}

\begin{remark}
	Note that locally complete double Segal spaces are closed under horizontal and vertical oppposites, and transposes.
\end{remark}

\begin{definition}\label{def.dbl-cat}
	A double Segal space $\P$ is called a double $\infty$-category if it is local with respect to (\textit{hc}), 
	i.e.\ if $\Vert(\P)^{(1)}$ is complete and $\Hor(\P)$ is locally complete (see \ref{rem.replacement-hc}).
	We will write $$\DblCat_\infty \subset \Cat^2(\S)$$ for the reflective subcategory spanned 
	by the double $\infty$-categories.
\end{definition}

\begin{remark}\label{rem.dbl-cat}
	We may view double $\infty$-categories as simplicial objects in $\Cat_\infty$ as follows.
	Applying $\fun(\Delta^\op, -)$ to the embedding that 
	was discussed in \ref{rem.joyal-tierney}, we will obtain a fully faithful functor 
	$$
	i : \fun(\Delta^\op, \Cat_\infty) \rightarrow \PSh(\Delta^{\times 2})
	$$
	that carries a simplicial $\infty$-category $\P'$ to the bisimplicial space 
	$$
	i(\P')_{n,m} = \map_{\Cat_\infty}([m], \P'_n).
	$$
	One readily verifies that this restricts to an equivalence between:
	\begin{enumerate}
		\item the \textit{Segal objects} in $\Cat_\infty$, i.e.\ simplicial objects $\P' : \Delta^\op \rightarrow \Cat_\infty$ 
		so that each functor 
		$\P'_n \rightarrow \P'_1 \times_{\P'_0} \dotsb \times_{\P'_0} \P'_1$
	   is an equivalence for every $n \geq 0$,
	   \item the double $\infty$-categories.
	\end{enumerate}
	We will leave this identification implicit, and view every double $\infty$-category as a Segal object in $\Cat_\infty$.
\end{remark}

\begin{remark}
	The subcategory $\DblCat_\infty$ is not closed under transposition, as one may readilt check. The 
	involutions $(-)^\hop$ and $(-)^\vop$ do restrict to an endofunctor on $\DblCat_\infty$.
	Suppose that $\P$ is a double $\infty$-category, viewed as a simplicial $\infty$-category. Then the vertical opposite $\P$ 
	is described level-wise by  
	$$
	(\P^\vop)_n = \P_n^\op.
	$$
\end{remark}

\begin{example}\label{ex.ccat-infty}
	In \cite{EquipI}, we discuss the double $\infty$-category $$\CCAT_\infty$$ which is loosely described as follows:
	\begin{itemize}
		\item its objects are are $\infty$-categories,
		\item its vertical arrows are functors,
		\item its horizontal arrows are  \textit{profunctors}, i.e.\ a horizontal arrow $F : \C \rightarrow \D$ is given by a functor 
		$\D^\op \times \C \rightarrow \S$,
		\item its 2-cells are given as follows: a 2-cell \[
	\begin{tikzcd}
		\mathscr{A} \arrow[r ,"F"name=f] \arrow[d,"f"'] & \mathscr{B} \arrow[d, "g"] \\
		\C \arrow[r,"G"to=T] & \D
		\arrow[from=f,to=t,Rightarrow, shorten <= 6pt, shorten >= 6pt]
	\end{tikzcd}
	\]
	corresponds to a natural transformation 
		\[
			\begin{tikzcd}[column sep = large]
				|[alias=f]|\mathscr{B}^\op \times \mathscr{A} \arrow[r,"F"]\arrow[d, "g^\op \times f"'] & \S. \\
				\D^\op \times \C \arrow[ur, "G"'name=t]
				\arrow[from=f,to=t,Rightarrow, shorten >= 8pt, shorten <= 4pt]
			\end{tikzcd}
		\]
	\end{itemize}
	More generally, we construct in \cite{EquipI} a double $\infty$-ca\-te\-go\-ry $$\CCAT_\infty(\E),$$
	for every suitable $\infty$-category $\E$, whose objects are \textit{$\infty$-categories internal to $\E$}. In \cite{HaugsengEnriched}, 
	Haugseng constructs a similarly described double $\infty$-category $$\CCAT_\infty^\mathscr{V}$$ of $\infty$-categories enriched in a suitably monoidal $\infty$-category $\mathscr{V}$.
\end{example}

In practice, double $\infty$-categories, such as these in \ref{ex.ccat-infty}, are usually more conviently constructed as Segal objects in $\Cat_\infty$. We will briefly discuss how 
one may verify the completeness of double $\infty$-categories when viewed as simplicial $\infty$-categories.

\begin{notation}
	Let $\P$ be a double $\infty$-category, viewed as a Segal object in $\Cat_\infty$. Then we write $\P_\mathrm{eq}$ for the $\infty$-category 
	defined by the pullback square 
	\[
	\begin{tikzcd}
		\P_\mathrm{eq} \arrow[r]\arrow[d] & \P_3 \arrow[d,"{(\{0\leq 2\}^*,\{1 \leq 3\}^*)}"] \\
		\P_0\times \P_0 \arrow[r,"s_0^* \times s_0^*"] & \P_1 \times \P_1.
	\end{tikzcd}
	\]
\end{notation}

\begin{proposition}\label{prop.loc-complete-dss}
	Suppose that $\P$ is a double Segal space. Then the following assertions are equivalent:
	\begin{enumerate}
		\item $\Vert(\P)$ is complete and $\Hor(\P)$ is locally complete,
		\item $\P$ is a double $\infty$-category and the canonical functor $\P_0 \rightarrow \P_\mathrm{eq}$ is fully faithful.
	\end{enumerate}
\end{proposition}
\begin{proof}
	We must show that a double $\infty$-category $\P$ is local with respect to (\textit{vhc}) 
	if and only if $\P_0 \rightarrow \P_\mathrm{eq}$ is fully faithful.
	This functor is fully faithful if and only if it is left orthogonal to the inclusion $\{0,1\}\rightarrow [1]$, i.e.\ if and only if the induced square
	\[
	\begin{tikzcd}
		\map_{\Cat_\infty}([1], \P_{0}) \arrow[r]\arrow[d] & \map_{\Cat_\infty}([1], \P_{\mathrm{eq}}) \arrow[d] \\
		\map_{\Cat_\infty}(\{0,1\}, \P_0) \arrow[r] & \map_{\Cat_\infty}(\{0,1\}, \P_{\mathrm{eq}})
	\end{tikzcd}
	\]
	is a pullback square. This can be identified with the commutative square 
	\[
	\begin{tikzcd}
		\map_{\PSh(\Delta^{\times 2})}([1]_v \times \{0\}_h, \P) \arrow[r]\arrow[d] & \map_{\PSh(\Delta^{\times 2})}([1]_v \times J_h, \P) \arrow[d] \\
		\map_{\PSh(\Delta^{\times 2})}(\{0,1\}_v \times \{0\}_h, \P)\arrow[r] & \map_{\PSh(\Delta^{\times 2})}(\{0,1\}_v \times J_h, \P).
	\end{tikzcd}
	\]
	In light of the pushout square that appears in \ref{rem.replacement-hc},
	this square is a pullback square if and only if the restriction map 
	$$
	\map_{\PSh(\Delta^{\times 2})}([1]_v, \P) \rightarrow \map_{\PSh(\Delta^{\times 2})}([1;J]_v, \P)
	$$
	induced by (\textit{vhc}) is an equivalence.
\end{proof}

\begin{example}
	In \cite{EquipI}, we show that $\CCAT_\infty(\E)$ is locally complete 
	for every suitable $\infty$-category $\E$. In particular, $\CCAT_\infty$ is locally complete. 
	Throughout this article, we will denote its vertical fragment by 
	$$
	\CAT_\infty := \Vert(\CCAT_\infty).
	$$
	It is our preferred model for the $(\infty,2)$-category of $\infty$-categories.
\end{example}

\begin{proposition}\label{prop.complete-dss}
	Let $\P$ be a double Segal space. Then the following assertions are equivalent:
	\begin{enumerate}
		\item $\Vert(\P)$ and $\Hor(\P)$ are complete,
		\item $\P$ is local with respect to (\textit{hc}) and (\textit{vc}),
		\item $\P$ is a double $\infty$-category and the functor $\P_0 \rightarrow \P_\mathrm{eq}$ is an equivalence.
	\end{enumerate}
\end{proposition}
\begin{proof}
	It is readily verified that (1) and (2) are equivalent using \ref{rem.replacement-hc}. To show that (2) and (3) are equivalent, 
	we must show that a double $\infty$-category $\P$ is local with respect to (\textit{vc}) if and only if 
	$\P_0 \rightarrow \P_\mathrm{eq}$ is an equivalence. But this readily 
	follows from \ref{prop.loc-complete-dss}, and 
	the observation that one recovers the map  $\map_{\PSh(\Delta^{\times 2})}([0]_h, \P) \rightarrow \map_{\PSh(\Delta^{\times 2})}(J_h, \P)$ 
	if one applies $\map_{\Cat_\infty}([0],-)$ to the functor $\P_0 \rightarrow \P_\mathrm{eq}$.
\end{proof}

\begin{definition}
	A double Segal space is called \textit{complete} if it meets the equivalent conditions of \ref{prop.complete-dss}. We will write 
	$$\DblCat_\infty^c \subset \DblCat_\infty$$ for the full and reflective subcategory spanned by the complete double Segal spaces.
\end{definition}

\begin{remark}
	The complete double $\infty$-categories are closed under horizontal and vertical opposites, and transposes.
\end{remark}

We conclude this subsection by summarizing the reflective subcategory of
$\PSh(\Delta^{\times 2})$ that were constructed above:
\begin{table}[H]
	\centering
\renewcommand{\arraystretch}{1.2}
\begin{tabular}{| c | c | c |}
	\hline
	\textbf{Subcategory} & \textbf{Local objects} & \textbf{Generating local maps} \\
	\hline 
	$\Cat^2(\S)$ & \textit{Double Segal spaces} & (\textit{Seg}) \\ 
	\hline
	$\Cat^2(\S)^{lc}$ & \textit{Locally complete double Segal spaces} & (\textit{Seg}), (\textit{lvc}), (\textit{lhc})\\
	\hline
	$\DblCat_\infty$ & \textit{Double $\infty$-categories} & (\textit{Seg}), (\textit{hc}) \\
	\hline
	$\DblCat_\infty^c$ & \textit{Complete double $\infty$-categories}  & (\textit{Seg}), (\textit{hc}), (\textit{vc})\\
	\hline
\end{tabular}
\caption{Relevant notions of double Segal spaces}\label{table.dss}
\end{table}

\begin{table}[H]
	\centering
	\renewcommand{\arraystretch}{1.2}
	\begin{tabular}{| c | c | c |}
		\hline
		\textbf{Subcategory} & \textbf{Local objects} & \textbf{Generating local maps} \\
		\hline
		$\Seg^2(\S)$ & \textit{2-Fold Segal spaces} & (\textit{Seg}), (\textit{deg}) \\
		\hline
		$\Seg^2(\S)^{lc}$ & \textit{Locally complete 2-fold Segal spaces} & (\textit{Seg}), (\textit{deg}), (\textit{lhc}) \\
		\hline
		$\Cat_{(\infty,2)}$ & \textit{$(\infty,2)$-Categories} & (\textit{Seg}), (\textit{deg}), (\textit{hc}), (\textit{vc})\\
		\hline
	\end{tabular}
	\caption{Relevant notions of 2-fold Segal spaces}\label{table.2fss}
\end{table}

\subsection{Exponentiability of double Segal spaces} 
We will now show the following:

\begin{proposition}\label{prop.exp-ideals-dss}
	The reflective subcategories listed in \ref{table.dss} are exponential ideals of $\PSh(\Delta^{\times 2})$. In particular, 
	all these subcategories are cartesian closed and inherit the internal homs from $\PSh(\Delta^{\times 2})$.
\end{proposition}

In order to prove \ref{prop.exp-ideals-dss}, we will use the following result of Rezk:

\begin{lemma}[Rezk]\label{lem.Jcilinder}
	Suppose that $K = [0]$ or $J$ and consider the the pushout square 
	\[
		\begin{tikzcd}
			\{0,1\} \times [1] \arrow[r]\arrow[d] & \{0,1\} \times K \arrow[d] \\
			{[1]} \times [1] \arrow[r] & P(K)
		\end{tikzcd}
	\]
	in $\PSh(\Delta)$. Then the canonical map $P(K) \rightarrow [1] \times K$ is 
	is contained in the saturated class 
	of morphisms generated by the spine inclusions.
\end{lemma}
\begin{proof}
	This is essentially the content of \cite[Subsection 12.3]{RezkSeg}. 
\end{proof}

\begin{proof}[Proof of \ref{prop.exp-ideals-dss}]
	Let $\C$ be one of these full subcategories of $\PSh(\Delta^{\times 2})$.
	Write $S$ for the generating local maps of the localization $\C$, as listed in 
	the table in \ref{ssection.dss}.
	Let $X$ be a bisimplicial space. Then \ref{lem.exp-ideals} asserts 
	that we have to show that the functor 
	$$
	X \times (-) : \PSh(\Delta^{\times 2}) \rightarrow \PSh(\Delta^{\times 2})
	$$
	carries each map $f\in S$ to a map in the saturation of $S$. Since the product 
	preserves colimits in each variable, we may reduce to showing this for $X = [n,m]$. 
	But one can readily verify that $[n,m]$ is a retract of $[1,0]^{\times n} \times [0,1]^{\times m}$. 
	Thus it suffices to handle the cases that 
	$X = [1,0]$ and $X = [0,1]$.

	Now, in any of the cases that $f$ is of type (\textit{Seg}), the desired result can directly be obtained 
	from the 1-dimensional case \cite[Proposition 10.3]{RezkSeg}.
	
	Suppose that $f$ is of the form (\textit{hc}). If $X = [1]_h$, the result readily follows.
	Hence, it remains to show that 
	$$
	[1]_v \times ([1]_h \times J_v) = [1]_h \times ([1]_v \times J_v) \rightarrow [1]_h \times [1]_v = [1]_v \times [1]_h
	$$
	sits in the saturation of (\textit{Seg}) and (\textit{hc}). 
	Let $P(K) \rightarrow [1] \times K$ be the map of \ref{lem.Jcilinder} for $K = [0]$ and $J$. 
	It follows from this lemma that the vertical arrows in the commutative square
	\[
		\begin{tikzcd}
			{[1]_h} \times P(J)_v \arrow[r]\arrow[d] & {[1]_h} \times P([0])_v \arrow[d] \\
			{[1]_h} \times [1]_v \times J_v \arrow[r] & {[1]_h} \times [1]_v.
		\end{tikzcd}
	\]
	are contained in the saturation of (\textit{Seg}). The upper map is a pushout along $
	[1]_h \times \{0,1\}_v \times J_v \rightarrow [1]_h \times \{0,1\}_v$, so it lies in the saturation of (\textit{hc}). Consequently, 
	by 2-out-of-3, the bottom map lies in the saturation generated by (\textit{Seg}) and (\textit{hc}), as desired.
	
	Let us now handle the map $f$ of type (\textit{lhc}) and the case that
	 $X = [1]_h$. One can check that the map 
	$$
	([1] \times [1;J])_h \rightarrow ([1] \times [1])_h
	$$
	is obtained by taking pushouts of the rows in the commutative diagram
	\[
		\begin{tikzcd}
			{[2; J, 0]}_h \arrow[d] & \arrow[l] [1;J]_h \arrow[r] \arrow[d]& {[2;0,J]_h} \arrow[d] \\
			{[2]}_h & \arrow[l] [1]_h \arrow[r] & {[2]_h}.
		\end{tikzcd}
	\]
	The middle vertical map is precisely (\textit{lhc}). Consequently, it suffices to check that the outer 
	vertical maps are contained in the saturation of (\textit{Seg}) and (\textit{lhc}). But this readily follows 
	from the observation that the canonical inclusions 
	$$
	[1;J]_h \cup_{[0]_h} [1]_h \rightarrow [2; J, 0]_h, \quad \text{and} \quad [1]_h \cup_{[0]_h} [1;J]_h \rightarrow [2; 0, J]_h
	$$
	are contained in the saturation (\textit{Seg}). 
	
	Next, we handle the case that $f$ is of type (\textit{lhc}) and $X = [1]_v$. We have to verify that 
	$$
	[1]_v \times [1;J]_h \rightarrow  [1]_v \times [1]_h
	$$
	sits in the saturation of (\textit{lhc}) and (\textit{Seg}). 
	Let $(-) \boxtimes (-)$ denote the pushout-product for arrows in $\PSh(\Delta^{\times 2})$. 
	We can then write
	$$
	[1]_v \times [1;K]_h = (\{0,1\}_h \rightarrow [1]_h) \boxtimes (([1] \times K)_v \rightarrow [1]_v).
	$$
	for $K = [0]$ and $J$.
	Using the pushout-product, we define
	$$
	Q(K) := (\{0,1\}_h \rightarrow [1]_h) \boxtimes (P(K)_v \rightarrow [1]_v),
	$$
	and obtain a commutative square
	\[
		\begin{tikzcd}
			Q(J) \arrow[r]\arrow[d] & Q([0]) \arrow[d] \\
			{[1]_v \times [1;J]_h} \arrow[r] &  {[1]_v \times [1]_h}.
		\end{tikzcd}
	\]
	The vertical arrows lie in the saturation of (\textit{Seg}) on account of \ref{lem.Jcilinder}. 
	The top arrow is a pushout along 
	$$
	(\{0,1\}_h \rightarrow [1]_h) \boxtimes ((\{0,1\} \times J)_v \rightarrow \{0,1\}_v) \rightarrow (\{0,1\}_h \rightarrow [1]_h) \boxtimes (\{0,1\}_v \rightarrow \{0,1\}_v)
	$$
	and this is precisely the map $\{0,1\}_v \times [1;J]_h \rightarrow \{0,1\}_v$, which 
	is contained in the saturation of (\textit{lhc}). Thus the result follows from 2-out-of-3.

	The cases (\textit{vc}) and (\textit{lvc}) formally follow from the above cases, and this finishes the proof.
\end{proof}

\begin{definition}
	The \textit{double Segal space of functors} between a 
	bisimplicial space $X$ and a double Segal space $\P$ is defined to be 
	$$
	\FFUN(X, \P) := \mathrm{Hom}_{\Cat^2(\S)}(X, \P) = \Hom_{\PSh(\Delta^{\times 2})}(X, \P).
	$$
\end{definition}

\begin{definition}
	The \textit{vertical} and \textit{horizontal cotensor products}  of a 2-fold Segal space $X$ with a double Segal space are defined as
	$$
	[X,\P] := \FFUN(X_v, \P) \quad \text{ and } \quad \{X,\P\} := \FFUN(X_h, \P)
	$$
	respectively.
\end{definition}

Note that $\PSh(\Delta^{\times 2})_\mathrm{deg}$ is \textit{not} an exponential ideal of $\PSh(\Delta^{\times 2})$. However, 
it is still cartesian closed since it is a presheaf $\infty$-topos (see \ref{rem.quotient-grid}).
The following can be readily deduced from \ref{prop.exp-ideals-dss}:

\begin{corollary}
	The full subcategories $\Seg^2(\S)$, $\Seg^2(\S)^{lc}$ and $\Cat_{(\infty,2)}$ are exponentiable 
	ideals of $\PSh(\Delta^{\times 2})_{\mathrm{deg}}$.
\end{corollary}

\begin{notation}
	The \textit{2-fold Segal space of functors} between 2-fold Segal spaces $X$ and $Y$ is denoted by 
	$$
	\FUN(X,Y) := \Hom_{\Seg^2(\S)}(X,Y) = \Hom_{\PSh(\Delta^{\times2})_{\mathrm{deg}}}(X,Y).
	$$
\end{notation}

The following result asserts that the defined cotensor products are compatible with fragments:

\begin{proposition}\label{prop.cotensor-vert}
	Let $X$ and $\P$ be a 2-fold and double Segal space respectively. There are canonical equivalences of 2-fold Segal spaces
	$$
	\Vert([X,\P]) \xrightarrow{\simeq} \FUN(X, \Vert(\P)), \quad \Hor(\{X,\P\}) \xrightarrow{\simeq} \FUN(X, \Hor(\P)).
	$$
\end{proposition}
\begin{proof}
	This follows immediately from the fact that the inclusions $(-)_h$ and $(-)_v$ preserve products.
\end{proof}

\begin{remark}\label{rem.vert-cotensor-dblcat}
	Let $X$ be an $\infty$-category and suppose that $\P$ is a double $\infty$-category. Then the vertical cotensor
	 double $\infty$-category
	$[X,\P]$ is given by the composite 
	$$
	\Delta^\op \rightarrow \Cat_\infty \xrightarrow{\fun(X,-)} \Cat_\infty
	$$
	as a simplicial $\infty$-category (see \ref{rem.dbl-cat}).
	To see this, we note that $X_v$ is always a double $\infty$-category and $(X_v)_n = X$ for all $n$.
	Since the functor $[X, -] : \DblCat_\infty \rightarrow \DblCat_\infty$ is right adjoint to the functor 
	$$
	X_v \times (-) : \DblCat_\infty \rightarrow \DblCat_\infty,
	$$
	the observation follows.
\end{remark}

\begin{example}\label{ex.cotensors-spans}
	Let $X$ be an $\infty$-category and $\C$ be an $\infty$-category with finite limits. Then 
	we may apply the vertical cotensor construction to the double $\infty$-category $\SSPAN(\C)$  of spans that is 
	constructed in \cite{HaugsengSpans}.
	One readily deduces using \ref{rem.vert-cotensor-dblcat} that there is a canonical equivalence of double $\infty$-categories
	$$
	[X,\SSPAN(\C)] \simeq \SSPAN(\fun(X,\C)).
	$$
\end{example}

\begin{example}\label{ex.indexed-cats}
	Let $T$ be an $\infty$-category. 
	We may apply the above construction to the double $\infty$-category $\CCAT_\infty$ (see \ref{ex.ccat-infty})
	so that we obtain a (locally complete) double $\infty$-category
	$$
	[T^\op, \CCAT_\infty]
	$$
	whose objects are  called \textit{$T$-indexed $\infty$-categories}. Its vertical fragment is given by the 
	 $(\infty,2)$-category of functors $\FUN(T^\op, \CAT_\infty)$.
	It is shown in \cite{EquipI} that $\CCAT_\infty \simeq \CCAT_\infty(\S)$,
	so there is an equivalence
	$$
	[T^\op, \CCAT_\infty] \simeq [T^\op, \CCAT_\infty(\S)].
	$$
	See \ref{ex.ccat-infty} for the notation.
	The right hand-side can be identified with the double $\infty$-category $\CCAT_\infty(\PSh(T))$, as may 
	be concluded directly from the construction of $\CCAT_\infty(-)$ in \cite{EquipI}, so that we obtain an equivalence
	$$
	[T^\op, \CCAT_\infty] \simeq  \CCAT_\infty(\PSh(T)).
	$$

	More generally, one can consider the vertical cotensor product 
	$$
	[\X^{1-\op}, \CCAT_\infty],
	$$
	where $\X$ is an $(\infty,2)$-category. The objects of this double $\infty$-category are functors $\X^{1-\op} \rightarrow \CAT_\infty$, and these are usually 
	called \textit{2-presheaves}.
\end{example}

\subsection{The squares construction}\label{ssection.squares-def}

We will make use of the \textit{squares functor} for 2-fold Segal spaces in this article.  In the strict context, this construction goes back to Ehresmann \cite{Ehresmann}.
It has been studied by Grandis and Par\'e \cite{GrandisPare}, and recently by Guetta-Moser-Sarazola-Verdugo \cite{GuettaMoserSarazolaVerdugo} as well. 
In the $\infty$-categorical context, it was introduced in the work of Gaitsgory--Rozenblyum on derived algebraic geometry, 
where it plays an important role in setting up the $(\infty,2)$-categorical foundations for their treatment of six-functor formalisms. It has also been studied recently 
by Abell\'an \cite{Abellan}.

To define the squares functor, we make use of the gaunt 2-categories
\[
[n] \otimes [m] = 
\begin{tikzcd}[column sep = small, row sep = small]
	(0,0)\arrow[d]  \arrow[r] & |[alias=f1]|(1,0)\arrow[d] \arrow[r] & \dotsb \arrow[r] & |[alias=f2]| (n-1,0)\arrow[r]\arrow[d] & |[alias=f3]| (n,0)\arrow[d] \\
	|[alias=t1]|(0,1) \arrow[r]\arrow[d] & |[alias=t2]|(1,1) \arrow[r]\arrow[d] & \dotsb \arrow[r] & |[alias=t3]|(n-1, 1)\arrow[d]\arrow[r] & |[alias=t4]|(n,1)\arrow[d] \\
	\vdots \arrow[d] &\vdots \arrow[d] & & \vdots \arrow[d] & \vdots\arrow[d]\\
	|[alias=g1]|(0, m-1)\arrow[d] \arrow[r] & |[alias=g2]|(1,m-1)\arrow[d] \arrow[r] & \dotsb \arrow[r] & |[alias=g3]|(n-1, m-1) \arrow[d] \arrow[r] &  |[alias=g4]| (n, m-1)\arrow[d] \\
	|[alias=y1]|(0,m) \arrow[r] &|[alias=y2]| (1,m) \arrow[r] & \dotsb \arrow[r] & |[alias=y3]|(n-1,m) \arrow[r] & (n,m)
	\arrow[from=f1,to=t1,Rightarrow, shorten <= 6pt, shorten >= 6pt]
	\arrow[from=f2,to=t2,Rightarrow, shorten <= 6pt, shorten >= 6pt]
	\arrow[from=f3,to=t3,Rightarrow, shorten <= 6pt, shorten >= 6pt]
	\arrow[from=t2,to=g1,Rightarrow, shorten <= 10pt, shorten >= 10pt]
	\arrow[from=t3,to=g2,Rightarrow, shorten <= 10pt, shorten >= 10pt]
	\arrow[from=t4,to=g3,Rightarrow, shorten <= 10pt, shorten >= 10pt]
	\arrow[from=g2,to=y1,Rightarrow, shorten <= 6pt, shorten >= 6pt]
	\arrow[from=g3,to=y2,Rightarrow, shorten <= 6pt, shorten >= 6pt]
	\arrow[from=g4,to=y3,Rightarrow, shorten <= 6pt, shorten >= 6pt]
\end{tikzcd}
\]
which
are obtained by taking the oplax Gray tensor product of $[n]$ and $[m]$; see \cite{Gray} for the original definition. 
These 2-categories can be organized into 
a bicosimplicial object
$$
[\cdot] \otimes [\cdot] : \Delta^{\times 2} \rightarrow \PSh(\Delta^{\times 2})_{\mathrm{deg}} : ([n],[m]) \mapsto [n] \otimes [m].
$$
We will need the following result:

\begin{proposition}\label{prop.combinatorics-grids}
	Let $n,m \geq 1$. Then the canonical maps
	\begin{enumerate}[label=(\roman*)]
	\item $
	[1] \otimes [m] \cup_{[0] \otimes [m]} \dotsb \cup_{[0] \otimes [m]} [1] \otimes [m] \rightarrow [n] \otimes [m],
	$
	\item $
	[n] \otimes [1] \cup_{[n] \otimes [0]} \dotsb \cup_{[n] \otimes [0]} [n] \otimes [1] \rightarrow [n] \otimes [m],
	$
	\end{enumerate}
	are equivalences of 2-fold Segal spaces.
\end{proposition}

Instead of giving a direct proof of \ref{prop.combinatorics-grids} here, we will show how it can be leveraged from the literature. We will make use 
of the following observation: 

\begin{lemma}\label{lem.eq-infty2-vs-Segal}
	Suppose that $f : X \rightarrow Y$ is a map between 2-fold Segal spaces so that:
	\begin{enumerate}
	\item $f_{0,0}$ is an equivalence of spaces,
	\item $f^{(1)}$ is carried to an equivalence by the reflector $\Seg(\S) \rightarrow \Cat_\infty$,
	\item $f$ is carried to an equivalence by the reflector $\Seg^2(\S) \rightarrow \Cat_{(\infty,2)}$.
	\end{enumerate}
	Then $f$ is an equivalence of 2-fold Segal spaces.
\end{lemma}
\begin{proof}
	In light of the Segal condition and assumption (1),
	we have to verify that $f_{1,0}$ and $f_{1,1}$ are equivalences of spaces. Since (2) holds, the map 
	$$
	X_{1,0} \rightarrow Y_{1,0} \times_{Y_{0,0}^{\times 2}} X_{0,0}^{\times 2}
	$$
	induced by $f_{1,0}$ must be an equivalence; see \cite[Theorem 1.2.13]{LurieInfty2} for instance. Hence $f_{1,0}$ is an equivalence. 
	In light of \cite[Remark 7.19]{HaugsengSpans}, assumption (1) implies that 
	$$
	X(x,y)_1 \rightarrow Y(f(x),f(y))_1 \times_{Y(f(x),f(y))_0^{\times 2}} X(x,y)_0^{\times 2}
	$$
	is an equivalence for all objects $x,y \in X$. Recall 
	that the map $X(x,y)_n \rightarrow Y(f(x),f(y))_n$ is precisely the comparison map 
	$$X_{1,n} \rightarrow Y_{1,n} \times_{Y_{0,n}^{\times 2}} X_{0,n}^{\times 2}$$ on fibers above $(x,y) \in X_{0,0}^{\times 2} \xrightarrow{\simeq} X_{0,n}^{\times 2}$.
	Since $f_{1,0}$ is an equivalence, it follows that $X(x,y)_0 \rightarrow Y(x,y)_0$ is an equivalence. Thus 
	$X(x,y)_1 \rightarrow Y(f(x),f(y))_1$ must be an equivalence. 
	Since this holds for all $x,y$ in $X$, this in turn implies that $f_{1,1}$ is an equivalence.
\end{proof}

\begin{proof}
	We will just treat map (1); the other one is handled similarly. Let us consider the inclusion
	$$
	i : [1] \otimes [m] \cup_{[0] \otimes [m]} \dotsb \cup_{[0] \otimes [m]} [1] \otimes [m] \rightarrow [n] \otimes [m],
	$$
	of bisimplicial sets, where the iterated pushout on the left is computed in $\PSh(\Delta^{\times 2})$. 
	Then one readily verifies that $i_{0,0}$ is the identity. Moreover, one can readily show that $i^{(1)}$ lies in the saturated class 
	spanned by the 1-dimensional spine inclusions (see \ref{rem.segal-spaces}).
	Let us consider the reflectors $$
	L : \PSh(\Delta^{\times 2})_{\mathrm{deg}} \rightarrow \Seg^2(\S), \quad L' : \PSh(\Delta) \rightarrow \Seg(\S).
	$$
	One can verify that the (colimit-preserving)  functor $(-)^{(1)} : \PSh(\Delta^{\times 2})_{\mathrm{deg}} \rightarrow \PSh(\Delta)$ 
	carries  maps in (\textit{Seg}) to the saturated class of 1-dimensional spine inclusions. So it follows that $L'(Li)^{(1)} \simeq L'i^{(1)}$ is an equivalence.
	A similar argument shows that $(Li)_{0,0}$ is an equivalence, as every map in (\textit{Seg}) is carried to an equivalence by the evaluation functor 
	$(-)_{0,0}$.

	Now, \ref{lem.eq-infty2-vs-Segal} asserts that it is sufficient to check that $Li$ is carried to an equivalence 
	by the reflector $\Seg^2(\S) \rightarrow \Cat_{(\infty,2)}$.
	This follows from the work of Gagna-Harpaz-Lanari \cite{GagnaHarpazLanari} and \cite[Proposition 5.1.9]{HHLN}.
\end{proof}

\begin{construction}
	Since $\Seg^2(\S)$ is cocomplete, the bicosimplicial object 
	$$
	[\cdot] \otimes [\cdot] : \Delta^{\times 2} \rightarrow \Seg^2(\S) : ([n], [m]) \rightarrow [n] \otimes [m] 
	$$
	extends to a colimit preserving functor $\PSh(\Delta^{\times 2}) \rightarrow \Seg^2(\S)$. On account of  \ref{prop.combinatorics-grids}, 
	there exists a unique dotted extension in the diagram
	\[
		\begin{tikzcd}
			\PSh(\Delta^{\times 2}) \arrow[r]\arrow[d, "L"'] & \Seg^2(\S). \\
			\Cat^2(\S) \arrow[ur, dotted, "\mathrm{Gr}"']
		\end{tikzcd}
	\]
	The horizontal functor admits a right adjoint, which restricts to a right adjoint
	$$
	\Sq : \Seg^2(\S) \rightarrow \Cat^2(\S)
	$$ 
	for $\mathrm{Gr}$. This is called the \textit{squares functor}. For a 2-fold Segal space $X$, it is level-wise given by 
	$$
	\Sq(X)_{n,m} := \map_{\Seg^2(\S)}([n] \otimes [m], X).
	$$
\end{construction}

\begin{construction}\label{constr.tautological-maps}
	There is a canonical functor
	$$
	\iota_h^{n,m} : [n] \otimes [m] \rightarrow [n;m,\dotsc, m]
	$$
	between gaunt 2-categories that is natural in $([n], [m]) \in \Delta^{\times 2}$.
	The functor $\iota_h^{n,m}$ selects the non-degenerate oplax $n\times m$ grid in $[n; m,\dotsc,m]$ where all the vertical arrows in the grid are identities,
	pictured as 
	\[
		\begin{tikzcd}[row sep = small]
			0\arrow[d,equal]  \arrow[r] & |[alias=f1]|1\arrow[d,equal] \arrow[r] & \dotsb \arrow[r] & |[alias=f2]| n-1\arrow[r]\arrow[d,equal] & |[alias=f3]| n\arrow[d,equal] \\
			|[alias=t1]|0 \arrow[r]\arrow[d,equal] & |[alias=t2]|1 \arrow[r]\arrow[d,equal] & \dotsb \arrow[r] & |[alias=t3]|n-1\arrow[d,equal]\arrow[r] & |[alias=t4]|n\arrow[d,equal] \\
			\vdots \arrow[d,equal] &\vdots \arrow[d,equal] & & \vdots \arrow[d,equal] & \vdots\arrow[d,equal]\\
			|[alias=g1]|0\arrow[d,equal] \arrow[r] & |[alias=g2]|1 \arrow[d,equal] \arrow[r] & \dotsb \arrow[r] & |[alias=g3]|n-1 \arrow[r]\arrow[d,equal] &  |[alias=g4]| n\arrow[d,equal] \\
			|[alias=y1]|0 \arrow[r] &|[alias=y2]| 1 \arrow[r] & \dotsb \arrow[r] & |[alias=y3]|n-1 \arrow[r] & n.
			\arrow[from=f1,to=t1,Rightarrow, shorten <= 6pt, shorten >= 6pt, "0"]
	\arrow[from=f2,to=t2,Rightarrow, shorten <= 6pt, shorten >= 6pt]
	\arrow[from=f3,to=t3,Rightarrow, shorten <= 6pt, shorten >= 6pt, "0"]
	\arrow[from=t2,to=g1,Rightarrow, shorten <= 10pt, shorten >= 10pt]
	\arrow[from=t3,to=g2,Rightarrow, shorten <= 10pt, shorten >= 10pt]
	\arrow[from=t4,to=g3,Rightarrow, shorten <= 10pt, shorten >= 10pt]
	\arrow[from=g2,to=y1,Rightarrow, shorten <= 6pt, shorten >= 6pt, "m"]
	\arrow[from=g3,to=y2,Rightarrow, shorten <= 6pt, shorten >= 6pt]
	\arrow[from=g4,to=y3,Rightarrow, shorten <= 6pt, shorten >= 6pt, "m"]
		\end{tikzcd}
	\]
	There is also a similarly defined functor 
	$$
	\iota_v^{n,m} : [n]^\op \otimes [m] \rightarrow [m;n\dotsc, n]
	$$
	between gaunt 2-categories.
	The functor $\iota_v^{n,m}$ selects the non-degenerate oplax $n \times m$ grid in $[m;n\dotsc,n]$ where all horizontal arrows in the grid are identities, 
	pictured as 
	\[
\begin{tikzcd}[column sep = small]
	0\arrow[d]  \arrow[r,equal] & |[alias=f1]|0\arrow[d] \arrow[r,equal] & \dotsb \arrow[r,equal] & |[alias=f2]| 0\arrow[r,equal]\arrow[d] & |[alias=f3]| 0\arrow[d] \\
	|[alias=t1]|1 \arrow[r,equal]\arrow[d] & |[alias=t2]|1 \arrow[r,equal]\arrow[d] & \dotsb \arrow[r,equal] & |[alias=t3]| 1\arrow[d]\arrow[r,equal] & |[alias=t4]|1\arrow[d] \\
	\vdots \arrow[d] &\vdots \arrow[d] & & \vdots \arrow[d] & \vdots\arrow[d]\\
	|[alias=g1]|m-1\arrow[d] \arrow[r,equal] & |[alias=g2]|m-1\arrow[d] \arrow[r,equal] & \dotsb \arrow[r,equal] & |[alias=g3]|m-1\arrow[d] \arrow[r,equal] &  |[alias=g4]| m-1\arrow[d] \\
	|[alias=y1]|m \arrow[r,equal] &|[alias=y2]| m \arrow[r,equal] & \dotsb \arrow[r,equal] & |[alias=y3]|m \arrow[r,equal] & m.
	\arrow[from=f1,to=t1,Rightarrow, shorten <= 6pt, shorten >= 6pt, "n"]
	\arrow[from=f2,to=t2,Rightarrow, shorten <= 6pt, shorten >= 6pt]
	\arrow[from=f3,to=t3,Rightarrow, shorten <= 6pt, shorten >= 6pt,"0"]
	\arrow[from=t2,to=g1,Rightarrow, shorten <= 10pt, shorten >= 10pt]
	\arrow[from=t3,to=g2,Rightarrow, shorten <= 10pt, shorten >= 10pt]
	\arrow[from=t4,to=g3,Rightarrow, shorten <= 10pt, shorten >= 10pt]
	\arrow[from=g2,to=y1,Rightarrow, shorten <= 6pt, shorten >= 6pt ,"n"]
	\arrow[from=g3,to=y2,Rightarrow, shorten <= 6pt, shorten >= 6pt]
	\arrow[from=g4,to=y3,Rightarrow, shorten <= 6pt, shorten >= 6pt, "0"]
\end{tikzcd}
\] 

	Let $X$ be a 2-fold Segal space. Then \ref{rem.quotient-grid} implies that $X_h$ and $X_v$ are (naturally) described by 
	\begin{gather*}
	(X_h)_{n,m} \simeq \map_{\PSh(\Delta^{\times 2})}([n;m\dotsc, m]_h, X_h), \\
	(X_v^\hop)_{n,m} \simeq \map_{\PSh(\Delta^{\times 2})}([m;n\dotsc, n]_v, X_v).
	\end{gather*}
	In light of these descriptions, we can construct maps
	$$
	X_h \rightarrow \Sq(X), \quad X_v \rightarrow \Sq(X)
	$$
	using the natural transformations $\iota_h$ and $\iota_v$.
\end{construction}

\begin{proposition}\label{prop.vert-hor-sq}
	Let $X$ be a 2-fold Segal space.
	The functors $X_h \rightarrow \Sq(X)$ and $X_v \rightarrow \Sq(X)$ are adjunct to equivalences 
	$X \rightarrow \Hor(\Sq(X))$ and $X \rightarrow \Vert(\Sq(X))$.
\end{proposition}
\begin{proof}
	We will just show the statement for the horizontal case; the other case is handled similarly. We have to show that the map
	$$
	\map_{\Seg^2(\S)}([n;m \dotsc, m], X) \rightarrow \map_{\Seg^2(\S)}([n;m, \dotsc, m], \Hor(\Sq(X))) 
	$$
	is an equivalence for all $n,m \geq 0$. Note that this map can be identified with the map
	$$
	\map_{\Cat^2(\S)}([n;m \dotsc, m]_h, X_h) \rightarrow \map_{\Cat^2(\S)}([n;m, \dotsc, m]_h, \Sq(X)) 
	$$
	induced by $X_h \rightarrow \Sq(X)$. In light of \ref{rem.quotient-grid},  this is an equivalence if and only if the square 
	\[ 
	\begin{tikzcd}
		\map_{\Cat^2(\S)}([n;m \dotsc, m]_h, X_h) \arrow[r]\arrow[d] & \map_{\Cat^2(\S)}([n,m], \Sq(X)) \arrow[d] \\ 
		\map_{\Cat^2(\S)}([0,0], \Sq(X))^{\times(n+1)} \arrow[r] & \map_{\Cat^2(\S)}([0,m], \Sq(X))^{\times (n+1)}
	\end{tikzcd}
	\]
	is a pullback square. By construction of the map $X_h \rightarrow \Sq(X)$, this square 
	is obtained from the commutative square 
	\[
	\begin{tikzcd}
		\{0,\dotsc, n\} \otimes [m] \arrow[r]\arrow[d] & {[n]} \otimes [m] \arrow[d] \\
		\{0,\dotsc, n\} \arrow[r] & {[n;m,\dotsc, m]}
	\end{tikzcd}
	\]
	of gaunt 2-categories by applying $\map_{\Seg^2(\S)}(-,X)$. But this square is a pushout square in $\Seg^2(\S)$,
	as argued in \cite[Lemma 5.1.11]{HHLN}. Namely, to show that it is a pushout, one may reduce to $n,m \in \{0,1\}$ using the co-Segal condition. The only left-over non-trivial case 
	is $n = m =1$, where it follows from the fact that the canonical map $[1;1] \cup_{[1]^{\sqcup 2}} [2] \rightarrow [1] \otimes [1]$ is an equivalence --- already in $\PSh(\Delta^{\times 2})_{\mathrm{deg}}$. 
\end{proof}

\begin{corollary}\label{cor.sq-restrictions}
	The squares functor restricts to right adjoints 
	$\Sq : \Seg^2(\S)^{lc} \rightarrow \Cat^2(\S)^{lc}$ and $\Sq : \Cat_{(\infty,2)} \rightarrow \DblCat_\infty^c$ as well.
\end{corollary}
\begin{proof}
	It readily follows from \ref{prop.vert-hor-sq} that the squares functor restricts as such.
	The left adjoints are then given by the composites 
	$$
	\Seg^2(\S)^{lc} \rightarrow \Seg^2(\S) \xrightarrow{\mathrm{Gr}} \Cat^2(\S) \rightarrow \Cat^2(\S)^{lc},$$
	$$\Cat_{(\infty,2)} \rightarrow \Seg^2(\S) \xrightarrow{\mathrm{Gr}} \Cat^2(\S) \rightarrow \DblCat_\infty^c,$$
	where the functors that appear first and last in the composites are inclusions and localizations respectively.
\end{proof}

\section{Companionships and conjunctions, and homotopy coherence}\label{section.comp-conj}

Let us start by defining companionships and conjunctions in a double Segal space.

\begin{definition}\label{def.compconj}
	Let $f : x \rightarrow y$ be a vertical arrow in a double Segal space $\P$. A horizontal arrow 
	$F : x \rightarrow y$ in $\P$ is called the \textit{companion} of $f$  
	if there exist two 2-cells
	\[
		\eta = \begin{tikzcd}
			x\arrow[d, equal] \arrow[r, equal, ""name=f] & x\arrow[d, "f"] \\
			x \arrow[r, "F"'name=t] & y
			\arrow[from=f,to=t,Rightarrow, shorten <= 6pt, shorten >= 6pt]
		\end{tikzcd} 
		\quad 
		\text{and}
		\quad
		\epsilon = \begin{tikzcd}
			x \arrow[r, "F"name=f]\arrow[d, "f"' ] & y \arrow[d, equal] \\
			y \arrow[r, equal, ""'name=t] & y
			\arrow[from=f,to=t,Rightarrow, shorten <= 6pt, shorten >= 6pt]
		\end{tikzcd}
	\] 
	that satisfy the following two \textit{triangle identities}:
	\[
		\begin{tikzcd}
			x\arrow[d, equal] \arrow[r, equal, ""'name=h1] & x\arrow[d, "f"] \\
			x \arrow[r, ""name=h2]\arrow[d, "f"'] & y \arrow[d, equal] \\
			y \arrow[r, equal, ""'name=h3] & y
			\arrow[from=h1, to=h2, phantom, "\scriptstyle\eta"]
			\arrow[from=h2, to=h3, phantom, "\scriptstyle\epsilon"]
		\end{tikzcd}
		\simeq
		\begin{tikzcd}
			x\arrow[r,equal] \arrow[d, "f"'name=t1] & x \arrow[d, "f"name=f1] \\
			y \arrow[r, equal]  & y,
			\arrow[from=f1, to=t1, equal, shorten <= 14pt, shorten >= 14pt]
		\end{tikzcd}
		\quad 
		\begin{tikzcd}
			x \arrow[r, equal, ""name=h1]\arrow[d, equal] & x \arrow[d] \arrow[r, "F"name=h3] & y\arrow[d, equal] \\
			x \arrow[r, "F"'name=h2] & y \arrow[r, equal, ""'name=h4] & y
			\arrow[from=h1, to=h2, phantom, "\scriptstyle\eta"]
			\arrow[from=h3, to=h4, phantom, "\scriptstyle\epsilon"]
		\end{tikzcd}
		\simeq
		\begin{tikzcd}
			x\arrow[d,equal] \arrow[r, "F"name=f1] & y \arrow[d,equal] \\
			x \arrow[r, "F"'name=t1] & y.
			\arrow[from=f1, to=t1, equal, shorten <= 12pt, shorten >= 10pt]
		\end{tikzcd}
	\]
	In this case, $(f,F)$ is called a \textit{companionship}, $\eta$ is called the \textit{companionship unit}, and $\epsilon$ is called the 
	\textit{companionship counit}.

	Dually, a horizontal arrow $F':y\rightarrow x$ is called the \textit{conjoint} of $f$ 
	when there exist two 2-cells in $\P$
	\[
		\eta' = \begin{tikzcd}
			x\arrow[d, "f"' ] \arrow[r, equal, ""name=f] & x\arrow[d, equal] \\
			y \arrow[r, "F'"'name=t] & x
			\arrow[from=f,to=t,Rightarrow, shorten <= 6pt, shorten >= 6pt]
		\end{tikzcd} 
		\quad 
		\text{and}
		\quad
		\epsilon' = \begin{tikzcd}
			y \arrow[r, "F'"name=f]\arrow[d, equal] & x \arrow[d, "f"] \\
			y \arrow[r, equal, ""'name=t] & y
			\arrow[from=f,to=t,Rightarrow, shorten <= 6pt, shorten >= 6pt]
		\end{tikzcd}
	\] 
	that compose as follows:
	\[
		\begin{tikzcd}
			x\arrow[d, "f"'] \arrow[r, equal, ""'name=h1] & x\arrow[d, equal] \\
			y \arrow[d,equal ]\arrow[r, ""name=h2] & x \arrow[d, "f"] \\
			y \arrow[r, equal, ""'name=h3] & y
			\arrow[from=h1, to=h2, phantom, "\scriptstyle\eta'"]
			\arrow[from=h2, to=h3, phantom, "\scriptstyle\epsilon'"]
		\end{tikzcd}
		\simeq
		\begin{tikzcd}
			x\arrow[r,equal] \arrow[d, "f"'name=t1] & x \arrow[d, "f"name=f1] \\
			y \arrow[r, equal]  & y,
			\arrow[from=f1, to=t1, equal, shorten <= 14pt, shorten >= 14pt]
		\end{tikzcd}
		\quad 
		\begin{tikzcd}
			y \arrow[r, "F'"name=h1]\arrow[d, equal] & x \arrow[d] \arrow[r, equal, ""name=h3] & x\arrow[d, equal] \\
			y \arrow[r, equal, ""'name=h2] & y \arrow[r,"F'"'name=h4] & x
			\arrow[from=h1, to=h2, phantom, "\scriptstyle\eta'"]
			\arrow[from=h3, to=h4, phantom, "\scriptstyle\epsilon'"]
		\end{tikzcd}
		\simeq
		\begin{tikzcd}
			y\arrow[d,equal] \arrow[r, "F'"name=f1] & x \arrow[d,equal] \\
			y \arrow[r, "F'"'name=t1] & x.
			\arrow[from=f1, to=t1, equal, shorten <= 12pt, shorten >= 10pt]
		\end{tikzcd}
	\]
	In this case, $(f,F')$ is called a \textit{conjunction}, $\eta'$ is called the \textit{conjunction unit}, and $\epsilon'$ is called the 
	\textit{conjunction counit}.
	\end{definition}
	
	\begin{remark}
		These notions were also considered by Gaitsgory and Rozenblyum in \cite[Subsection 10.5.1]{GR}, albeit using different terminology.
	\end{remark}
	
	\begin{remark}\label{rem.comp-conj-dual}
		Suppose that $f : x \rightarrow y$ and $F : y \rightarrow x$ are a vertical and horizontal arrow 
		of a double Segal space $\P$ respectively. 
		Then one may easily deduce that the following assertions are equivalent:
		\begin{itemize}
			\item the pair $(f,F)$ forms a conjunction in $\P$,
			\item the pair $(f^\op, F)$ forms a companionship in $\P^\vop$,
			\item the pair $(f, F^\op)$ forms a companionship in $\P^\hop$.
		\end{itemize}
		Consequently, the theory of companionships is formally dual to the theory of conjunctions.
	\end{remark}
	
	\begin{remark}
		Suppose that $f : x \rightarrow y$ and $g : y \rightarrow z$ are vertical arrows of a double Segal space $\P$, with 
		companions $F$ and $G$ respectively. 
		Then it is instructive to verify $gf$ admits a companion as well, which is given by the composite $GF$. Dually, 
		if $F'$ and $G'$ are conjoints for $f$ and $g$ respectively, then $gf$ has a conjoint given by $F'G'$.
	\end{remark}

\begin{example}\label{ex.compconj-sq}
	We consider the double Segal space $\Sq(X)$ of squares in a 2-fold Segal space $X$ 
	that was defined in \ref{ssection.squares-def}.
	Let $f$ be an arrow of $X$, which we will view as a vertical arrow of $\Sq(X)$. Then one easily checks that the companion of $f$ in $\Sq(X)$ is given by $f$, now viewed as a horizontal arrow. 
	It is also instructive to verify that the conjoint
	 of $f$ in $\Sq(X)$ exists if and only if $f$ admits a right adjoint $g$, in which case 
	the conjoint of $f$ is given by $g$. In that case, the conjunction unit and counit for $(f,g)$ corresponds to the adjunction unit and counit for $(f,g)$.
\end{example}

\begin{example}
	Let $\C$ be an $\infty$-category with finite limits. 
	Then we may consider the double $\infty$-category $\SSPAN(\C)$ of spans in $\C$ that was constructed in \cite{HaugsengSpans}.
 If $f : x \rightarrow y$ is an arrow of $\C$
	that is viewed as a vertical arrow of $\SSPAN(\C)$, then it has both a companion and conjoint given by the spans 
	$$
	(\id_x, f) : x \rightarrow x \times y, \quad  (f, \id_x) : x \rightarrow y \times x.
	$$
\end{example}

\begin{example}
	In the double $\infty$-category $\CCAT_\infty$ of $\infty$-categories that was informally described in 
	\ref{ex.ccat-infty}, the companion and conjoint of a functor $f : \C \rightarrow \D$ are given by the profunctors
	$$
	\D^\op \times \C \rightarrow \S : (d,c) \mapsto \map_\D(d,f(c)), \quad \C^\op \times \D \rightarrow \S : (c,d) \mapsto \map_\D(f(c),d)
	$$
	respectively. This is shown in \cite{EquipI}.
\end{example}

As \ref{ex.compconj-sq} already suggests, the notion of adjunctions, companionships and conjunctions are closely related. The following observation attests to this as well.

\begin{proposition}\label{prop.conj-comp-adjunction}
	Let $f : x \rightarrow y$ be a vertical arrow in a double Segal space $\P$ with a companion $F : x \rightarrow y$ and a conjoint 
	$G:y\rightarrow x$.
	Then the pair $(F,G)$ forms an adjunction in the horizontal fragment $\Hor(\P)$.
\end{proposition}
\begin{proof}
	The proof is analogous to the proof of the strict case \cite[Proposition 5.3]{ShulmanFramedBicats}.
	The candidate unit $\eta$ for the adjunction is defined by the pasting 
	\[
		\eta = \begin{tikzcd}
			x\arrow[d, equal] \arrow[r, equal, ""name=f] & x\arrow[d, "f"]\arrow[r,equal,""name=f2] & x\arrow[d,equal] \\
			x \arrow[r, "F"'name=t] & y \arrow[r, "G"'name=t2] & x
			\arrow[from=f,to=t,Rightarrow, shorten <= 6pt, shorten >= 6pt]
			\arrow[from=f2,to=t2,Rightarrow, shorten <= 6pt, shorten >= 6pt]
		\end{tikzcd}
	\]
	of the companionship and conjunction unit. The candidate counit $\epsilon$ is given by the pasting 
	of the companionship and conjunction counit:
	\[
		\epsilon = \begin{tikzcd}
			y \arrow[r, "G"name=f]\arrow[d, equal] & x \arrow[r, "F"name=f2]\arrow[d, "f" ] & y \arrow[d, equal] \\
			y \arrow[r, equal,""'name=t] & y \arrow[r, equal,""'name=t2] & y.
			\arrow[from=f,to=t,Rightarrow, shorten <= 6pt, shorten >= 6pt]
			\arrow[from=f2,to=t2,Rightarrow, shorten <= 6pt, shorten >= 6pt]
		\end{tikzcd} 
	\]
	It now readily follows from
	the triangle identities for companionships and conjunctions that $\eta$ and $\epsilon$ satisfy the triangle identities for adjunctions.
\end{proof}

\subsection{Homotopy coherence}

We will now show that every companionship uniquely upgrades to a homotopy coherent one. 
To make this statement concrete, we will need to set up some notation.

\begin{definition}
	The \textit{free-living companionship} is the (discrete) double Segal space defined by
	$$
	\comp := \Sq([1]).
	$$
	The dual \textit{free-living conjunction} is given by
	$$
	\conj := \comp^\hop.
	$$
\end{definition}

\begin{definition}\label{def.lower-triangle}
	The \textit{free-living lower triangle}  
	is the double Segal space that can be pictured as
	\[
		L= \begin{tikzcd}[column sep =small, row sep = small]
		  0 \arrow[r,equal, ""name=f]\arrow[d,equal] & 0 \arrow[d] \\
		  0 \arrow[r, ""'name=t]& 1.
		  \arrow[from=f,to=t,Rightarrow, shorten <= 6pt, shorten >= 6pt]
		\end{tikzcd}
	\]
	It is defined by the pushout square 
	\[
		\begin{tikzcd}
			{[\{0,1\}, 0] }\cup_{[\{0\}, \{0\}]} [0, \{0,1\}] \arrow[r]\arrow[d] & {[1,1]} \arrow[d] \\
			{[0,0]} \arrow[r] & L
		\end{tikzcd}	
	\]
	of bisimplicial spaces.
\end{definition}

\begin{notation}\label{not.space-units}
	If $\P$ is a double Segal space, then we will write 
	$$C_\P \subset \map_{\PSh(\Delta^{\times 2})}(L, \P)$$
	for the subspace spanned by the companionship units.
\end{notation}

The remainder of this section is dedicated to proving the following two statements:

\begin{theorem}\label{thm.htpy-coh-comp-1}
	Let $\P$ be a double Segal space. Then restriction along 
	the canonical inclusion $L \rightarrow \comp$ 
	induces an equivalence 
	$$
	\map_{\Cat^2(\S)}(\comp, \P) \xrightarrow{\simeq} C_\P.
	$$
	Equivalently, if $\eta$ is a companionship unit in $\P$, then the fiber 
	$$
	\map_{\Cat^2(\S)}(\comp, \P)_\eta = \map_{\Cat^2(\S)}(\comp, \P) \times_{\map_{\Cat^2(\S)}(L, \P)} \{\eta\}
	$$
	is contractible.
\end{theorem}

\begin{theorem}\label{thm.htpy-coh-comp-2}
	Suppose that $\P$ is a double Segal space so that $\Hor(\P)$ is locally complete. Then the restriction arrow 
	$$
	\map_{\Cat^2(\S)}(\comp, \P) \rightarrow \map_{\Cat^2(\S)}([1]_v, \P)
	$$
	is a monomorphism with image given by the subspace of vertical arrows of $\P$ that admit a companion. 
\end{theorem}

\begin{remark}
	In the context of (strict) double categories, \ref{thm.htpy-coh-comp-2} can be recovered as a special 
	case of a theorem of  Grandis and Par\'e \cite[Theorem 1.8]{GrandisPare}.
\end{remark}
One may formally deduce the following results by applying the transpose $(-)^\tp$ and horizontal opposite $(-)^\hop$  
dualities.

\begin{corollary}\label{cor.htpy-coh-comp-3}
	Suppose that $\P$ is a double Segal space so that $\Vert(\P)$ is locally complete. Then the restriction arrow 
	$$
	\map_{\Cat^2(\S)}(\comp, \P) \rightarrow \map_{\Cat^2(\S)}([1]_h, \P)
	$$
	is a monomorphism with image given by the subspace of horizontal arrows of $\P$ that are a companion.
\end{corollary}

\begin{corollary}\label{cor.htpy-coh-conj}
	Suppose that $\P$ is a double Segal space. Then the following assertions holds:
	\begin{itemize}
		\item if $\Hor(\P)$ is locally complete, then the restriction map 
		$$
	\map_{\Cat^2(\S)}(\conj, \P) \rightarrow \map_{\Cat^2(\S)}([1]_v, \P)
	$$
		is a monomorphism with image given by the subspace of vertical arrows of $\P$ that admit a conjoint,
		\item if $\Vert(\P)$ is locally complete, then the restriction map 
		$$
	\map_{\Cat^2(\S)}(\conj, \P) \rightarrow \map_{\Cat^2(\S)}([1]_h, \P)
	$$
	is a monomorphism with image given by the subspace of horizontal arrows of $\P$ that are a conjoint.
	\end{itemize}
\end{corollary}

\begin{remark}
	In particular, note that every locally complete double Segal space $\P$ meets the conditions that are posed in \ref{thm.htpy-coh-comp-2}, 
	\ref{cor.htpy-coh-comp-3} and \ref{cor.htpy-coh-conj}.
\end{remark}

\begin{corollary}\label{cor.closure-limits-dss-with-companions}
	The full subcategory spanned by the double Segal spaces $\P$ for which:
	\begin{itemize}
		\item $\Hor(\P)$ is locally complete,
		\item every vertical arrow of $\P$ admits a companion,
	\end{itemize}
	is a reflective subcategory of $\Cat^2(\S)$, and particular, closed under limits.
\end{corollary}
\begin{proof}
	On account of \ref{thm.htpy-coh-comp-2}, this is precisely the full subcategory of double Segal spaces 
	local with respect to the maps $[1;J]_h \rightarrow [1]_h$ and $[1]_v \rightarrow \comp$.
\end{proof}

\begin{remark}\label{rem.free-adj}
	In \cite{SchanuelStreet}, Schanuel and Street define the \textit{free-living adjunction} (gaunt) 2-category $$\adj,$$ and show 
	that every adjunction in a 2-category can be upgraded to a functor out of this 2-category.  This result was then upgraded by Riehl--Verity to the $(\infty,2)$-categorical context \cite{RiehlVerityAdj},
	where 
	it is shown that every left adjoint arrow in an $(\infty,2)$-category uniquely (up to contractible choice) extends to a functor out of the free-living adjunction.
	
	One may combine 
	\ref{ex.compconj-sq} and \ref{cor.htpy-coh-conj}
	to conclude that the $(\infty,2)$-category 
	obtained by applying the left adjoint $\mathrm{Gr}(-)$ of the squares functor $\Sq : \Cat_{(\infty,2)} \rightarrow \DblCat_\infty^c$ 
	to the free-living conjunction, carries the same universal property of $\adj$ exhibited by Riehl--Verity. Consequently, we deduce 
	that there is an equivalence 
	$$
	\mathrm{Gr}(\conj) \simeq \adj.
	$$

	In fact, one can write down an explicit functor $\conj \rightarrow \Sq(\adj)$ using 
	the description of the free-living adjunction of Schanuel--Street. If one succeeds to show that the adjunct map 
	$\mathrm{Gr}(\conj) \rightarrow \adj$ is an equivalence --- without reasoning with universal properties --- 
	then \ref{cor.htpy-coh-conj} can be used to provide an independent proof of the main result of Riehl--Verity \cite{RiehlVerityAdj}.
\end{remark}

\subsection{A special extension theorem}\label{ssection.special-lifting} The main ingredient of the proof of \ref{thm.htpy-coh-comp-1} and 
\ref{thm.htpy-coh-comp-2}
is a special extension result for double Segal spaces with respect to companionship units. 
We first set up the necessary notation.

\begin{notation}
	We will consider the quotient $L[n,m]$ that is defined by the pushout square
	\[
        \begin{tikzcd}
             { [\{0\}, \{0 \leq 1\}]} \cup_{[\{0\}, \{0\}]} [\{0 \leq \dotsb \leq n\}, \{0\}] \arrow[r]\arrow[d] &  {[n,m]}.\arrow[d] \\
             {[0,0]} \arrow[r] & L[n,m],
        \end{tikzcd}
    \]
	of bisimplicial spaces, so that $L[1,1] = L$.
\end{notation}

\begin{notation}\label{not.special-lifting-notation}
	Let $k \geq 0$ and suppose that $S$ is a subset of $\{0,\dotsc, k\}$. Then we will use the notation of \cite[Subsection 2.2.1]{Joyal} and 
	consider the following simplicial subset
	$$
	\textstyle \Lambda^S[k] := \bigcup_{i \in \{0,\dotsc, k\}\setminus S}d_i[k-1] \subset [k].
	$$
	For $n, m \geq 0$ and subsets $S$ and $T$ of $[n]$ and $[m]$ respectively, we will consider the bisimplicial subset 
	$$
	\Lambda^{S,T}[n,m] := \Lambda^S[n]_h \times [m]_v \cup_{\Lambda^S[n]_h \times \Lambda^T[m]_v} [n]_h \times \Lambda^T[m]_v \subset [n,m].
	$$
	Whenever $S \subset \{0,\dotsc, n-1\}$ and $T \subset \{0, \dotsc, m-1\}$, we will consider the quotients
	\[
        \begin{tikzcd}
             { [\{0\}, \{0 \leq 1\}]} \cup_{[\{0\}, \{0\}]} [\{0 \leq \dotsb \leq n\}, \{0\}] \arrow[r]\arrow[d] &  \Lambda^{S,T}[n,m]\arrow[d] \\
             {[0,0]} \arrow[r] & \Lambda^{S,T}_L[n,m].
        \end{tikzcd}
    \]
	For $n\geq 1$ and $m \geq 2$, we have a map
	$$
	L^{\sqcup n} \cong \{1, \dotsc, n\} \times L \rightarrow \Lambda^{S,T}_L[n,m] 
	$$
	where the restriction to $\{i\} \times L$ factors through the $(1,1)$-bisimplex
	$$
	[\{0 \leq i\}, \{0 \leq 1\}] : \{i\} \times [1,1] \rightarrow \Lambda^{S,T}[n,m].
	$$
	For brevity, we will also make use of the notation 
	$$
	\Gamma_L^T[n,m] := \Lambda^{\emptyset, T}_L[n,m].
	$$
\end{notation}

Let $\P$ be a double Segal space, and 
suppose that we are given a map $\sigma : \Gamma^0_L[n,m] \rightarrow \P$ of bisimplicial spaces 
with $n \geq 1$ and $m\geq 2$. Moreover, we impose the condition that each restriction 
$
\sigma|(\{i\} \times L) : L \rightarrow \P 
$
classifies a companionship unit for $i =1, \dotsc, n$. 
The special extension theorem that we will prove, asserts that the dotted lift in the diagram 
\[
	\begin{tikzcd}
		\Gamma^0_L[n,m] \arrow[r,"\sigma"]\arrow[d] & \P \\
		L[n,m] \arrow[ur, dotted]
	\end{tikzcd}
\]
exists under these conditions, and that it is unique up to contractible choice. Here the left vertical arrow is the inclusion.

Before we state the theorem, let us sketch informally how one can produce such a lift for the minimal case that $n=1$ and $m=2$. In this case, 
$\sigma$ classifies a vertical arrow $k : y \rightarrow b$ and two 2-cells 
\[
	\eta = \begin{tikzcd}
		x\arrow[d, equal] \arrow[r, equal,""name=f] & x\arrow[d, "f"] \\
		x \arrow[r, "F"'name=t] & y,
		\arrow[from=f,to=t,Rightarrow, shorten <= 6pt, shorten >= 6pt]
	\end{tikzcd}
	\quad 
	\alpha = \begin{tikzcd}
		x\arrow[d, "h"'] \arrow[r, equal,""name=f] & x\arrow[d, "kf"] \\
		a \arrow[r, "G"'name=t] & b,
		\arrow[from=f,to=t,Rightarrow, shorten <= 6pt, shorten >= 6pt]
	\end{tikzcd}
\]
in $\P$. By assumption, $\eta$ is a companionship unit and has an associated 
counit $\epsilon$. The candidate lift $L[1,2] \rightarrow \P$ 
 classifies the diagram in $\P$ that is obtained from the diagram 
\[
        \begin{tikzcd}
            x \arrow[r,equal]\arrow[d,equal, ""'name=f1] & x \arrow[r,equal]\arrow[d,equal, ""name=t1] & x \arrow[d, "f"name=t5] \\
            x \arrow[r,equal]\arrow[d,equal] & x  \arrow[r,"F"] \arrow[d, "f"'name=f4] & y\arrow[d,equal, ""name=t4] \\ 
            |[alias=f3]| x \arrow[d, "h"'] & |[alias=t3]|y \arrow[r,equal] \arrow[d, "k"'name=f2] & y \arrow[d, "k"name=t2] \\
            a \arrow[r, "G"] & b\arrow[r,equal] &  b
			\arrow[from=f1,to=t1,equal, shorten <=15pt, shorten >= 15pt]
			\arrow[from=f2,to=t2,equal, shorten <=15pt, shorten >= 15pt]
			\arrow[from=f3,to=t3, phantom, "{\scriptstyle \alpha}"]
			\arrow[from=f4,to=t4, phantom, "{\scriptstyle \epsilon}"]
			\arrow[from=t1,to=t5, phantom, "{\scriptstyle \eta}"]
        \end{tikzcd}
 \]
by horizontally pasting the columns (and the vertical boundaries of the subdivided 2-cell in the resulting second row). 
One can check that this is indeed a lift by using the companionship triangle identities. The proof of the special extension theorem 
for the case that $n=1$ and $m=2$ is essentially an elaboration of this idea.

\begin{theorem}\label{thm.lifting-thm}
	Let $\P$ be a double Segal space.
    Let $n \geq 1$ and $m \geq 2$. Then the fibered map 
    \[
    \begin{tikzcd}[column sep =tiny]
    \map_{\PSh(\Delta^{\times 2})}(L[n,m], \P) \arrow[rr]\arrow[dr] && \map_{\PSh(\Delta^{\times 2})}(\Gamma^0_L[n,m], \P)\arrow[dl] \\
    & \map_{\PSh(\Delta^{\times 2})}(L, \P)^{\times n} 
    \end{tikzcd}
    \]
    becomes an equivalence when pulled back to the subspace $C_\P^{\times n}$ of $\map_{\PSh(\Delta^{\times 2})}(L, \P)^{\times n}$. In other words, 
	the map on fibers 
	$$
	\map_{\PSh(\Delta^{\times 2})}(L[n,m], \P)_\eta \rightarrow \map_{\PSh(\Delta^{\times 2})}(\Gamma^0_L[n,m], \P)_{\eta}
	$$
	is an equivalence for every tuple $\eta = (\eta_1, \dotsc, \eta_n) \in C_\P^{\times n}$.
\end{theorem}

\begin{remark}
	There are more special extension or lifting results of a similar flavor in higher category theory:
	\begin{itemize}
		\item \textit{Joyal's lifting theorem} \cite[Theorem 2.2]{JoyalQCatKan} for quasi-categories with respect to equivalences,
		\item a version of the preceding theorem for Segal spaces by Rezk \cite[Lemma 11.10]{RezkSeg},
		\item a result for quasi-categorically enriched categories by Riehl and Verity \cite[Proposition 4.3.6]{RiehlVerityAdj} with respect to adjunction counits.
	\end{itemize}
\end{remark}

To prove \ref{thm.lifting-thm}, we present the following reduction step:

\begin{lemma}\label{lem.reduction}
	Suppose that \ref{thm.lifting-thm} holds for $(n,m) = (1,2)$. Then it holds for all $n \geq 1$ and $m\geq 2$.
\end{lemma}

In turn, we need the following ingredient to make this reduction:

\begin{lemma}\label{lem.gen-horns}
	Let $n, m \geq 0$, and suppose that $S$ and $T$ are subsets of $[n]$ and $[m]$ respectively. Suppose that one of the complements 
	$[n] \setminus S$ or $[m] \setminus T$ is not convex. 
	Then the inclusion $\Lambda^{S,T}[n,m] \rightarrow [n,m]$ is contained in the saturation of (\textit{Seg}).
\end{lemma}
\begin{proof}
	Without loss of generality, let us assume $[n] \setminus S$ is not convex. It follows from \cite[Lemma 3.5]{JoyalTierney} and \cite[Proposition 2.12]{Joyal} that $\Lambda^S[n]_h \rightarrow [n]_h$ is contained in the saturation 
	of (\textit{Seg}). Thus \ref{prop.exp-ideals-dss} now implies that $\Lambda^S[n]_h \times [m]_v \rightarrow [n,m]$ and 
	$\Lambda^S[n]_h \times \Lambda^T[m]_v \rightarrow [n]_h \times \Lambda^T[m]_v$ are contained in (\textit{Seg}) as well. In turn,
	this implies that 
	that the total composite 
	$$
	\Lambda^S[n]_h \times [m]_v  \rightarrow \Lambda^{S,T}[n,m] \rightarrow [n,m]
	$$
	as well as the left inclusion in the composite are contained in (\textit{Seg}). So the desired result follows from 2-out-of-3.
\end{proof}

\begin{proof}[Proof of \ref{lem.reduction}]
	The proof is by induction.
    We  will demonstrate that if \ref{thm.lifting-thm} holds for a pair $(n,m)$ with $n\geq 1$, $m \geq 2$, then it will hold for the succesive pairs $(n+1,m)$ and $(n,m+1)$ as well. 
	
	Let us start by treating the first pair.  Note that we have inclusions
	$$
	\Lambda^{\{1\}, \{0\}}_L[n+1,m] \rightarrow \Gamma^0_L[n+1, m] \rightarrow L[n+1, m].
	$$
	The total composite is contained in the saturation of (\textit{Seg}) by \ref{lem.gen-horns}, as it is a pushout along $\Lambda^{\{1\}, \{0\}}[n+1,m] \rightarrow [n+1, m]$.
	By 2-out-of-3, it suffices to check that the map 
	$$
	\map(\Gamma^0_L[n,m], \P) \rightarrow \map(\Lambda^{\{1\},\{0\}}_L[n+1,m], \P)
	$$
	induces an equivalence on the fibers above tuples in $C_\P^{\times(n+1)}$.
	We now observe that we have a pushout square 
	\[
		\begin{tikzcd}
			\Gamma^0_L[n,m] \arrow[r]\arrow[d] & \Lambda^{\{1\}, \{0\}}_L[n+1,m] \arrow[d] \\
			{L[n,m]} \arrow[r, "(d_1)_h \times \id_v"] & \Gamma^0_L[n+1,m].
		\end{tikzcd}
	\] 
	Consequently, in the induced commutative diagram 
	\[
		\begin{tikzcd}[column sep = small]
		 \map(\Gamma^{0}_L[n+1,m], \P)\arrow[d] \arrow[r] &  \map(\Lambda^{\{1\},\{0\}}_L[n+1,m], \P) \arrow[r]\arrow[d] & \map(L, \P)^{\times (n+1)} \arrow[d] \\
		 \map(L[n,m], \P) \arrow[r] &  \map(\Gamma^0_L[n,m], \P) \arrow[r] & \map(L, \P)^{\times n}, \\ 
		\end{tikzcd}
	\]
	 the left square must be a pullback square. Suppose that we have a tuple $\eta \in C_\P^{\times (n+1)}$. 
	Then the above diagram specializes to a pullback square
	\[
		\begin{tikzcd}[column sep = small]
		 \map(\Gamma^{0}_L[n+1,m], \P)_{\eta}\arrow[d] \arrow[r] &  \map(\Lambda^{\{1\},\{0\}}_L[n+1,m], \P)_{\eta}\arrow[d] \\
		 \map(L[n,m], \P)_{\hat{\eta}} \arrow[r] &  \map(\Gamma^0_L[n,m], \P)_{\hat{\eta}},
		\end{tikzcd}
	\]
	where $\hat{\eta}$ is given by the tuple $(\eta_2, \eta_3, \dotsc, \eta_{n+1})$. 
	So the desired conclusion follows from the induction hypothesis.

	Secondly, we will handle the pair $(n,m+1)$. The argument is similar. In light of 
	\ref{lem.gen-horns},  the total composite
    $$
    \Lambda^{\emptyset, \{0,2\}}_L[n,m+1] = \Gamma^{\{0,2\}}_L[n,m+1] \rightarrow \Gamma^0_L[n,m+1] \rightarrow L[n,m+1]
    $$
    is contained in the saturation of (\textit{Seg}) since it can be written as a pushout along the inclusion $\Lambda^{\emptyset, \{0,2\}}[n,m+1] \rightarrow [n,m+1]$. 
	By 2-out-of-3, the theorem holds for the pair $(n, m+1)$ whenever the map
     \[
     \map(\Gamma^{0}_L[n,m+1], \P) \rightarrow \map(\Gamma^{\{0,2\}}_L[n,m+1], \P) \\
    \]
    is an equivalence on the fibers above tuples in $C_\P^{\times n}$. But now we note that we have a pushout 
	square \[
		\begin{tikzcd}
			\Gamma^0_L[n,m] \arrow[r]\arrow[d] & \Gamma^{\{0,2\}}_L[n,m+1] \arrow[d] \\
			{L[n,m]} \arrow[r, "\id_h \times (d_2)_v"] & \Gamma^0_L[n,m+1].
		\end{tikzcd}
	\] Thus the desired conclusion may be inferred from the induction hypothesis again.
\end{proof}

\begin{proof}[Proof of \ref{thm.lifting-thm}]
	We may reduce to the case that $(n,m) = (1,2)$ in light of \ref{lem.reduction}. We will consider the sub double Segal spaces 
	\[
		R:=\begin{tikzcd}[column sep = tiny, row sep = tiny]
			(0,0) \arrow[r]\arrow[d] & (1,0) \arrow[d] \\
			(0,1) & (1,1)
		\end{tikzcd}
		\quad \subset \quad 
		\vrectangle :=\begin{tikzcd}[column sep = tiny, row sep = tiny]
			(0,0) \arrow[r, ""name=f]\arrow[d] & (1,0) \arrow[d] \\
			(0,1) \arrow[d] & (1,1) \arrow[d] \\
			(0,2)\arrow[r, ""'name=t] & (1,2)
			\arrow[from=f,to=t,Rightarrow, shorten <= 6pt, shorten >= 6pt]
		\end{tikzcd}
		\quad \subset \quad [1,2].
	\]
	Suppose that $\eta : [1,1] \rightarrow \P$ is a 2-cell of $\P$ that classifies a unit of a companionship. Then we have to show that the map 
	$$
	\map([1,2], \P) \times_{\map([1,1], \P)} \{\eta\} \rightarrow \map(\Gamma^0[1,2], \P) \times_{\map([1,1], \P)} \{\eta\}
	$$
	is an equivalence. In turn, this translates to the condition for the map
	$$
	\map([1,2], \P) \times_{\map([1,1], \P)} \{\eta\} \rightarrow \map(\vrectangle, \P) \times_{\map(R, \P)} \{\eta|R\}
	$$
	to be an equivalence. 

	Let us write $T$ for the smallest bisimplicial subset of $\comp$ that contains the bisimplices 
	\[
		\begin{tikzcd}[column sep =tiny, row sep = tiny]
			0 \arrow[r,equal]\arrow[d,equal] & 0 \arrow[r,equal] \arrow[d] & 1 \arrow[d,equal] \\
			0 \arrow[r] & 1 \arrow[r,equal]  & 1,
		\end{tikzcd}
	   \quad 
		\begin{tikzcd}[column sep =tiny, row sep = tiny]
			 0 \arrow[r,equal] \arrow[d,equal] & 0 \arrow[d] \\
			0 \arrow[d] \arrow[r] & 1\arrow[d,equal] \\
			 1 \arrow[r,equal] & 1.
		\end{tikzcd}
	\]
	The companionship data 
	associated to the unit $\eta$ witness the existence of an extension 
	$\eta' : T \rightarrow \P$
	so that the composite $[1,1] \rightarrow L \rightarrow T$ recovers $\eta$.
	It is sufficient to demonstrate that $\P$ is local with respect to the inclusion of bisimplicial spaces
	$$
	\vrectangle \cup_R T \rightarrow [1,2] \cup_{[1,1]} T,
	$$
	because we have a commutative square
	\[
	\begin{tikzcd}
		\map([1,2] \cup_{[1,1]} T, \P) \times_{\map(T,\P)} \{\eta'\} \arrow[d]\arrow[r] & \map([1,2], \P) \times_{\map([1,1], \P)} \{\eta\} \arrow[d] \\ 
		\map(\vrectangle \cup_R T, \P) \times_{\map(T,\P)} \{\eta'\} \arrow[r] & \map(\vrectangle,\P)\times_{\map(R,\P)} \{\eta|R\},
	\end{tikzcd}
	\]
	where the horizontal arrows are equivalences by the pasting lemma for pullback squares.
	Let $L : \PSh(\Delta^{\times 2}) \rightarrow \Cat^2(\S)$ denote the reflector. 
	Then we equivalently have to show that the induced map 
	$$
	i : \vrectangle' := \vrectangle \cup_R LT \simeq L(\vrectangle \cup_R T) \rightarrow L([1,2] \cup_{[1,1]} T) \simeq [1,2] \cup_{[1,1]} LT =: [1,2]'
	$$
	constitutes an equivalence in the subcategory $\Cat^2(\S)$, where both pushouts are now computed in this subcategory.
	
	Throughout the remainder of the proof, we will exploit the language of pasting shapes that was developed in \cite{Pasting}. 
	We will make use of the \textit{composable 2-dimensional subshape} 
	(see Definitions 3.2 and 3.28 of \cite{Pasting}) 
	defined by
	\[
        \Sigma := \begin{tikzcd}[column sep = tiny, row sep = tiny]
            (0,0) \arrow[r, ""name=a1]\arrow[d] & (1,0)  \arrow[r, ""name=b1]\arrow[d] & (2,0) \arrow[d] \\
            (0,1) \arrow[r, ""name=a2]\arrow[d] & (1,1) \arrow[r, ""name=b2] \arrow[d] & (2,1) \arrow[d] \\ 
            (0,2)\arrow[d] & (1,2)\arrow[r,""name=b3] \arrow[d] & (2,2) \arrow[d] \\
            (0,3)\arrow[r, ""name=a3] & (1,3) \arrow[r, ""name=b4] &  (2,3).
			\arrow[from=a1, to=a2, "{\scriptstyle A_1}", phantom]
			\arrow[from=a2, to=a3, "{\scriptstyle A_{23}}", phantom]
			\arrow[from=b1, to=b2, "{\scriptstyle A_4}", phantom]
			\arrow[from=b2, to=b3, "{\scriptstyle A_5}", phantom]
			\arrow[from=b3, to=b4, "{\scriptstyle A_6}", phantom]
        \end{tikzcd}
		\subset 
		[2,3] = \begin{tikzcd}[column sep = tiny, row sep = tiny]
            (0,0) \arrow[r, ""name=a1]\arrow[d] & (1,0)  \arrow[r, ""name=b1]\arrow[d] & (2,0) \arrow[d] \\
            (0,1) \arrow[r, ""name=a2]\arrow[d] & (1,1) \arrow[r, ""name=b2] \arrow[d] & (2,1) \arrow[d] \\ 
            (0,2)\arrow[d]\arrow[r, ""name=a3] & (1,2)\arrow[r,""name=b3] \arrow[d] & (2,2) \arrow[d] \\
            (0,3)\arrow[r, ""name=a4] & (1,3) \arrow[r, ""name=b4] &  (2,3).
			\arrow[from=a1, to=a2, "{\scriptstyle A_1}", phantom]
			\arrow[from=a2, to=a3, "{\scriptstyle A_{2}}", phantom]
			\arrow[from=a3, to=a4, "{\scriptstyle A_{3}}", phantom]
			\arrow[from=b1, to=b2, "{\scriptstyle A_4}", phantom]
			\arrow[from=b2, to=b3, "{\scriptstyle A_5}", phantom]
			\arrow[from=b3, to=b4, "{\scriptstyle A_6}", phantom]
        \end{tikzcd}
	\]
	In the pictures above, we have also labeled the \textit{vertebrae} (see \cite[Definition 3.5]{Pasting}) of the shapes. 
	We will write $[-]$ for the nerve functor of \cite[Construction 3.11]{Pasting} for 
	2-dimensional pasting shapes. 
	On account of \cite[Theorem 3.49]{Pasting}, the canonical inclusions
	\begin{gather*}
	D_0 := [A_1 \cup A_4] \cup [A_4 \cup A_5]  \cup [A_{23}\cup A_6] \rightarrow [\Sigma], \\
	D_1 := [A_1 \cup A_4] \cup [A_2 \cup A_4 \cup A_5] \cup [A_2 \cup A_3 \cup A_6]  \rightarrow [2,3] ,
	\end{gather*}
	are contained in the saturation of (\textit{Seg}). We can now write down a map of bisimplicial sets
	$
	g' : D_1 \rightarrow [1,2] \cup_{[1,1]} T 
	$
	so that \begin{itemize}
		\item $g'|[A_1 \cup A_4]$ selects the degenerate $(2,1)$-bisimplex $$[2,1] \xrightarrow {(s_0)_h \times \id_v} [1,1] \rightarrow  [1,2] \cup_{[1,1]} T,$$
		\item $g'|[A_2 \cup A_4 \cup A_5]$ is the canonical map that surjects onto the image of $T$ in $[1,2] \cup_{[1,1]} T$,
		\item $g'|[A_2 \cup A_3 \cup A_6]$ is the canonical map that surjects onto the image $[1,2]$ in $[1,2] \cup_{[1,1]} T$.
	\end{itemize}
	Note that $g'$ restricts to a map $f'$ that fits in the commutative diagram
	\[
		\begin{tikzcd}[column sep = tiny, row sep = tiny]
			D_0 \arrow[dr]\arrow[rr, "\hspace{8pt}f'"]\arrow[dd] && \vrectangle \cup_R T \arrow[dr]\arrow[dd] \\
			& |[alias=f]|{[\Sigma]}\arrow[rr, crossing over, dashed, "\hspace{-20pt}f"] && \vrectangle' \arrow[dd, "i"]\\ 
			D_1\arrow[dr]\arrow[rr, "\hspace{20pt}g'"] && {[1,2] \cup_{[1,1]} T}\arrow[dr] \\
			& |[alias=t]|{[2,3]} \arrow[rr, dashed, "g"] &&{[1,2]'}.
			\arrow[from=f, to=t, crossing over]
		\end{tikzcd}
	\]
	Here, the slanted arrows are localizations with respect to the reflector $L$. The dashed arrows in the cube hence exist and are unique (up to contractible choice). 
	The resulting map $g$ selects the diagram in $[1,2]'$ that can be pictured as 
	\[
        \begin{tikzcd}[column sep = tiny, row sep = tiny]
            (0,0) \arrow[r,equal]\arrow[d,equal, ""'name=f1] & (0,0) \arrow[r,equal,""name=b1]\arrow[d,equal, ""name=t1] & (0,0) \arrow[d] \\
            (0,0) \arrow[r,equal,""name=a1]\arrow[d,equal] & (0,0)  \arrow[r, ""name=b2] \arrow[d] & (0,1) \arrow[d,equal] \\ 
            (0,0) \arrow[d]\arrow[r, ""name=a2] & (0,1)\arrow[r,equal, ""name=b3] \arrow[d, ""'name=f2] & (0,1) \arrow[d, ""name=t2] \\
            (2,0) \arrow[r, ""name=a3] & (2,1)\arrow[r,equal] &  (2,1).
			\arrow[from=f1,to=t1,equal, shorten <=15pt, shorten >= 15pt]
			\arrow[from=f2,to=t2,equal, shorten <=15pt, shorten >= 15pt]
			\arrow[from=a1,to=a2,Rightarrow, shorten <= 6pt]
			\arrow[from=a2,to=a3,Rightarrow, shorten <= 6pt]
			\arrow[from=b1,to=b2,Rightarrow, shorten <= 6pt]
			\arrow[from=b2, to=b3,Rightarrow, shorten <= 6pt]
        \end{tikzcd}
    \]
	Note that there is a canonical map $[1,2] \rightarrow [\Sigma]$ that factors the inclusion of the outer 2-cell in $[\Sigma]$.
	By construction, the restriction 
	$$
	[1,1] \xrightarrow{\id_h \times (d_2)_v} [1,2] \rightarrow [\Sigma] \xrightarrow{f} \vrectangle' 
	$$
	is given by the map $[1,1] \rightarrow [A_1 \cup A_4] \rightarrow D_0 \rightarrow T \rightarrow \vrectangle'$. So, the universal property of the pushout 
	gives rise to a unique map 
	$$
	r : [1,2]' \rightarrow \vrectangle' 
	$$
	so that $r|[1,2] = f|[1,2]$. Our next goal is now to show that $r$ is an inverse to $i$.

	We will first show that $ri \simeq \id_{\vrectangle'}$. This precisely entails demonstrating that the restriction
	$$
	\vrectangle \rightarrow [1,2] \rightarrow [\Sigma] \xrightarrow{f} \vrectangle' 
	$$
	is equivalent to the canonical inclusion $\vrectangle \rightarrow \vrectangle'$. To this end, we note that 
	there is a subshape 
	\[
        \Sigma' := \begin{tikzcd}[column sep = tiny, row sep = tiny]
            (0,0) \arrow[r, ""name=a1]\arrow[d] & (1,0)  \arrow[r, ""name=b1]\arrow[d] & (2,0) \arrow[dd] \\
            (0,1) \arrow[r, ""name=a2]\arrow[d] & (1,1)  \arrow[d] &  \\ 
            (0,2)\arrow[d] & (1,2)\arrow[r,""name=b2] \arrow[d] & (2,2) \arrow[d] \\
            (0,3)\arrow[r, ""name=a3] & (1,3) \arrow[r, ""name=b3] &  (2,3)
			\arrow[from=a1,to=a2,Rightarrow, shorten <= 6pt]
			\arrow[from=a2,to=a3,Rightarrow, shorten <= 6pt]
			\arrow[from=b1,to=b2,Rightarrow, shorten <= 6pt]
			\arrow[from=b2, to=b3,Rightarrow, shorten <= 6pt]
        \end{tikzcd}
		\quad \subset\quad \Sigma.
	\]
	The composite $\vrectangle \rightarrow [1,2] \rightarrow [\Sigma]$ factors through the inclusion 
	$[\Sigma'] \subset [\Sigma]$, yielding a map 
	$\vrectangle \rightarrow [\Sigma']$ that admits a retraction $\rho : [\Sigma'] \rightarrow \vrectangle$ corresponding to the diagram 
	$$
	\begin{tikzcd}[column sep = tiny, row sep = tiny]
		(0,0) \arrow[r, ""name=f1]\arrow[d,equal] & (1,0)\arrow[d,equal] \arrow[r,equal] & (1,0) \arrow[dd] \\
		(0,0) \arrow[r, ""'name=t1]\arrow[d] & |[alias=f2]|(1,0) \arrow[d] & |[alias=t2]| \\
		(0,1)\arrow[d] & (1,1) \arrow[r,equal]\arrow[d, ""'name=f3] & (1,1) \arrow[d, ""name=t3] \\ 
		(0,2)\arrow[r, ""name=t4] & (1,2) \arrow[r, equal] & (1,2).
		\arrow[from=f1,to=t1,equal, shorten <=15pt, shorten >= 15pt]
		\arrow[from=f3,to=t3,equal, shorten <=15pt, shorten >= 15pt]
		\arrow[from=f2,to=t2,equal, shorten <=8pt, shorten >= 17pt]
		\arrow[from=t1,to=t4,Rightarrow]
	\end{tikzcd}
	$$
	Note that the \textit{spine} $\Sp[\Sigma']$ of 
	$\Sigma'$ (see \cite[Definition 3.51]{Pasting}) is contained in $D_0$, and moreover, the restriction $f'|\Sp[\Sigma']$ factors as the composite 
	\[
			\Sp[\Sigma'] \xrightarrow{\rho} \vrectangle \rightarrow \vrectangle \cup_R T.
		\]
	On account of \cite[Corollary 3.52]{Pasting}, the map $\Sp[\Sigma'] \rightarrow [\Sigma']$ is contained the 
	saturation of (\textit{Seg}).
	Hence, it must follow that the restriction $f|[\Sigma']$ factors as the composite 
	\[
			[\Sigma'] \xrightarrow{\rho} \vrectangle \rightarrow \vrectangle'.
	\]
	The desired conclusion follows from this observation. 

	Finally, we have to show that $ir \simeq \id_{[1,2]'}$. We must demonstrate that the composite 
	$$
	[1,2] \xrightarrow{\id_h \times (d_2)_v} [1,3] \xrightarrow{(d_1)_h \times \id_v} [2,3] \xrightarrow{g} [1,2]'
	$$
	is equivalent to the canonical map $[1,2] \rightarrow [1,2]'$. To this end, we note that the image 
	of the spine $\Sp[1,3]$ under the face map ${(d_1)_h \times \id_v}$ is contained in ${[A_1 \cup A_4]} \cup [A_2 \cup A_5] \cup [A_3 \cup A_6] \subset D_1$.
	Moreover, one check that the following square 
	\[	
		\begin{tikzcd}
			\Sp[1,3] \arrow[r]\arrow[d, "\id_h \times (s_1)_v"'] & {D_1} \arrow[d, "g'"] \\
			{[1,2]} \arrow[r] &{ [1,2]} \cup_{[1,1]} T.
		\end{tikzcd}
	\]
	of bisimplicial sets commutes.
	Thus the restriction $g \circ (d_1)_h \times \id_v$ will factor as 
	\[
		[1,3] \xrightarrow{\id_h \times (s_1)_v} [1,2] \rightarrow [1,2]'.
	\]
	This concludes the proof, since the left arrow is a retraction for $\id_h \times (d_2)_v$.
\end{proof}

\subsection{Proofs of the homotopy coherence results} Now that we have 
established \ref{thm.lifting-thm}, we shift our focus to the proof of \ref{thm.htpy-coh-comp-1}.
We will construct a filtration for $\comp$. Note that the $(n,m)$-bisimplices of $\comp$ are given by the \textit{set} of functors $[n] \otimes [m] \rightarrow [1]$ 
of gaunt 2-categories. Since $[1]$ is a poset, these are described by maps of posets $[n] \times [m] \rightarrow [1]$.

\begin{construction}\label{con.filtration-comp}
	Let $\sigma_n$ be the $(n,n)$-bisimplex of $\comp$ that is given by the map of posets 
	$[n] \times [n] \rightarrow [1]$ so that $\sigma_n(i,j) = 0$ if $j \leq n-i$ and $\sigma_n(i,j) = 1$ otherwise.
	For instance, $\sigma_2$ is described by the picture 
	\[
		\begin{tikzcd}[column sep =tiny, row sep = tiny]
			0 \arrow[r,equal]\arrow[d,equal] & 0 \arrow[r,equal] \arrow[d,equal] & 0 \arrow[d] \\
			0 \arrow[d,equal] \arrow[r,equal] &0 \arrow[d] \arrow[r] & 1\arrow[d,equal] \\
			0 \arrow[r] & 1 \arrow[r,equal] & 1.
		\end{tikzcd}
	\]
	We will write $S_n \subset \comp$ for the smallest bisimplicial subset that contains $\sigma_n$. This yields a sequence of bisimplicial subsets 
	$$
	\textstyle L = S_1 \subset S_2 \subset \dotsb \subset S_n \subset S_{n+1} \subset \dotsb \subset \comp.
	$$ 
\end{construction}

\begin{lemma}\label{lem.filtration-comp}
	The inclusions $S_i \subset \comp$ exhibit that
	$
	\textstyle \colim_{ i \in \mathbb{N}} S_i = \comp.
	$
\end{lemma}
\begin{proof}
	One may deduce that the canonical map $\colim_{ i \in \mathbb{N}} S_i  \rightarrow \bigcup_{i \in \N}S_i$ is an equivalence from similar considerations as in \cite[Remark 3.50]{Pasting}. 
	Now one may easily conclude that the union is given by $\comp$ by verifying that any non-degenerate bisimplex 
	of $\comp$ can be written as an iterated face of an appropriate choice of bisimplex $\sigma_k$.
\end{proof}

\begin{proof}[Proof of \ref{thm.htpy-coh-comp-1}]
	We will show the second equivalent assertion for a fixed companionship unit $\eta \in \P$.
	In light of \ref{lem.filtration-comp}, it will be sufficient to show that the map
	$
	\map(S_{n+1}, \P)_\eta \rightarrow \map(S_n, \P)_\eta 
	$
	is an equivalence for every $n \geq 1$. Let $T \subset S_{n+1}$ be the bisimplicial subset generated by the face $(d_0, \id)^*\sigma_{n+1}$.
    Then we have a pushout squares
    \[
    \begin{tikzcd}
        \Gamma^0_L[n-1,n] \arrow[r]\arrow[d] & S_n\arrow[d]  \\ 
        {L[n-1, n]} \arrow[r, "{(d_0, \id)^*\sigma_{n+1}}"] & T,
    \end{tikzcd}
   \quad 
    \begin{tikzcd}
        \Gamma^0_L[n,n] \arrow[r]\arrow[d] & T \arrow[d]  \\ 
        L[n, n] \arrow[r, "{\sigma_{n+1}}"] & S_{n+1},
    \end{tikzcd}
    \]
	so that the desired result can be obtained from \ref{thm.lifting-thm}.
\end{proof}

\begin{remark}
	From the filtration of \ref{con.filtration-comp} one can read off what data equivalently determines (and is determined by) a companionship 
	unit and its associated 
	higher coherences that are encoded in $\comp$. This yields observations that are very similar to the situation for adjunctions, cf.\ \cite[Tag 02F4]{Kerodon} and \cite{RiehlVerityAdj}.
\end{remark}

To prove \ref{thm.htpy-coh-comp-2}, we need the following observation:

\begin{lemma}\label{lem.two-companionship-units}
	Suppose that we have 2-cells 
	\[
		\eta = \begin{tikzcd}
			x \arrow[d,equal]\arrow[r, equal,""name=f] & x \arrow[d, "f"] \\ 
			x \arrow[r, "F"'name=t] & y, 
			\arrow[from=f,to=t,Rightarrow, shorten <= 6pt, shorten >= 6pt]
		\end{tikzcd}
		\quad 
		\alpha = \begin{tikzcd}
			x\arrow[d,equal] \arrow[r, "F"name=f] & y \arrow[d,equal] \\ 
			x \arrow[r, "G"'name=t] & y, 
			\arrow[from=f,to=t,Rightarrow, shorten <= 6pt, shorten >= 6pt]
		\end{tikzcd}
	\]
	 in a double Segal space $\P$. If $\eta$ is a companionship unit, then the following assertions are equivalent:
	\begin{enumerate}
		\item the vertical composite $\eta'$ of $\eta$ and $\alpha$ is also a companionship unit,
		\item the horizontal 2-cell $\alpha$ admits an inverse $\beta$ in $\Hor(\P)$.
	\end{enumerate}
\end{lemma}
\begin{proof}
	Suppose that $\epsilon$ is a companionship counit for $\eta$. 
	If (1) holds, then one readily checks that the horizontal composite of $\eta$ and the companionship counit for $\eta'$ is 
	the desired inverse $\beta$ for $\alpha$.
	Conversely, if (2) holds, then one readily verifies that the vertical composite 
	of $\beta$ and $\epsilon$ is a companionship counit for $\eta'$.
\end{proof}

\begin{proof}[Proof of \ref{thm.htpy-coh-comp-2}]
	On account of \ref{thm.htpy-coh-comp-1}, it suffices to show that the composite $$
	C_\P \rightarrow \map(L, \P) \rightarrow \map([1]_v, \P)
	$$ is a monomorphism.\footnote{A similar claim appears in the work of Gaitsgory--Rozenblyum as well \cite[Lemma 5.1.5]{GR}, but without proof.}
	Equivalently, this entails demonstrating that the restriction along the fold map 
	$
	L \cup_{[1]_v} L \rightarrow L
	$
	induces an equivalence
	$$
	C_\P \rightarrow \map(L \cup_{[1]_v} L, \P) \times_{\map(L,\P)^{\times 2}} C_\P^{\times 2}.
	$$

	Let $A \subset \Sq([1;1])$ be the bisimplicial subset that is generated by the $(1,2)$-bisimplex that is pictured as 
	\[
		\begin{tikzcd}[column sep = small, row sep = small]
			0 \arrow[r,equal]\arrow[d,equal] & 0 \arrow[d, "0"] \\
			0 \arrow[r, "0"name=f]\arrow[d,equal] & 1 \arrow[d,equal] \\
			0 \arrow[r,"1"name=t] & 1.
		\end{tikzcd}
	\]
	There are two inclusions $\ell_1 : L \rightarrow A$ and $\ell_2 : L \rightarrow A$ that select 
	the top triangle and composite triangle in $A$ respectively.
	These give rise to a factorization 
	$
	L \cup_{[1]_v} L \rightarrow A \rightarrow L
	$
	of the fold map.
	Note that the first map is a pushout along $\Gamma^0_L[1,2] \rightarrow L[1,2]$. It follows from 
	\ref{thm.lifting-thm} that the inclusion $L \cup_{[1]_v} L \rightarrow A$ 
	induces to an equivalence
	$$
	\map(A, \P) \times_{\map(L,\P)}  C_\P  \xrightarrow{\simeq}  \map(L \cup_{[1]_v} L, \P) \times_{\map(L,\P)} C_\P 
	$$
 	after base changing via $\ell_1$. We may now base change via $\ell_2$ as well, 
	so that we obtain an equivalence 
	$$
	D:= \map(A, \P) \times_{\map(L,\P)^{\times 2}}  C_\P^{\times 2} \xrightarrow{\simeq} \map(L \cup_{[1]_v} L, \P) \times_{\map(L,\P)^{\times 2}} C_\P^{\times 2}.
	$$
	Since monomorphisms are closed under base change, $D$ maybe described as the subspace 
	of $\map(A, \P)$ spanned by the maps $f : A \rightarrow \P$ so that the restrictions 
	$f\ell_1$ and $f\ell_2$ are companionship units. On account of \ref{lem.two-companionship-units}, these two conditions for $f$ 
	are equivalent to conditions that $f\ell_1$ is a companionship unit and the restriction $f|[1;1]_h$ is an invertible 2-cell.
	
	We will now use the fact that the restriction map $$\map([1;J]_h, \P) \rightarrow \map([1;1]_h, \P)$$ is a monomorphism 
	onto the horizontal 2-cells that admit an inverse. This can be deduced from applying \cite[Theorem 6.2]{RezkSeg} to the restricted Segal space $\Hor(\P)_{1,\bullet}$. 
	Combining this with the above description of $D$, it follows that the canonical map
	$
	A \rightarrow A \cup_{[1;1]_h} [1;J]_h 
	$
	induces an equivalence 
	$$
	\map( A \cup_{[1;1]_h} [1;J]_h, \P) \times_{\map(L, \P)} C_\P \xrightarrow{\simeq} D,
	$$
	where the pullback in the domain is formed via $\ell_1$.
	Now, we note that we an obvious commutative diagram 
	\[
		\begin{tikzcd}
			& L \cup_{[1]_h} [1;1]_h \arrow[r]\arrow[d] & L \cup_{[1]_h} [1;J]_h \arrow[d]\arrow[dr] \\
			L \cup_{[1]_v} L \arrow[r]& A \arrow[r] & A \cup_{[1;1]_h} [1;J]_h \arrow[r] & L
		\end{tikzcd}
	\]
	where the bottom arrows compose to the fold map.
	The square that appears in the diagram is a pushout square, and the left vertical arrow is contained in the saturation of (\textit{Seg}). Hence, the right vertical arrow is contained in the saturation of (\textit{Seg}) as well.
	The slanted arrow is contained in the saturation of (\textit{lhc}). Consequently, the map 
	$C_\P \rightarrow \map(L \cup_{[1]_v} L, \P) \times_{\map(L,\P)^{\times 2}} C_\P^{\times 2}$ decomposes into equivalences
	\begin{align*}
	C_\P &\xrightarrow{\simeq } \map( A \cup_{[1;1]_h} [1;J]_h, \P) \times_{\map(L, \P)} C_\P \\&\xrightarrow{\simeq} D\\ &\xrightarrow{\simeq}  \map(L \cup_{[1]_v} L, \P) \times_{\map(L,\P)^{\times 2}} C_\P^{\times 2}. \qedhere
	\end{align*}
\end{proof}

\begin{remark}
	Suppose that $f$ is a vertical arrow in a double Segal space $\P$ (without any completeness assumptions)
	that admits a companion $F$. Then the proof of \ref{thm.htpy-coh-comp-2} shows that the fiber of the map 
	$$
	C_\P \rightarrow \map_{\Cat^2(\S)}([1]_v, \P)
	$$
	above $f$ is equivalent to the space of automorphisms
	$$
	\map_{\Cat^2(\S)}([1;J]_h, \P) \times_{\map_{\Cat^2(\S)}([1]_h, \P)} \{F\}.
	$$
\end{remark}

\section{Companions and conjoints in functor double Segal spaces}\label{section.comp-conj-fundbl}

Our next goal is to identity the {companions} and {conjoints} in functor double Segal spaces 
using the machinery of \ref{section.comp-conj}. 

\subsection{Companionable and conjointable 2-cells}
To 
give the characterization of the companions and conjoints in functor double Segal spaces, we introduce the following double categorical variant on 
adjointable squares:

\begin{definition}\label{def.companionable}
	Let $\P$ be a locally complete double Segal space.
	A 2-cell 
	\[
		\alpha = \begin{tikzcd}
			a \arrow[r, "F"name=f] \arrow[d,"h"'] & b \arrow[d,"k"] \\
			c \arrow[r, "G"'name=t] & d
			\arrow[from=f,to=t,Rightarrow, shorten <= 6pt, shorten >= 6pt]
		\end{tikzcd}
	\] in $\P$ 
	is said to be \textit{companionable} if $F$ and $G$ are the companions 
	of arrows $f : a\rightarrow b$ and $g : c \rightarrow d$ respectively, so that the vertical 2-cell $kf \rightarrow gh $ given by the pasting 
	\[
	\begin{tikzcd}
		a \arrow[r,equal,""name=a1] \arrow[d,equal] & a\arrow[d, "f"] \\ 
		a \arrow[r, "F"name=a2]\arrow[d,"h"'] & b \arrow[d,"k"] \\
		c \arrow[r, "G"'name=a3]\arrow[d,"g"'] & d \arrow[d,equal] \\
		d \arrow[r,equal, ""name=a4] &d,
		\arrow[from=a1, to=a2, Rightarrow, shorten <= 6pt]
		\arrow[from=a2,to=a3,phantom, "{\scriptstyle \alpha}", shorten <= 6pt]
		\arrow[from=a3,to=a4,Rightarrow]
	\end{tikzcd}
	\]
	constitutes an equivalence in $\Vert(\P)(a,d)$. Here the top and bottom 2-cell are the companionship unit and counit respectively.

	Dually, we say that $\alpha$ is \textit{conjointable} if $F$ and $G$ are the conjoints 
	of arrows $f : b \rightarrow a$ and $g : d \rightarrow c$ respectively, and so that the similarly obtained vertical 2-cell $gk \rightarrow hf$ is invertible in 
	$\Vert(\P)(b,c)$.
\end{definition} 

\begin{example}\label{ex.conjointable-sq}
	Let us consider the locally complete double Segal space $\Sq(X)$ of squares in an locally complete 
	2-fold Segal space $X$ (see \ref{cor.sq-restrictions}).
	Let $\alpha$ be a 2-cell in $\Sq(X)$ that corresponds to a lax commutative square
	\[
		\begin{tikzcd}
			a \arrow[r,"F"]\arrow[d,"h"'] & |[alias=f]|b \arrow[d ,"k"]\\
			|[alias=t]|c \arrow[r, "G"'] & d
			\arrow[from=f,to=t, Rightarrow]
		\end{tikzcd}
	\]
	in $X$. One may verify that $\alpha$ is companionable if and only if $\alpha$ strictly commutes, i.e.\ the 
	2-cell filling the square is invertible. From the above, one can deduce that $\alpha$ is conjointable if and only if $F$ and $G$ admit left adjoints $f$ and $g$
	respectively, so that the associated mate
	$$
	gk \rightarrow gkFf \rightarrow gGhf \rightarrow hf
	$$ 
	is an equivalence in $X(b,c)$. 
	In the terminology of Lurie, $\alpha$ is thus conjointable if and only if it is \textit{left adjointable} \cite[Definition 7.3.1.2]{HTT}. 
\end{example}

There is also another characterization of companionable squares:

\begin{proposition}\label{prop.alternative-char-companionables}
	Suppose that
	\[
		\alpha = \begin{tikzcd}
			a \arrow[r, "F"name=f] \arrow[d,"h"'] & b \arrow[d,"k"] \\
			c \arrow[r, "G"'name=t] & d
			\arrow[from=f,to=t,Rightarrow, shorten <= 6pt, shorten >= 6pt]
		\end{tikzcd}
	\] is a 2-cell in a locally complete double Segal space $\P$. Then the following assertions are equivalent: 
	\begin{enumerate}
		\item $\alpha$ is companionable,
		\item $F$ and $G$ are the companions of vertical arrows $f$ and $g$ respectively, so that the there is an equivalence
	\[
		\begin{tikzcd}
			a \arrow[r,equal, ""name=a1]\arrow[d,equal] & a \arrow[d, "f"] \\
			a \arrow[r, "F"name=a2] \arrow[d,"h"'] & b \arrow[d,"k"] \\
			c \arrow[r, "G"'name=a3] & d
			\arrow[from=a1,to=a2,Rightarrow, shorten <= 6pt]
			\arrow[from=a2,to=a3,phantom, "{\scriptstyle \alpha}"]
		\end{tikzcd}
		\simeq 
		\begin{tikzcd}
			a \arrow[d, "h"'name=f]\arrow[r,equal] & a \arrow[d, "h"name=t] \\
			c \arrow[r,equal, ""name=f2] \arrow[d,equal] & c \arrow[d,"g"] \\
			c \arrow[r, "G"'name=t2] & d
			\arrow[from=f,to=t,equal, shorten <= 10pt, shorten >= 10pt]
			\arrow[from=f2,to=t2,Rightarrow, shorten <= 6pt, shorten >= 6pt]
		\end{tikzcd}
	\]
	in the space of 2-cells $\P_{1,1}$. Here the top left and bottom right 2-cells are companionship units.
	\end{enumerate}
\end{proposition}
\begin{proof}
	It is clear that (2) implies (1) in light of the companionship triangle identities. Conversely, suppose that (1) holds. 
	Then we can expand the pasting that appears on the left in (2) as
	\[
		\begin{tikzcd}
			a \arrow[r,equal, ""name=a1]\arrow[d,equal] & a \arrow[d, "f"] \\
			a \arrow[r, "F"name=a2] \arrow[d,"h"'] & b \arrow[d,"k"] \\
			c \arrow[r, "G"'name=a3] & d
			\arrow[from=a1,to=a2,Rightarrow, shorten <= 6pt]
			\arrow[from=a2,to=a3,phantom, "{\scriptstyle \alpha}"]
		\end{tikzcd}
		\simeq 
		\begin{tikzcd}
			a\arrow[r,equal]\arrow[d,equal, ""'name=f1] & a \arrow[r,equal,""name=a1]\arrow[d,equal,""name=t1] & a \arrow[d, "f"] \\
			a\arrow[d, "h"'name=f2] \arrow[r,equal] & a \arrow[r, "F"name=a2] \arrow[d,"h"name=t2] & b \arrow[d,"k"] \\
			c \arrow[r,equal,""name=a4]\arrow[d,equal] & c \arrow[r, "G"'name=a3]\arrow[d,"g"'] & d\arrow[d,equal] \\
			c \arrow[r, "G"'name=a5] & d \arrow[r,equal, ""name=a6] & d.
			\arrow[from=f1,to=t1,equal, shorten <= 10pt, shorten >= 10pt]
			\arrow[from=f2,to=t2,equal, shorten <= 10pt, shorten >= 10pt]
			\arrow[from=a1,to=a2,Rightarrow, shorten <= 6pt]
			\arrow[from=a2,to=a3,phantom, "{\scriptstyle \alpha}"]
			\arrow[from=a4,to=a5,Rightarrow, shorten <= 6pt, shorten >= 6pt]
			\arrow[from=a3,to=a6,Rightarrow]
		\end{tikzcd}	
	\]
	Here the bottom right 2-cell is the companionship counit for the pair $(g,G)$. Since $\alpha$ is companionable, 
	the right column of the right hand-side pastes to an invertible 2-cell in $\Vert(\P)$. By assumption, 
	$\Vert(\P)$ is locally complete, hence this column must be equivalent to the identity 2-cell on the vertical arrow $gh$. Thus we deduce that 
	\[
		\begin{tikzcd}
			a\arrow[r,equal]\arrow[d,equal, ""'name=f1] & a \arrow[r,equal, ""name=a1]\arrow[d,equal,""name=t1] & a \arrow[d, "f"] \\
			a\arrow[d, "h"'name=f2] \arrow[r,equal] & a \arrow[r, "F"name=a3] \arrow[d,"h"name=t2] & b \arrow[d,"k"] \\
			c \arrow[r,equal, ""name=a4]\arrow[d,equal] & c \arrow[r, "G"'name=a3]\arrow[d,"g"'] & d\arrow[d,equal] \\
			c \arrow[r, "G"'name=a5] & d \arrow[r,equal,""name=a6] & d
			\arrow[from=f1,to=t1,equal, shorten <= 10pt, shorten >= 10pt]
			\arrow[from=f2,to=t2,equal, shorten <= 10pt, shorten >= 10pt]
			\arrow[from=a1,to=a2,Rightarrow, shorten <= 6pt]
			\arrow[from=a2,to=a3,phantom, "{\scriptstyle \alpha}"]
			\arrow[from=a4,to=a5,Rightarrow, shorten <= 6pt, shorten >= 6pt]
			\arrow[from=a3,to=a6,Rightarrow]
		\end{tikzcd}	
		\simeq 	
		\begin{tikzcd}
			a \arrow[d, "h"'name=f]\arrow[r,equal] & a \arrow[d, "h"name=t] \\
			c \arrow[r,equal, ""name=f2] \arrow[d,equal] & c \arrow[d,"g"] \\
			c \arrow[r, "G"'name=t2] & d
			\arrow[from=f,to=t,equal, shorten <= 10pt, shorten >= 10pt]
			\arrow[from=f2,to=t2,Rightarrow, shorten <= 6pt, shorten >= 6pt]
		\end{tikzcd}
	\]
	as desired.
\end{proof}

Using the language of companionable 2-cells, we can phrase the main theorem of this section:

\begin{theorem}\label{thm.companions-fun-dblcat}
	Let $X$ and $\P$ be a bisimplicial space and locally complete double Segal space respectively.
	Suppose that  $\alpha : h \rightarrow k$ is a horizontal natural transformation between functors $h,k : X \rightarrow \P$. 
	Then the following assertions are equivalent:
	\begin{enumerate}
		\item $\alpha$ is a companion in $\FFUN(X, \P)$,
		\item for every vertical arrow $f : x \rightarrow y$ in $X$, the associated naturality 2-cell 
		\[ 
		\alpha|(f\times\id_{[1]_h}) = \begin{tikzcd}
			h(x) \arrow[r, "\alpha_x"name=f]\arrow[d, "h(f)"'] & k(x) \arrow[d, "k(f)"] \\
			h(y) \arrow[r, "\alpha_y"'name=t]& k(y)
			\arrow[from=f,to=t,Rightarrow, shorten <= 6pt, shorten >= 6pt]
		\end{tikzcd}
		\]
		is companionable in $\P$.
	\end{enumerate}
	If  these equivalent conditions are met, then the vertical natural 
	transformation $\beta : h \rightarrow k$ in $\FFUN(X,\P)$ whose companion is given by $\alpha$, admits the following description.
	For every vertical arrow $f : x \rightarrow y$ in $X$, the commutative naturality square
	\[
		\begin{tikzcd}
			h(x) \arrow[r, "\beta_x"]\arrow[d, "h(f)"'] & k(x) \arrow[d, "k(f)"] \\
			h(y) \arrow[r, "\beta_y"]& k(y)
		\end{tikzcd}
	\]
	in $\Vert(\P)$ corresponds to the pasting in \ref{def.companionable} associated to the 2-cell in (2).
\end{theorem}

\begin{remark}
	There is a dual statement for conjoints which can be formally obtained from the above.
\end{remark}

\begin{corollary}
	Let $\P$ be a locally complete double Segal space. Suppose that $X$ 
	is a 2-fold Segal space and let $\alpha : h \rightarrow k$ be a horizontal natural transformation in the 
	horizontal cotensor product $\{X, \P\}$. 
	Then the following assertions are equivalent:
	\begin{enumerate}
		\item $\alpha$ is a companion,
		\item for each $x\in X$, the component horizontal arrow 
		$$
		\alpha_x : h(x) \rightarrow k(x)
		$$
		is a companion in $\P$.
	\end{enumerate}
\end{corollary}
\begin{proof}
	It is clear that (1) implies (2). Conversely, suppose that (2) holds. In light of \ref{thm.companions-fun-dblcat},
	we have to show that the naturality 2-cell of $\alpha$ associated with any vertical arrow $f$ of $X_h$, is companionable. These 
	naturality 2-cells must be horizontal identity 2-cells, since every $f : [1]_v \rightarrow X_h$ factors through the degeneracy map $[1]_v \rightarrow [0]_v$. The companionability 
	requirement thus translates to condition (2).
\end{proof}

\begin{corollary}\label{cor.closure-equipments}
	The following classes of locally complete double Segal spaces are closed 
	under vertical cotensor products:
	\begin{enumerate}
		\item the ones with all companions of vertical arrows,
		\item the ones with all conjoints of vertical arrows.
	\end{enumerate}
\end{corollary}

\begin{remark}
	We apply \ref{thm.companions-fun-dblcat} in \cite{EquipII} to obtain a characterization 
	of companions and conjoints in the double $\infty$-categories of $\infty$-categories internal to $\infty$-toposes.
\end{remark}

\subsection{Proof of the main result} 
We will proceed \textit{corepresentably}
in our demonstration of \ref{thm.companions-fun-dblcat}. We will make use of
\ref{thm.htpy-coh-comp-2} which 
allows us reduce the proof to the cases that $X = [1]_h$ and $X = [1]_v$ via a descent argument. These instances can then be tackled using explicit combinatorics. 

The first substantial input is the following characterization of companionship units in functor double Segal space:

\begin{lemma}\label{lem.companionship-units-fundbl}
	Let $X$ and $\P$ be a bisimplicial space and double Segal space respectively. Then a map $\eta : L \rightarrow \FFUN(X,\P)$ classifies a companionship unit if and only if 
	each restriction
	$$
	\eta_x : L \rightarrow \FFUN(X, \P) \xrightarrow{\{x\}^*} \P 
	$$
	classifies a companionship unit in $\P$ for every $x \in X$.
\end{lemma}
\begin{proof}
	We must show that ${\eta}$ lies in the image of the restriction monomorphism 
	$$
	\map_{\Cat^2(\S)}(\comp, \FFUN(X, \P)) \rightarrow \map_{\Cat^2(\S)}(L, \FFUN(X, \P))
	$$
	on account of \ref{thm.htpy-coh-comp-1}.
	We can write $X$ as a colimit of grids,
	$$
	X \simeq \colim_{[n,m] \rightarrow X \in \Delta^{\times 2}/X} [n,m],
	$$
	so that we may reduce to the case that $X = [n,m]$. In turn, such a grid may canonically be decomposed as a colimit 
	of free 0-cells, 1-cells and 2-cells. This allows us to reduce to the case that $n,m \in \{0,1\}$. Using the 
	fact that $[1,1]$ decomposes as the product $$[1,1] = [1,0] \times [0,1] = [1]_h \times [1]_v,$$ one can verify that the instance $X=[1,1]$
	follows from the cases $X=[1]_h$ and $X=[1]_v$. We 
	may further reduce to the single case $X=[1]_v$ by using the transpose duality.

	We will write $V$ and $S$ for the smallest bisimplicial subsets of $\comp$ containing the bisimplices
	$$
	\begin{tikzcd}[column sep = tiny, row sep = tiny]
		0 \arrow[r,equal]\arrow[d,equal ] & 0 \arrow[d] \\
		0 \arrow[r]\arrow[d] & 1\arrow[d,equal]  \\
		1 \arrow[r,equal] & 1,
	\end{tikzcd}\quad \text{and}\quad 
	\begin{tikzcd}[column sep = tiny, row sep = tiny]
		0 \arrow[r,equal]\arrow[d,equal]  & 0 \arrow[d,equal] \arrow[r,equal] & 0\arrow[d] \\
		0 \arrow[r,equal]\arrow[d,equal ] & 0 \arrow[d] \arrow[r] & 1\arrow[d,equal] \\
		0 \arrow[r] & 1 \arrow[r,equal] & 1
	\end{tikzcd}
	$$
	respectively. An extension of $\eta$ to a map $S \rightarrow \FFUN(X,\P)$ would witness that $\eta$ 
	is a companionship unit, since $S$ encodes a companionship counit and triangle identities (with extra coherences). 
	We will equivalently demonstrate that the adjunct map  
	$\eta^\sharp : [1]_v \times L \rightarrow \P$ extends to a map $[1]_v \times S \rightarrow \P$. In fact, we will end up showing that this extension 
	is essentially unique. 
	Suppose that $[1]_v \times L \subset A \subset B \subset [1]_v \times \comp$ are bisimplicial subsets. For the sake of this argument, 
	we will say that the inclusion $A \rightarrow B$ is \textit{good} if it induces an equivalence on fibers 
	$$
	\map_{\PSh(\Delta^{\times 2})}(B, \P)_\eta \rightarrow \map_{\PSh(\Delta^{\times 2})}(A, \P)_\eta.
	$$
	We will show that all inclusions 
	$$
	[1]_v\times L \rightarrow [1]_v \times V \rightarrow [1]_v \times S
	$$
	are good, and this will then prove the lemma.

	Let us start by handling the first inclusion. We will consider the inclusion $$[1]_v \times L \rightarrow [1]_v \times L \cup_{\{0,1\}_v \times L} \{0,1\}_v \times V =: A.$$ 
	The inclusion $L \rightarrow V$ is a pushout along $\Lambda^0_L[1,2] \rightarrow L[1,2]$ (see \ref{not.special-lifting-notation}). 
	Hence, it follows from \ref{thm.lifting-thm} and the fact that $\eta_0$ and $\eta_1$ are companionship units that this inclusion is good.
	By 2-out-of-3, it is now
	necessary and sufficient to show that the inclusion $A \rightarrow [1]_v \times V$
	is good. 
	Note that $A$ is precisely 
	the smallest bisimplicial subset of $[1]_v \times \comp$ containing the (1,2)-bisimplices
	\[
		\begin{tikzcd}[column sep = tiny, row sep = tiny]
			(0,0) \arrow[r,equal]\arrow[d,equal] & (0,0) \arrow[d]\\
			(0,0) \arrow[r]\arrow[d] & (0,1) \arrow[d] \\
			(1,0) \arrow[r] & (1,1)
		\end{tikzcd}
		\begin{tikzcd}[column sep = tiny, row sep = tiny]
		(0,0) \arrow[r,equal]\arrow[d] & (0,0) \arrow[d]\\
		(1,0) \arrow[r,equal]\arrow[d,equal] & (1,0) \arrow[d] \\
		(1,0) \arrow[r] & (1,1),
	\end{tikzcd}
	\begin{tikzcd}[column sep = tiny, row sep = tiny]
		(0,0) \arrow[r,equal]\arrow[d,equal] & (0,0) \arrow[d]\\
		(0,0) \arrow[r]\arrow[d] & (0,1) \arrow[d,equal] \\
		(0,1) \arrow[r,equal] & (0,1)
	\end{tikzcd}
	\begin{tikzcd}[column sep = tiny, row sep = tiny]
		(1,0) \arrow[r,equal]\arrow[d,equal] & (1,0) \arrow[d]\\
		(1,0) \arrow[r]\arrow[d] & (1,1) \arrow[d,equal] \\
		(1,1) \arrow[r,equal] & (1,1).
	\end{tikzcd}
	\]
	Let us define $B \subset [1]_v \times V$ to be the smallest bisimplicial subset containing $A$ and the 
	(1,2)-bisimplex 
	\[
		\begin{tikzcd}[column sep = tiny, row sep = tiny]
			(0,0) \arrow[r,equal]\arrow[d,equal] & (0,0) \arrow[d]\\
			(0,0) \arrow[r]\arrow[d] & (0,1) \arrow[d] \\
			(1,1) \arrow[r] & (1,1).
		\end{tikzcd}
	\]
	It is readily verified that the inclusion $A \rightarrow B$ is a pushout along $\Gamma^0_L[1,2] \rightarrow L[1,2]$. Using \ref{thm.lifting-thm} once again, 
	we deduce that 
	it suffices to check that $B \rightarrow [1]_v \times V$ is good.
	Now we use that $[1]_v \times V$ is generated by the non-degenerate (1,3)-bisimplices
	\[
		\sigma_1 := \begin{tikzcd}[column sep = tiny, row sep = tiny]
			(0,0) \arrow[r,equal]\arrow[d] & (0,0) \arrow[d] \\
			(1,0)\arrow[d,equal] \arrow[r,equal] & (1,0)\arrow[d] \\
			(1,0) \arrow[d]\arrow[r] & (1,1)\arrow[d,equal] \\
			(1,1) \arrow[r,equal] & (1,1). 
		\end{tikzcd}\quad 
		\sigma_2 := \begin{tikzcd}[column sep = tiny, row sep = tiny]
			(0,0) \arrow[r,equal]\arrow[d,equal] & (0,0) \arrow[d] \\
			(0,0)\arrow[d] \arrow[r] & (0,1)\arrow[d,equal]\\
			(0,1) \arrow[d]\arrow[r,equal] & (0,1)\arrow[d]\\
			(1,1) \arrow[r,equal] & (1,1), 
		\end{tikzcd}\quad 
		\sigma_3 := \begin{tikzcd}[column sep = tiny, row sep = tiny]
			(0,0) \arrow[r,equal]\arrow[d,equal] & (0,0) \arrow[d] \\
			(0,0)\arrow[d] \arrow[r] & (0,1)\arrow[d] \\ 
			(1,0) \arrow[d]\arrow[r] & (1,1)\arrow[d,equal] \\
			(1,1) \arrow[r,equal] & (1,1), 
		\end{tikzcd}
	\]
	Let $C_i$ denote the smallest bisimplicial subset of $[1]_v \times V$ containing $B$ and $\sigma_1, \dotsc, \sigma_i$.
	This yields a finite filtration 
	$$
	B = C_0 \rightarrow C_1 \rightarrow C_2 \rightarrow C_3 = [1]_v \times V.
	$$
	On account of \ref{lem.gen-horns} and \ref{thm.lifting-thm}, each inclusion 
	is good: 
	the first inclusion is a pushout along 
	$\Lambda^{\emptyset, 1}[1,3] \rightarrow [1,3]$, and the second and third maps are pushouts along  $\Gamma^0_L[1,3] \rightarrow L[1,3]$.
	
	Finally, we show that the map 
	$$
	[1]_v \times V \rightarrow [1]_v \times S 
	$$
	is good. This demonstration is similar. We will first consider the inclusion $$[1]_v \times V \rightarrow  [1]_v \times V \cup_{\{0,1\}_v \times V} \{0,1\}_v \times S =: D.$$
	The inclusion $V \rightarrow S$ is a pushout along the map $\Gamma^0_L[2,2] \rightarrow L[2,2]$, so that we can use 
	\ref{thm.lifting-thm} to conclude that the above inclusion is good.
	Hence, it suffices to show that the inclusion $D \rightarrow [1]_v \times S$ is good.
	Let $E$ be the smallest bisimplicial subset of $[1]_v \times \comp$ containing $D$ and the (3,3)-bisimplex 
	\[
		\begin{tikzcd}[column sep = tiny, row sep = tiny]
			(0,0) \arrow[r,equal]\arrow[d,equal] & (0,0) \arrow[r,equal]\arrow[d,equal] & (0,0) \arrow[d] \\
			(0,0) \arrow[r,equal]\arrow[d] & (0,0) \arrow[d] \arrow[r] & (0,1) \arrow[d] \\
			(1,0) \arrow[r] & (1,1) \arrow[r,equal] & (1,1).
		\end{tikzcd}
	\]
	Then $D \rightarrow E$ is a pushout along $\Gamma^0_L[3,3] \rightarrow L[3,3]$, so it suffices 
	to check that $E \rightarrow [1]_v \times S$ is good by \ref{thm.lifting-thm}. Again, we filter this as follows. Consider 
	the generating (2,3)-bisimplices of $[1]_v \times S$:
	\[
		\tau_1 := \begin{tikzcd}[column sep = tiny, row sep = tiny]
		(0,0) \arrow[r,equal]\arrow[d] &	(0,0) \arrow[r,equal]\arrow[d] & (0,0) \arrow[d] \\
		(1,0)\arrow[d,equal] \arrow[r,equal]&	(1,0)\arrow[d,equal] \arrow[r,equal] & (1,0)\arrow[d] \\
		(1,0) \arrow[r,equal] \arrow[d,equal] &	(1,0) \arrow[d]\arrow[r] & (1,1)\arrow[d,equal] \\
		(1,0) \arrow[r] &	(1,1) \arrow[r,equal] & (1,1). 
		\end{tikzcd}\quad 
		\tau_2 := \begin{tikzcd}[column sep = tiny, row sep = tiny]
			(0,0) \arrow[r,equal]\arrow[d,equal] & (0,0) \arrow[r,equal]\arrow[d,equal] & (0,0) \arrow[d] \\
			(0,0) \arrow[r,equal] \arrow[d,equal] & (0,0)\arrow[d] \arrow[r] & (0,1)\arrow[d,equal]\\
			(0,0) \arrow[r]\arrow[d] & (0,1) \arrow[d]\arrow[r,equal] & (0,1)\arrow[d]\\
			(1,0) \arrow[r] & (1,1) \arrow[r,equal] & (1,1), 
		\end{tikzcd}
	\] 
	\[
		\tau_3 := \begin{tikzcd}[column sep = tiny, row sep = tiny]
			(0,0) \arrow[r,equal]\arrow[d,equal] & (0,0) \arrow[r,equal]\arrow[d,equal] & (0,0) \arrow[d] \\
			(0,0) \arrow[r,equal]\arrow[d] & (0,0)\arrow[d] \arrow[r] & (0,1)\arrow[d] \\ 
			(1,0) \arrow[r]\arrow[d,equal] & (1,0) \arrow[d]\arrow[r] & (1,1)\arrow[d,equal] \\
			(1,0) \arrow[r] & (1,1) \arrow[r,equal] & (1,1), 
		\end{tikzcd}
	\]
	Let $F_i$ denote the smallest bisimplicial subset of $[1]_v \times S$ containing $E$ and $\tau_1, \dotsc, \tau_i$.
	Then we obtain a filtration 
	$$
	E = F_0 \rightarrow F_1 \rightarrow F_2 \rightarrow F_3 = [1]_v \times S,
	$$
	so that each map is good: the first map is a pushout along 
	$\Lambda^{\emptyset, 1}[2,3] \rightarrow [2,3]$ and the second and third maps are pushouts along  $\Gamma^0_L[2,3] \rightarrow L[2,3]$.
\end{proof}

\begin{construction}
We
consider the double Segal space 
$$
Q := [1]_v \times [1]_h \cup_{\{0,1\}_v \times [1]_h} \{0,1\}_v \times L,
$$
induced by the map $[1]_h \rightarrow L$ that selects the non-degenerate horizontal arrow. This pushout is computed in $\Cat^2(\S)$. It comes with a canonical functor
$$
q : Q \rightarrow [1]_v \times L.
$$ 
Now there are two 2-cells 
$$
c_1, c_2 : [1,1] \rightarrow Q
$$
that select the pastings 
\[
\begin{tikzcd}[column sep = tiny, row sep = tiny]
	(0,0) \arrow[r,equal, ""name=a1]\arrow[d,equal] & (0,0) \arrow[d] \\
	(0,0)\arrow[d] \arrow[r, ""name=a2] & (0,1)\arrow[d] \\ 
	(1,0) \arrow[r, ""name=a3] & (1,1)
	\arrow[from=a1,to=a2,Rightarrow, shorten <= 6pt]
	\arrow[from=a2,to=a3,Rightarrow, shorten <= 6pt]
	\end{tikzcd}
	\quad 
	\text{and} \quad 
	\begin{tikzcd}[column sep = tiny, row sep = tiny]
		(0,0) \arrow[r,equal]\arrow[d, ""'name=f] & (0,0) \arrow[d, ""name=t] \\
		(1,0)\arrow[d,equal] \arrow[r,equal, ""name=f2] & (1,0)\arrow[d] \\ 
		(1,0) \arrow[r, ""name=t2] & (1,1),
		\arrow[from=f,to=t,equal, shorten <= 16pt, shorten >= 16pt]
		\arrow[from=f2,to=t2,Rightarrow, shorten <= 6pt]
	\end{tikzcd}
\]
respectively.
\end{construction}

One may readily verify that the constructed pastings $c_1$ and $c_2$ are co-equalized by $q$ in $[1]_v \times L$, i.e.\ there is a commutative square 
\[
		\begin{tikzcd}
			{[1,1]} \cup_{\Lambda^{1, \emptyset}[1,1]} {[1,1]} \arrow[r,"{(c_1,c_2)}"]\arrow[d] & Q \arrow[d,"q"] \\
			{[1,1]} \arrow[r] & {[1]}_v \times L
		\end{tikzcd}
\]
of double Segal spaces.
In fact, a crucial observation is that $[1]_v \times L$ is universal with regard to this property:

\begin{lemma}\label{lem.companions-cotensor1}
	The above square is a pushout square of double Segal spaces. 
\end{lemma}
\begin{proof}
	Let us consider the following diagram of bisimplicial spaces 
	\[
		\begin{tikzcd}
			{[1,1]^{\sqcup 2}} \arrow[rr, "{[1, d_1]^{\sqcup 2}}"]\arrow[d] && {[1,2]}^{\sqcup 2} \arrow[r, "{(c_1',c_2')}"]\arrow[d] & Q\arrow[d, equal]\\
			{[1,1]} \cup_{\Lambda^{1, \emptyset}[1,1]} [1,1]\arrow[d] \arrow[rr] && {[1,2]} \cup_{\Lambda^{1, \emptyset}[1,1]} [1,2]\arrow[d] \arrow[r] & Q\arrow[d,"q"]\\
			{[1,1]} \arrow[rr] && {[1]_v \times [1,1]} \arrow[r] & {[1]}_v \times L,
		\end{tikzcd}
	\]
	where all pushouts are computed in $\PSh(\Delta^{\times 2})$.
	Here $c_1'$ and $c_2'$ are the evident extensions of $c_1$ and $c_2$.
	The total square of the left column is a pushout square in $\PSh(\Delta^{\times 2})$, which comes 
	from the decomposition 
	$$
	[1]_v \times [1,1] \simeq [1]_h \times [1]_v \times [1]_v \simeq [1]_h \times ([2]_v \cup_{[1]_v} [2]_v).
	$$
	The top left square is a pushout square as well. Hence, the bottom left square is a pushout square. 
	Consequently, the lemma holds if the right bottom square is carried to a pushout square by the reflector 
	$\PSh(\Delta^{\times 2}) \rightarrow \Cat^2(\S)$.

	Consider the \textit{bisimplicial set} $$
	Q' := [1]_v \times [1]_h \cup_{\{0,1\}_v \times [1]_h} \{0,1\}_v \times L \subset [1]_v \times L,
	$$
	where the pushout is now computed in $\PSh(\Delta^{\times 2})$.  
	It comes with a canonical map $Q' \rightarrow Q$ that is contained in the saturation of (\textit{Seg}). The right bottom square of the above diagram 
	fits in a further commutative diagram 
	\[
		\begin{tikzcd}
			\Lambda^{1,1}[1,2] \cup_{\Lambda^{1, \emptyset}[1,1]} \Lambda^{1,1}[1,2] \arrow[r]\arrow[d] & {[1,2]} \cup_{\Lambda^{1, \emptyset}[1,1]} [1,2] \arrow[r]\arrow[d] & {[1]_v} \times [1,1] \arrow[d] \\
			Q' \arrow[r] & Q \arrow[r] & {[1]}_v \times L.
		\end{tikzcd}
	\]
	One readily verifies that the outer square is a pushout square in $\PSh(\Delta^{\times 2})$. 
	The desired result now follows from the abservation that the top left horizontal arrow is contained in the saturation of 
	(\textit{Seg}) on account of \ref{lem.gen-horns}.
\end{proof}

We will now treat the final ingredient for the proof of \ref{thm.companions-fun-dblcat}.

\begin{construction}
	We will make use  of the bisimplicial subset 
	$$
	\textstyle P:=  {[1]_h} \times [1]_h \bigcup_{{\{0,1\}_h \times [1]_h}} \{0,1\}_h \times L
	$$
	of $[1]_h \times L$. Here, the pushout is computed in $\PSh(\Delta^{\times 2})$.
\end{construction}

\begin{lemma}\label{lem.companions-cotensor2}
	Suppose that $\eta_0$ and $\eta_1$ are companionship units in a double Segal space $\P$. 
	Then the inclusion $P \rightarrow [1]_h \times L$ induces an equivalence on fibers 
	\[
	\begin{tikzcd}
	\map_{\PSh(\Delta^{\times 2})}([1]_h \times L, \P) \times_{\map(\{0,1\}_h \times L, \P)} \{(\eta_0, \eta_1)\} \arrow[d] \\ 
	 \map_{\PSh(\Delta^{\times 2})}(P, \P) \times_{\map(\{0,1\}_h \times L, \P)} \{(\eta_0, \eta_1)\}.
	\end{tikzcd}
	\]
\end{lemma}
\begin{proof}
	Note that $[1]_h \times L$ is generated 
	by the bisimplices 
	\[
		\begin{tikzcd}[column sep = tiny, row sep = tiny]
			(0,0) \arrow[r] \arrow[d,equal] & (1,0) \arrow[r, equal]\arrow[d,equal] & (1,0) \arrow[d] \\ 
			(0,0) \arrow[r] & (1,0) \arrow[r] & (1,1),
		\end{tikzcd} \quad \begin{tikzcd}[column sep = tiny, row sep = tiny]
			(0,0) \arrow[r,equal] \arrow[d,equal] & (0,0) \arrow[r]\arrow[d] & (1,0) \arrow[d] \\ 
			(0,0) \arrow[r] & (0,1) \arrow[r] & (1,1).
		\end{tikzcd}
	\]
	We can attach the first bisimplex to $P$ by pushing out along $\Lambda^{1, \emptyset}[2,1] \rightarrow [2,1]$. Then we may attach 
	the second bisimplex to the resulting bisimplicial space via a pushout along $\Gamma^0_L[1,2]^\tp \rightarrow L[1,2]^\tp$. 
	So, the desired result follows from \ref{lem.gen-horns} and \ref{thm.lifting-thm}.
\end{proof}

\begin{proof}[Proof of \ref{thm.companions-fun-dblcat}]
	Suppose that (1) holds so that we obtain an extension 
	$$
	\bar{\alpha} :  X \times \comp \rightarrow \P 
	$$
	of $\alpha$. For any vertical arrow $f$ in $X$, we can consider the two restrictions 
	$$
	[1,1] \overset{c_1}{\underset{c_1}{\rightrightarrows}} Q \xrightarrow{q} [1]_v \times L \xrightarrow{f\times \mathrm{incl}} X \times \comp \rightarrow \P. 
	$$
	We already observed at the start of this subsection that 
	 $c_1$ and $c_2$ are co-equalized by $q$, hence the two restrictions are equivalent. 
	Thus the naturality 2-cell $\alpha$ for $f$ is companionable on account of 
	characterization (2) of \ref{prop.alternative-char-companionables}.

	Conversely, suppose that (2) holds. Then we must show that ${\alpha}$ lies in the image of the restriction monomorphism 
	$$
	\map_{\Cat^2(\S)}(\comp, \FFUN(X, \P)) \rightarrow \map_{\Cat^2(\S)}([1]_h, \FFUN(X, \P))
	$$
	on account of \ref{thm.htpy-coh-comp-2}. We may apply the same reasoning as in the proof of \ref{lem.companionship-units-fundbl} 
	to reduce to the cases that $X=[1]_h$ and $X=[1]_v$. 
	On account of \ref{thm.htpy-coh-comp-1} and \ref{lem.companionship-units-fundbl}, we must show that in these cases 
	a dotted extension in the diagram
	\[
		\begin{tikzcd}
			X \times [1]_h  \arrow[r, "\alpha"]\arrow[d] & \P \\
			X \times L \arrow[ur, dotted, "\bar{\alpha}"']
		\end{tikzcd}
	\]
	exists so that the restriction of $\bar{\alpha}$ to $\{x\} \times L$ classifies a companionship unit for each $x \in X$.

	Consider the case that $X = [1]_v$. Then $\alpha$ is the datum of a companionable 2-cell in $\P$. The associated companionship units 
	yield an extension $\alpha' : Q \rightarrow \P$ of $\alpha$. In light of characterization (2) 
	of \ref{prop.alternative-char-companionables}, 
	the companionability of $\alpha$ bears witnesses to the co-equalization of $c_1$ and $c_2$ by $\alpha'$, so that \ref{lem.companions-cotensor1} implies that the desired extension $\bar{\alpha}$ exists.  The case that 
	$X = [1]_h$ is covered by \ref{lem.companions-cotensor2}.
\end{proof}

\section{Application: double \texorpdfstring{$\infty$}{∞}-categories of \texorpdfstring{$(\infty,2)$}{(∞,2)}-functors}\label{section.app-lax-nt}

The goal of this section is to give an application of \ref{thm.companions-fun-dblcat} to $(\infty,2)$-category theory. 
Given two $(\infty,2)$-categories $\X$ and $\Y$, we will first show that one can construct a double $\infty$-category 
$$
\FFUN^\lax(\X,\Y) 
$$
with:
\begin{itemize}
	\item the objects given by functors $\X \rightarrow \Y$,
	\item the vertical arrows given by natural transformations between functors $\X \rightarrow \Y$,
	\item the horizontal arrows given by \textit{lax} natural transformations between functors $\X \rightarrow \Y$,
\end{itemize}
as a particular vertical cotensor product. We will then show that one can use \ref{thm.companions-fun-dblcat} 
to recover Haugseng's 
recognition theorem for adjunctions in $(\infty,2)$-categories 
of functors and lax natural transformations \cite[Theorem 4.6]{HaugsengLax}, for a particular 
choice of model for the Gray tensor product (cf.\ \ref{rem.gray-comparisons}).

\subsection{The Gray tensor product}\label{ssection.gray}
There is a wide array of different approaches to defining the Gray tensor product for $(\infty,2)$-categories. A nice 
summary of the different models can be found in the introduction of \cite{CampionMaehara}.  Since we have developed enough theory, we can give a brief and reasonably self-contained treatment of a model of this tensor product using double $\infty$-categorical technology 
that was proposed by Gaitsgory and Rozenblyum \cite[Section 10.4.5]{GR}. The following is justified by 
\ref{cor.sq-restrictions}:

\begin{definition}
	The  \textit{Gray tensor product} of $(\infty,2)$-categories $\X$ and $\Y$ is defined 
	to be the (essentially) unique $(\infty,2)$-category $\X \otimes \Y$ with the property that
	$$
	\map_{\Cat_{(\infty,2)}}(\X\otimes \Y, \mathscr{Z}) \simeq \map_{\DblCat_\infty^c}(\X_h \times \Y_v, \Sq(\mathscr{Z}))
	$$
	naturally for all $\mathscr{Z} \in \Cat_{(\infty,2)}$. 
\end{definition}

Note that the above definition is functorial in both variables so that we obtain a functor
$$
(-) \otimes (-) : \Cat_{(\infty,2)} \times \Cat_{(\infty, 2)} \rightarrow \Cat_{(\infty,2)}.
$$

\begin{remark}
It directly follows from these definitions that the Gray tensor product preserves colimits in each variable and that 
$[n] \otimes [m] \in \Cat_{(\infty,2)}$ agrees with the strict Gray tensor product of $[n]$ and $[m]$ for all $[n], [m] \in \Delta$.
\end{remark}

\begin{remark}\label{rem.gray-comparisons}
	This definition of the Gray tensor product has not yet been compared to other constructions 
	of Gray tensor products that appear in the literature. 
	
	However, we can say something about its restriction to $(\infty,1)$-categories:
	the restricted bifunctor $(-) \otimes (-) : \Cat_\infty \times \Cat_\infty \rightarrow \Cat_{(\infty,2)}$ is the unique 
	one that preserves 
	colimits in each variable and restricts to the strict Gray tensor product on $\Delta \times \Delta$. In light of \cite[Proposition 5.1.9]{HHLN}, 
	the Gray tensor product of Gagna, Harpaz and Lanari \cite{GagnaHarpazLanari} satisfies these conditions, thus it coincides with $(-) \otimes (-)$ when restricted 
	to $(\infty,1)$-categories.

	To compare with other Gray tensor products in full generality, one can readily compute the values of $[1;1] \otimes [1]$ and $[1;1] \otimes [1;1]$ in terms of certain 
	 pushouts by using the description $[1;1]_h = [0,1]^{\sqcup 2} \cup_{[0,0]^{\sqcup 2}} [1,1]$. This yields fundamental formulas where the only Gray tensor products appearing are reduced to the form $([1] \times [1]) \otimes [1]$ and $([1] \times [1]) \otimes ([1] \times [1])$.
	If these formulas are true for other Gray tensor products, this allows you to reduce to the above case since 
	every $(\infty,2)$-category can be written as a canonical colimit of the free cells $[0]$, $[1]$ and $[1;1]$.
\end{remark}

\begin{definition}
	Let $\X$ and $\Y$ be two $(\infty,2)$-categories. We define the \textit{double $\infty$-category 
	of functors from $\X$ to $\Y$ with lax natural transformations} as the vertical cotensor product
	$$
	\FFUN^\lax(\X, \Y) := [\X, \Sq(\Y)].
	$$
	The horizontal arrows of this double $\infty$-category correspond precisely to functors $[1] \otimes \X \rightarrow \Y$ and these are called \textit{lax natural transformations}.
\end{definition}

\begin{proposition}\label{prop.fun-lax-props}
	Let $\X$ and $\Y$ be $(\infty,2)$-categories, then: 
	\begin{enumerate}
		\item $\FFUN^\lax(\X,\Y)$ is complete and admits all companions,
		\item there is a natural equivalence $\Vert(\FFUN^\lax(\X,\Y)) \simeq \FUN(\X, \Y)$,
		\item a lax natural transformation $\alpha : h \rightarrow k$ is a companion if and only if 
		for every arrow $x \rightarrow y$ in $\X$, the associated lax square 
		\[
		\begin{tikzcd}
			h(x) \arrow[r, "\alpha_x"] \arrow[d]& |[alias=f]| k(x) \arrow[d] \\
			|[alias=t]| h(y) \arrow[r, "\alpha_y"'] & k(y)
			\arrow[from=f,to=t, Rightarrow]
		\end{tikzcd}
	\]
	commutes, i.e.\ the 2-cell filling the square is invertible.
	\end{enumerate}
\end{proposition}
\begin{proof}
	 Note that (1) follows from \ref{prop.exp-ideals-dss} and \ref{cor.closure-equipments}. Assertion (2)
	 follows from \ref{prop.cotensor-vert} and (3) can be deduced from \ref{ex.conjointable-sq} and \ref{thm.companions-fun-dblcat}. 
\end{proof}

\begin{definition}
	We will write 
	$$
	\fun^\lax(\X,\Y) \text{ and }\FUN^\lax(\X, \Y)
	$$
	for the horizontal $\infty$- and $(\infty,2)$-categorical fragments of $\FFUN^\lax(\X,\Y)$ respectively. 
\end{definition}

\subsection{Lax adjunctions between functors}
We can now give a double categorical proof of Haugseng's 
recognition theorem for adjunctions in $(\infty,2)$-categories 
of functors and lax natural transformations \cite[Theorem 4.6]{HaugsengLax}:

\begin{theorem}\label{thm.haugseng}
	Let $v : k \rightarrow h$ be a lax natural transformation between functors $k,h: \X \rightarrow \Y$. Then the 
	following assertions are equivalent:
	\begin{enumerate}
		\item $v$ is a right adjoint in $\FUN^\lax(\X, \Y)$,
		\item  $v$ is a conjoint in $\FFUN^\lax(\X,\Y)$,
		\item for any arrow $x \rightarrow y$ in $\X$, the horizontal morphisms in the lax square
		\[
			\begin{tikzcd}
				k(x) \arrow[r, "v_x"] \arrow[d]& |[alias=f]| h(x) \arrow[d] \\
				|[alias=t]| k(y) \arrow[r, "v_y"'] & h(y)
				\arrow[from=f,to=t, Rightarrow]
			\end{tikzcd}
		\]
		in $\Y$ admit left adjoints and the associated mate is an equivalence (cf. \ref{ex.conjointable-sq}).
		
	\end{enumerate}
	If these equivalent conditions are met, the left adjoint $u$ of $v$ is strict and for any arrow $x \rightarrow y$ in $\X$, the horizontal morphisms in the 
	commutative square
	\[
			\begin{tikzcd}
				h(x) \arrow[r, "u_x"] \arrow[d]& |[alias=f]| k(x) \arrow[d] \\
				|[alias=t]| h(y) \arrow[r, "u_y"] & k(y)
			\end{tikzcd}
		\]
		are given by left adjoints to $v_x$ and $v_y$, and its commutativity is witnessed by the mate of the square in (2).
\end{theorem}

We will need the following lemma:

\begin{lemma}\label{prop.adjunctions-cotensor-hor}
	Let $\P$ be a locally complete double Segal space that admits all companions. 
	Suppose that $X$ is a 2-fold Segal space and $v : k \rightarrow h$ is a horizontal arrow 
	in $[X,\P]$.
	 Then the following assertions are equivalent:
	\begin{enumerate}
		\item $v$ is a conjoint,
		\item $v$ is a right adjoint in $\Hor([X,\P])$ and  $v_x : k(x) \rightarrow h(x)$ is a conjoint in $\P$ for all $x \in X$.
	\end{enumerate}
	If these equivalent conditions are met, then the 
	left adjoint of $v$ is given by the companion of the vertical arrow 
	in $[X,\P]$ whose conjoint is $v$. 
\end{lemma}
\begin{proof}
	Note that $[X,\P]$ again admits all companions on account of \ref{cor.closure-equipments}.
	If (1) holds, then the left adjoint to $v$ exists and is given by the companion 
	of the vertical arrow of $[X, \P]$ whose conjoint is $v$, see \ref{prop.conj-comp-adjunction}. Clearly, 
	each component $v_x$ is also a conjoint in this case.

	Conversely, let us assume that (2) holds. Suppose that $u : h \rightarrow k$ 
	is the left adjoint of $v$ and denote the unit and counit 
	of the adjunction $(u,v)$ by $\eta$ and $\epsilon$ respectively. Moreover, 
	let us write $f_x : h(x) \rightarrow k(x)$ for the vertical arrow in $\P$ whose conjoint is $v_x$ for each $x \in X$. 
	Then it must follow that $u_x$ is the companion of $f_x$ by \ref{prop.conj-comp-adjunction} and unicity of left adjoints. Moreover, 
	$\eta(x)$ and $\epsilon(x)$ can be identified with the horizontal pastings of the companion and conjoint unit, and companion and conjoint counit,
	\[
		\begin{tikzcd}
			h(x) \arrow[r, equal, ""name=f]\arrow[d,equal] & h(x) \arrow[r,equal, ""name=f2]\arrow[d, "f_x"] & h(x) \arrow[d,equal] \\
			h(x) \arrow[r, "u_x"name=t] & k(x) \arrow[r, "v_x"name=t2] & h(x),
			\arrow[from=f,to=t,Rightarrow, shorten <= 6pt]
			\arrow[from=f2,to=t2,Rightarrow, shorten <= 6pt]
		\end{tikzcd}
		\quad \text{ and } \quad
		\begin{tikzcd}
			k(x) \arrow[r, "v_x"name=f]\arrow[d,equal] & h(x)\arrow[d,"f_x"] \arrow[r, "u_x"name=f2] & k(x)\arrow[d,equal] \\ 
			k(x) \arrow[r, equal, ""name=t] & k(x) \arrow[r,equal, ""name=t2] & k(x),
			\arrow[from=f,to=t,Rightarrow, shorten <= 6pt]
			\arrow[from=f2,to=t2,Rightarrow, shorten <= 6pt]
		\end{tikzcd}
	\] 
	respectively. To show that $v$ is a conjoint, \ref{thm.companions-fun-dblcat} implies that 
	it suffices to show 
	that for any arrow $f : x \rightarrow y$ in $X$, the pastings 
	\[
	\begin{tikzcd}
		h(x) \arrow[r,equal, ""name=a1] \arrow[d,"f_x"'] & h(x)\arrow[d, equal] \\ 
		k(x) \arrow[r, "v_x"name=a2]\arrow[d,"k(f)"'] & h(x) \arrow[d,"h(f)"] \\
		k(y) \arrow[r, "v_y"name=a3]\arrow[d,equal] & h(y) \arrow[d,"f_y"] \\
		k(y) \arrow[r,equal, ""name=a4] & k(y),
		\arrow[from=a1,to=a2,Rightarrow, shorten <= 6pt]
		\arrow[from=a2,to=a3,Rightarrow, shorten <= 6pt]
		\arrow[from=a3,to=a4,Rightarrow, shorten <= 6pt]
	\end{tikzcd} \quad \text { and }\quad
	\begin{tikzcd}
		h(x) \arrow[r,equal, ""name=a1] \arrow[d,equal] & h(x)\arrow[d, "f_x"] \\ 
		h(x) \arrow[r, "u_x"name=a2]\arrow[d,"h(f)"'] & k(x) \arrow[d,"k(f)"] \\
		h(y) \arrow[r, "u_y"name=a3]\arrow[d, "f_y"'] & k(y) \arrow[d,equal] \\
		k(y) \arrow[r,equal, ""name=a4] & k(y),
		\arrow[from=a1,to=a2,Rightarrow, shorten <= 6pt]
		\arrow[from=a2,to=a3,Rightarrow, shorten <= 6pt]
		\arrow[from=a3,to=a4,Rightarrow, shorten <= 6pt]
	\end{tikzcd}
	\]
	are inverses to each other in $\Vert(\P)(h(x), k(y))$. 
	This is now readily verified using the coherence and above pointwise decompositions of $\eta$ and $\epsilon$, and the companionship and conjunction identities.
\end{proof}

\begin{proof}[Proof of \ref{thm.haugseng}]
	The equivalence of (1) and (2) follows from \ref{prop.adjunctions-cotensor-hor} and the fact that all horizontal arrows in $\Sq(\Y)$ are companions. 
	In turn, the equivalence of the assertions (2) and (3) follows from \ref{thm.companions-fun-dblcat} and \ref{ex.conjointable-sq}. The final claim 
	in the theorem statement follows from 
	\ref{thm.companions-fun-dblcat} as well.
\end{proof}

\begin{corollary}\label{cor.thm-haugseng}
	Let $u : h \rightarrow k$ be a lax natural transformation between functors $h,k : \X \rightarrow \Y$. Then the following assertions are equivalent:
	\begin{enumerate}
		\item $u$ is a left adjoint in $\FUN^\lax(\X, \Y)$,
		\item $u$ is strict and admits a conjoint in $\FFUN^\lax(\X,\Y)$,
		\item $u$ is strict and each component $u_x : h(x) \rightarrow k(x)$ is a left adjoint in $\Y$.
	\end{enumerate}
	If these equivalent conditions are met, then for any arrow $x \rightarrow y$ in $\X$, the horizontal morphisms in the 
	lax naturality square of the right adjoint $v$ 
	\[
			\begin{tikzcd}
				k(x) \arrow[r, "v_x"] \arrow[d]& |[alias=f]| h(x) \arrow[d] \\
				|[alias=t]| k(y) \arrow[r, "v_y"'] & h(y)
				\arrow[from=f,to=t, Rightarrow]
			\end{tikzcd}
		\]
	are given by left adjoints to $v_x$ and $v_y$, and the filling 2-cell is given by the mate of the corresponding 
	naturality square of $u$. 
\end{corollary}

\begin{remark}
	A first complete proof of \ref{thm.haugseng}, albeit using a different definition of the Gray tensor product, was 
	recently given by Abell\'an, Gagna and Haugseng \cite{AbellanGagnaHaugseng}.
\end{remark}

\begin{remark}
	One readily verifies that  \ref{thm.haugseng} and \ref{cor.thm-haugseng} also go through whenever $\X \in \PSh(\Delta^{\times 2})_{\mathrm{deg}}$, and $\Y$ is a locally complete 2-fold Segal space. In this case, one can still 
	define $\FFUN^\lax(\X,\Y) := [\X, \Sq(\Y)]$, which will now be a locally complete double Segal space.
\end{remark}

\nocite{*}
\bibliographystyle{amsalpha}
\bibliography{Comp}

\providecommand{\bysame}{\leavevmode\hbox to3em{\hrulefill}\thinspace}
\providecommand{\MR}{\relax\ifhmode\unskip\space\fi MR }
\providecommand{\MRhref}[2]{%
  \href{http://www.ams.org/mathscinet-getitem?mr=#1}{#2}
}
\providecommand{\href}[2]{#2}
\begin{thebibliography}{GMSP23}

\bibitem[Abe23]{Abellan}
F.~Abell\'an, \emph{Comparing lax functors of $(\infty,2)$-categories}, arXiv:2311.12746, 2023.

\bibitem[AF20]{AyalaFrancis}
D.~Ayala and J.~Francis, \emph{Fibrations of $\infty$-categories}, Higher Structures \textbf{4} (2020), no.~1, 168--265.

\bibitem[AGH24]{AbellanGagnaHaugseng}
F.~Abell\'an, A.~Gagna, and R.~Haugseng, \emph{Straightening for lax transformations and adjunctions of $(\infty,2)$-categories}, arXiv:2404.03971, 2024.

\bibitem[BR13]{BergnerRezk1}
J.~Bergner and C.~Rezk, \emph{Comparison of models for $(\infty,n)$-categories {I}}, Geometry and Topology \textbf{17} (2013), no.~4, 2163--2202.

\bibitem[BR20]{BergnerRezk2}
\bysame, \emph{Comparison of models for $(\infty,n)$-categories, {II}}, Journal of Topology \textbf{13} (2020), no.~4, 1554--1581.

\bibitem[BSP21]{BarwickSchommerPries}
C.~Barwick and C.~Schommer-Pries, \emph{On the unicity of the homotopy theory of higher categories}, J. Amer. Math. Soc. \textbf{34} (2021), no.~4, 1011--1058.

\bibitem[CM23]{CampionMaehara}
T.~Campion and Y.~Maehara, \emph{A model-independent {G}ray tensor product for $(\infty,2)$-categories}, arXiv:2304.05965, 2023.

\bibitem[DPP10]{DawsonParePronk}
R.~Dawson, R.~Par\'e, and D.~Pronk, \emph{The span construction}, Theory Appl. Categ. \textbf{24} (2010), no.~13, 302--377.

\bibitem[Ehr63]{Ehresmann}
C.~Ehresmann, \emph{Cat\'egories structur\'ees {III}}, Cah. Topol. G\'eom. Diff\'er. Cat\'eg. \textbf{5} (1963), 1--21.

\bibitem[GHL21]{GagnaHarpazLanari}
A.~Gagna, Y.~Harpaz, and E.~Lanari, \emph{Gray tensor products and lax functors of $(\infty,2)$-categories}, Advances in Mathematics \textbf{391} (2021), 107986.

\bibitem[GMSP23]{GuettaMoserSarazolaVerdugo}
L.~Guetta, L.~Moser, M.~Sarazola, and Verdugo P., \emph{Fibrantly-induced model structures}, arXiv:2301.07801, 2023.

\bibitem[GP99]{GrandisPareLimits}
M.~Grandis and R.~Par\'e, \emph{Limits in double categories}, Cah. Topol. G\'eom. Diff\'er. Cat\'eg. \textbf{40} (1999), no.~3, 162--220.

\bibitem[GP04]{GrandisPare}
\bysame, \emph{Adjoint for double categories}, Cah. Topol. G\'eom. Diff\'er. Cat\'eg. \textbf{45} (2004), no.~3, 193--240.

\bibitem[GR17]{GR}
D.~Gaitsgory and N.~Rozenblyum, \emph{A study in derived algebraic geometry. {V}ol. {I}. {C}orrespondences and duality}, Mathematical Surveys and Monographs, vol. 221, American Mathematical Society, Providence, RI, 2017.

\bibitem[Gra74]{Gray}
J.~Gray, \emph{Formal category theory: {A}djointness for 2-categories}, Lecture Notes in Mathematics, vol. 391, Springer-Verlag, 1974.

\bibitem[Hau13]{HaugsengPhD}
R.~Haugseng, \emph{Weakly enriched higher categories}, Ph.D. thesis, Massachusetts Institute of Technology, 2013.

\bibitem[Hau16]{HaugsengEnriched}
\bysame, \emph{Bimodules and natural transformations for enriched $\infty$-categories}, Homology, Homotopy and Applications \textbf{18} (2016), no.~1, 71--98.

\bibitem[Hau18]{HaugsengSpans}
\bysame, \emph{Iterated spans and classical topological field theories}, Mathematische Zeitschrift \textbf{289} (2018), no.~3, 1427--1488.

\bibitem[Hau21]{HaugsengLax}
\bysame, \emph{On lax transformations, adjunctions, and monads in $(\infty,2)$-categories}, Higher Structures \textbf{5} (2021), no.~5, 244--281.

\bibitem[HHLN21]{HHLN}
R.~Haugseng, F.~Hebestreit, S.~Linskens, and J.~Nuiten, \emph{Lax monoidal adjunctions, two-variable fibrations and the calculus of mates}, arXiv:2011.08808, 2021.

\bibitem[Joh02]{Sketches}
P.~T. Johnstone, \emph{Sketches of an elephant: {A} topos theory compendium, {v}ol. 1}, Oxford Logic Guides, Clarendon Press, 2002.

\bibitem[Joy02]{JoyalQCatKan}
A.~Joyal, \emph{Quasi-categories and {K}an complexes}, J. Pure Appl. Algebra \textbf{175} (2002), no.~1-3, 207--222, Special volume celebrating the 70th birthday of Professor Max Kelly.

\bibitem[Joy08]{Joyal}
\bysame, \emph{The theory of quasi-categories and its applications}, Notes for a course given at the CRM, Barcelona, 2008, \url{https://mat.uab.cat/~kock/crm/hocat/advanced-course/Quadern45-2.pdf}.

\bibitem[JT07]{JoyalTierney}
A.~Joyal and M.~Tierney, \emph{Quasi-categories vs {Segal} spaces}, Categories in algebra, geometry and mathematical physics, Contemp. Math., vol. 431, Amer. Math. Soc., 2007, pp.~277--326.

\bibitem[Lac02]{Lack}
S.~Lack, \emph{A {Q}uillen model structure for 2-categories}, K-Theory (2002), no.~2, 171--205.

\bibitem[Lur09a]{HTT}
J.~Lurie, \emph{Higher topos theory}, Annals of Mathematics Studies, vol. 170, Princeton University Press, 2009.

\bibitem[Lur09b]{LurieInfty2}
\bysame, \emph{$(\infty,2)$-{C}ategories and the {G}oodwillie calculus {I}}, arXiv:0905.0462, 2009.

\bibitem[Lur17]{HA}
\bysame, \emph{Higher algebra}, 2017, \url{https://www.math.ias.edu/~lurie/papers/HA.pdf}.

\bibitem[Lur24]{Kerodon}
\bysame, \emph{Kerodon}, \url{https://kerodon.net}, 2024.

\bibitem[Mos20]{Moser}
L.~Moser, \emph{A double $(\infty,1)$-categorical nerve for double categories}, arXiv:2007.01848, 2020.

\bibitem[Rez01]{RezkSeg}
C.~Rezk, \emph{A model for the homotopy theory of homotopy theory}, Trans. Amer. Math. Soc. \textbf{353} (2001), no.~3, 973--1007.

\bibitem[Rez10]{RezkTheta}
\bysame, \emph{A cartesian presentation of weak {$n$}-categories}, Geom. Topol. \textbf{14} (2010), no.~1, 521--571.

\bibitem[Rui23]{EquipI}
J.~Ruit, \emph{Formal category theory in $\infty$-equipments {I}}, arXiv:2308.03583, 2023.

\bibitem[Rui24a]{EquipII}
\bysame, \emph{Formal category theory in $\infty$-equipments {II}}, arXiv:2408.15190, 2024.

\bibitem[Rui24b]{Pasting}
\bysame, \emph{A pasting theorem for iterated {S}egal spaces}, J. Pure Appl. Algebra \textbf{228} (2024), no.~11, 107712.

\bibitem[RV16]{RiehlVerityAdj}
E.~Riehl and D.~Verity, \emph{Homotopy coherent adjunctions and the formal theory of monads}, Advances in Mathematics \textbf{286} (2016), 802--888.

\bibitem[Shu08]{ShulmanFramedBicats}
M.~Shulman, \emph{Framed bicategories and monoidal fibrations}, Theory Appl. Categ. \textbf{20} (2008), no.~18, 650--738.

\bibitem[SS86]{SchanuelStreet}
S.~Schanuel and R.~Street, \emph{The free adjunction}, Cahiers Topologie G\'eom. Diff\'erentielle Cat\'eg. \textbf{27} (1986), no.~1, 81--83.

\bibitem[Vas19]{Vasilakopoulou}
C.~Vasilakopoulou, \emph{Enriched duality in double categories: {$\mathcal{V}$}-categories and {$\mathcal{V}$}-cocategories}, J. Pure Appl. Algebra \textbf{223} (2019), no.~7, 2889--2947.

\end{thebibliography}

\end{document}